\documentclass[11pt,leqno,a4paper]{amsart}

\usepackage{amsmath}
\usepackage{amsbsy}
\usepackage{amsthm}
\usepackage{amssymb}
\usepackage{amscd}
\usepackage{mathabx}
\usepackage{accents}
\usepackage{float}

\usepackage{tikz}
\usetikzlibrary{matrix,arrows,decorations.pathmorphing}

\usepackage{enumerate}
\usepackage{verbatim}

\usepackage{hyperref}
\usepackage{mathabx}

\usepackage[english]{babel}
\usepackage[utf8]{inputenc}
\usepackage[totalheight=22 true cm, totalwidth=14 true cm]{geometry}

\usepackage{mathrsfs}
\usepackage[all]{xy}

\newtheorem{thm}{Theorem}[section]
\newtheorem{cor}[thm]{Corollary}
\newtheorem{lemma}[thm]{Lemma}

\newtheorem{prop}[thm]{Proposition}

\newtheorem{defn}[thm]{Definition}

\theoremstyle{remark}

\theoremstyle{definition}

\newtheorem{rmk}[thm]{Remark}
\newtheorem{exa}[thm]{Example}

\newtheorem{notation}[thm]{Notation}
\numberwithin{equation}{thm}

\def\beq{\begin{equation}}
\def\eeq{\end{equation}}

\def\beqn{\begin{equation*}}
\def\eeqn{\end{equation*}}

\def\ben{\begin{enumerate}}
\def\een{\end{enumerate}}

\def\crash#1{}
\def\A{{\mathbb A}}

\def\C{{\mathbb C}}

\def\E{{\mathbb E}}
\def\F{{\mathbb F}}

\def\L{{\mathbb L}}

\def\N{{\mathbb N}}

\def\Q{{\mathbb Q}}
\def\R{{\mathbb R}}

\def\Z{{\mathbb Z}}

\def\l{\left}
\def\r{\right}

\def\Im{{\rm Im \,}}

\def\Max{{\rm Max \,}}

\def\cf{\emph{cf.}~}
\def\ie{\emph{i.e.}~}

\def\cA{{\mathcal A}}

\def\cC{{\mathcal C}}

\def\cM{{\mathcal M}}

\def\cO{{\mathcal O}}

\def\cQ{{\mathcal Q}}

\def\cZ{{\mathcal Z}}

\def\sE{{\mathscr E}}
\def\sF{{\mathscr F}}
\def\sG{{\mathscr G}}

\def\sP{{\mathscr P}}

\def\bC{{\mathbf C}}

\def\bW{{\mathbf W}}

\def\what{\widehat}

\def\a{\alpha}

\def\la{\lambda}
\def\La{\Lambda}
\def\na{{\rm na}}

\def\Spec{{\rm Spec \,}}
\def\Hom{{\rm Hom \,}}

\def\im{{\rm Im\,}}

\def\ker{{\rm Ker\,}}
\def\coker{{\rm Coker\,}}

\def\ul{\underline}
\def\im{{\rm Im\,}}
\def\coim{{\rm Coim\,}}

\def\eq{{\rm eq\,}}
\def\FF{{\rm \mathbf{FF}\,}}
\def\ol{\overline}

\def\limpro{\mathop{\lim\limits_{\displaystyle\leftarrow}}}

\def\limind{\mathop{\lim\limits_{\displaystyle\rightarrow}}}

\def\hocolim{\mathop{\mathop{\rm hocolim}\limits_{\displaystyle\longrightarrow}}}

\def\lt{\langle}
\def\gt{\rangle}

\def\Frac{{\rm Frac\,}}

\def\rhook{{\hookrightarrow}} 

\def\bSets{\mathbf{Sets}}

\def\Id{{\rm Id\,}}
\def\op{{\rm op\,}}
\def\sep{{\rm sep\,}}

\def\bMod{\mathbf{Mod}}
\def\bInd{\mathbf{Ind}}

\def\bCBorn{\mathbf{CBorn}}

\def\bAb{\mathbf{Ab}}

\def\wotimes{\widehat{\otimes}}
\def\ootimes{{\ol{\otimes}}}
\def\bSNSets{\mathbf{SNSets}}
\def\bNSets{\mathbf{NrSets}}
\def\bBanSets{\mathbf{BanSets}}

\def\bSN{\mathbf{SN}}
\def\bNr{\mathbf{Nr}}
\def\bBan{\mathbf{Ban}}

\def\bComm{\mathbf{Comm}}
\def\bsComm{\mathbf{sComm}}
\def\bsSets{\mathbf{sSets}}

\def\bSimp{\mathbf{Simp}}
\def\bMor{\mathbf{Mor}}
\def\bHo{\mathbf{Ho}}

\def\bFunc{\mathbf{Func}}

\def\bRings{\mathbf{Rings}}
\def\bAff{\mathbf{Aff}}
\def\bdAff{\mathbf{dAff}}
\def\bAn{\mathbf{An}}
\def\bdAn{\mathbf{dAn}}
\def\bStack{\mathbf{Stack}}
\def\bdStack{\mathbf{dStack}}
\def\bPerfField{\mathbf{PerfField}}

\def\bBox{{\pmb{\times}}}
\def\bTri{{\pmb{\ast}}}

\author{Federico Bambozzi, Oren Ben-Bassat, Kobi Kremnizer}

\title{Analytic geometry over $\F_1$ and the Fargues-Fontaine curve}

\thanks{The first author acknowledges the University of Regensburg with the support of the DFG funded CRC 1085 "Higher Invariants. Interactions between Arithmetic Geometry and Global Analysis".}

\begin{document}

\address{Federico Bambozzi, Fakult\"{a}t f\"{u}r Mathematik,
Universit\"{a}t Regensburg,
93040 Regensburg, 
Germany}
\bigskip
\email{Federico.Bambozzi@mathematik.uni-regensburg.de}
\address{Oren Ben-Bassat, Department of Mathematics, University of Haifa,
Mount Carmel,
Haifa, 31905
Israel}
\hfill \break
\bigskip
\email{ben-bassat@math.haifa.ac.il}
\address{Kobi Kremnizer, Mathematical Institute,
Radcliffe Observatory Quarter,
Woodstock Road,
Oxford
OX2 6GG, UK}
\hfill \break
\bigskip
\email{Yakov.Kremnitzer@maths.ox.ac.uk}

\begin{abstract}
This paper develops a theory of analytic geometry over the field with one element. The approach used is the analytic counter-part of the To\"en-Vaqui\'e  theory of schemes over $\F_1$, \ie the base category relative to which we work out our theory is the category of sets endowed with norms (or families of norms). Base change functors to analytic spaces over Banach rings are studied and the basic spaces of analytic geometry (like polydisks) are recovered as a base change of analytic spaces over $\F_1$. We end by discussing some applications of our theory to the theory of the Fargues-Fontaine curve and to the ring Witt vectors.
\end{abstract}

\maketitle

\tableofcontents

\section*{Introduction}

In recent years the theory of schemes over $\F_1$ has seen many attempts of foundations on better and more solid basis in order to overcome the limitations of the known theories. Nowadays all these different approaches form a huge zoo of theories of which it is difficult to keep track. We just mention the ones we know better: \cite{Toen}, \cite{Dur}, \cite{Har}, \cite{Lor}. A good attempt of giving an overview of the different theories and their relations has been done in \cite{PL}, several years ago. 

This paper does not aim to add another entry in this intricate panorama of theories of schemes over $\F_1$. Our aim is of doing something different: we develop a theory of analytic spaces over $\F_1$. To the best of our knowledge, the only previous attempt of developing such a theory was given by Vladimir Berkovich in the talk \cite{BerF1}, which has never been fully developed. Here we propose a different and much broader approach to such a foundational problem and present some of its possible applications to the theory of the Fargues-Fontaine curve. We do not aim to give a direct strong contribution to $p$-adic Hodge Theory but we would like to discuss some of its features - mainly the Fargues-Fontaine curve - from a new perspective in the hope of shedding some new light on its fundamental nature.

There are several reasons for working out such a theory of analytic geometry over $\F_1$. First of all, it is a natural question to ask if such a theory exists and if it is useful for the progress of mathematics. But even more important problems in arithmetic ask for methods that go beyond the theory of schemes and algebraic geometry as the Archimedean place of $\Z$ (and the Archimedean places of the rings of integers of number fields) are objects of analytic nature which cannot be fully understood through algebraic geometry. Thus, since the main motivations for introducing the idea of geometry over $\F_1$ come from arithmetic it seems more important to have a theory of analytic spaces over $\F_1$ with respect to a theory of schemes over $\F_1$. Another motivation for writing this work is to have a well-established basic language of derived analytic geometry over $\F_1$ and simplicial normed/bornological sets which we are planning to use in the near future for other, more involved, applications. Thus, the full strength of the homotopical methods introduced is not used in the applications provided in the last section of the present paper. 

The paper is structured as follows. In the first section we introduce the categories of normed and bornological sets and we study some of their basic properties. These are the categories with respect to which we develop our geometry, in the sense of HAG (cf. \cite{TVe2} and \cite{TVe3}). In section 2 we introduce the base change functor from the category of normed (or bornological) sets to the category of normed (Banach or bornological) $R$-modules. This functor is the analytic version of the free $R$-module functor of abstract algebra. Then, in section 3, we introduce the categories of simplicial bornological modules over $\F_1$ and over any Banach ring. We put on these categories a model structure which makes them well-suited for making them a HAG context. In particular, these categories turn out to be combinatorical symmetric monoidal model categories (see Theorem \ref{thm:model_monoidal} and the discussion before it). In section 4 we study some properties of bornological rings which are relevant in applications. In particular we introduce new Fr\'echet-like structures on the $p$-adic numbers which will be related with the Fr\'echet structure on the ring of functions on the affine covering of the Fargues-Fontaine curve and we will give a geometric interpretation of these structures in the analytic spectrum of $\Z$. Finally, in section 5 we introduce the analytic spaces over the field with one element. We start by directly introducing the notion $\infty$-analytic space and $\infty$-analytic stack with the respective derived notions. Then we define the category of analytic spaces over $\F_1$ as the full subcategory of derived analytic spaces that are concentrated in degree 0. In this way we can easily introduce the homotopy Zariski topology (as introduced by To\"en-Vezzosi in any HAG context) on the category of affine derived analytic spaces over $\F_1$ (see Definition \ref{defn:homotopy_zar_topology}). It has been proved in previous works of the authors (see \cite{BaBe}, \cite{BaBeKr} and \cite{BeKr}) that the homotopy Zariski topology on the category of simplicial Banach/bornological algebras over complete valued field restricts to the usual topology when it is considered over ``affine" (\ie affinoid, Stein or Stein-like) analytic spaces. It follows easily that the base change functor we defined for $\F_1$-bornological modules and $\F_1$-bornological algebras (\ie bornological monoids) induces a continuous functor from the category of $\infty$-analytic (resp. derived analytic, resp. analytic) spaces (resp. stacks) over $\F_1$ to $\infty$-analytic (resp. derived analytic, resp. analytic) spaces (resp. stacks) over $R$, for any Banach ring (see Proposition \ref{prop:base_changes}). In the last part of section 5 we discuss some of the basic objects of analytic geometry over $\F_1$, like the polydisks and annuli, and we see how the base change of them gives back the usual polydisks and annuli of analytic geometry.

The last section is devoted to some applications of the notions introduced so far. After recalling the basic notion of the theory of the Fargues-Fontaine curve we prove our main theorem (cf. Theorem \ref{thm:main}) which says that for some specific choices of a perfectoid field $\kappa$ of characteristic $p$, the ring of functions on the affine covering of the Fargues-Fontaine curve associated to $\kappa$ can be recovered, with its Fr\'echet structure, as a base change of a monoid from $\F_1$. Then, we describe how to see the action of the Frobenius on this Fr\'echet algebra also coming from an action over $\F_1$. In this way, we can define some analytic stacks over $\F_1$ whose base changes with suitable $p$-adic fields give Fargues-Fontaine curves for some perfectoid fields (cf. Theorem \ref{thm:FF}). These results can be easily generalized to many other cases and we do not aim to be exhaustive in the current treatment of the topic, since our main target here is just to show that the theory developed so far is worth being studied, avoiding too complex technicalities.

Then one can wonder what happens when the bornological monoids we showed to give the Fargues-Fontaine curve and its affine covering are base changed over $\R$ or $\C$. We briefly mention that, in the cases we studied, one gets a left half plane of the complex numbers, presented as the universal cover of the pointed disk via the exponential map, and the ``Archimedean" Fargues-Fontaine curve in this context would be a quotient of it. These results strengthen the idea that Witt vectors should be considered as the non-Archimedean replacement of the exponential map. We will investigate this topic deeply in a future work.

We end the last section by discussing an analytic version of the construction of the ring of Witt vectors. The idea is that although the construction of the ring of Witt vectors is a purely algebraic construction, it is usually applied to rings which have norms, valuation or some kind of additional structure. In this case it makes sense to ask for growth conditions on Witt vectors, as it is done in the theory of the Fargues-Fontaine curve. We provide a simple definition of a family of endofunctors on the category of Banach/bornological rings which virtually encompass all the known construction in a unified one for which Witt vectors are thought as an arithmetic analogue of Laurent series (as it is already done in \cite{FF} and \cite{FF2}).

For all applications provided in the last section much more can be said and we will devote future papers in better developing several instances that arose.

\subsection*{Acknowledgements}

The authors would like to thank Denis-Charles Cisinski, Florent Martin, Frederic Paugam, Raffael Singer and Adam Topaz for important suggestions and discussions about the topics of this paper.

\begin{notation} \label{notation:box}
In this work we deal with many closely related categories and very often we state properties that hold for all of them whose proof is identical in all cases. Therefore we implement the following notation for avoiding a lot of repetitions:
\begin{itemize}
\item The symbol $\bBox$ in the name of a category will stand for one of the following: $\bSN$ (semi-normed), $\bNr$ (normed), $\bBan$ (Banach). For example $\bBox\bSets$ denotes the category of semi-normed sets or the category of normed sets or the category of Banach sets.
\item We use the symbol $\ootimes$ in combination with the symbol $\bBox$ in the names of the categories to denote the closed monoidal structures these categories are equipped with. Hence, it will stand for the uncompleted tensor product $\otimes$ when $\bBox$ stands for $\bSN$ or $\bNr$ and for the completed tensor product $\wotimes$ when $\bBox$ stands for $\bBan$. 
\item We use the symbol $\bTri$ to mean either a Banach ring or $\F_1$.
\end{itemize}

Moreover, we use the following (quite standard) notations/conventions:

\begin{itemize}
\item All rings are commutative and with identity.
\item Given a category $\bC$ we often use the abuse of notation $X \in \bC$ for meaning that $X$ is an object of $\bC$.
\item Given a category $\bC$ we denote by $\bInd(\bC)$ the category of ind-objects of $\bC$ (cf. \S 8 expos\'e 1 of \cite{SGA4} for the precise definition of this category).
\item Given a category $\bC$ we denote by $\bInd^m(\bC)$ the category of essentially monomorphic ind-objects of $\bC$, which means the full subcategory of objects of $\bInd(\bC)$ which are isomorphic (as ind-objects) to some inductive system with monomorphic system morphisms.
\item For any closed symmetric monoidal category $\bC$ we denote by $\bComm(\bC)$ the category of commutative monoids over $\bC$.
\item For any category $\bC$ we denote with $\bMor(\bC)$ the class of morphisms of $\bC$.
\item $\R_+$ denotes the set of positive real number whereas $\R_{\ge 0}$ the set of non-negative ones.
\item $\bSets$ is the category of sets, $*\bSets$ the category of pointed sets, $\bAb$ the category of abelian groups, $\bRings$ the category of commutative rings. 
\item Give a closed symmetric monoidal category $(\bC, \otimes)$ and an object $X \in \bC$ we denote by
\[ S_\bC(X) = \bigoplus_{n \in \N} S^n(X) \]
the free symmetric tensor algebra over $X$, where $S^n(X)$ is the quotient of $X^{\otimes n}$ by the action of the symmetric group.
\item If $M$ is a (semi-)normed module over a Banach ring $R$ and $r \in \R_+$ is a scalar, we denote by $[M]_r$ the module obtained by rescaling the norm of $M$ by the factor $r$. More precisely, if $|\cdot|$ is the norm of $M$ then $[M]_r$ is equipped with the (semi-)norm $r |\cdot|$.
\end{itemize}
\end{notation}

\section{Semi-normed and bornological sets}

\begin{defn} \label{defn:semi_normed_set}
A \emph{semi-normed set} is a pointed set $(X, x_0)$ equipped with a function $f: X \to \R_{\ge 0}$ such that $f(x_0) = 0$. We will often say that $f$ is the \emph{semi-norm} of $(X, x_0)$ and we often will not mention the marked point.
\end{defn}

We consider on $(X, x_0)$ the uniformity induced by pulling back the metric uniformity of $\R_{\ge 0}$. We call it the \emph{canonical uniformity} of $(X, x_0)$.
In this way on $X$ it is defined a notion of closeness and a relative notion of Cauchy filters with respect to which it is possible to talk about the completeness of $X$.

\begin{defn} \label{defn:normed_set}
A \emph{normed set} is a pointed set $(X, x_0)$ equipped with a function $f: X \to \R_{\ge 0}$ such that $f(x_0) = 0$ and $f(x) \ne 0$ if $x \ne x_0$. In this case $f$ is called a \emph{norm}.
\end{defn}

\begin{rmk}
The data of a normed pointed set is equivalent to the data of a non-pointed set $X$ endowed with a strictly positive function $f: X \to \R_+$.
\end{rmk}

\begin{defn} \label{defn:Banach_set}
A \emph{Banach set} is a normed set $(X, x_0)$ for which all Cauchy filters with respect to the canonical uniformity converge to a limit point.
\end{defn}

\begin{defn} \label{defn:morphisms_seminormed_sets}
A \emph{bounded morphism} $\phi: (X, x_0) \to (Y, y_0)$ between two semi-normed sets is a morphism of pointed sets such that $|\phi(x)| \le C |x|$ for all $x \in X$, for a constant $C > 0$. A \emph{contracting morphism} between two semi-normed sets is a bounded morphism with $C = 1$.
\end{defn}

We define the following categories: 
\begin{itemize}
\item $\bSNSets$ the category of semi-normed sets with bounded morphisms;
\item $\bNSets$ the category of normed sets with bounded morphisms;
\item $\bBanSets$ the category of Banach sets with bounded morphisms;
\item $\bSNSets^{\le 1}$ the category of semi-normed sets with contracting morphisms;
\item $\bNSets^{\le 1}$ the category of normed sets with contracting morphisms;
\item $\bBanSets^{\le 1}$ the category of Banach sets with contracting morphisms.
\end{itemize}

As mentioned in Notation \ref{notation:box} we use the symbol $\bBox\bSets$ to denote the first three categories of the previous list and the symbol $\bBox\bSets^{\le 1}$ for the last three.

Among these categories are defined several functors. There are the (non-full) inclusions
\[ \bBox\bSets^{\le 1} \rhook \bBox\bSets. \]
These inclusions commute with finite limits and finite colimits as they are described in Proposition \ref{prop:lim_colim_ban_sets}. And the full inclusions
\[ \bBanSets \rhook \bNSets \rhook \bSNSets \] 
\[ \bBanSets^{\le 1} \rhook \bNSets^{\le 1} \rhook \bSNSets^{\le 1}, \]
which commute with projective limits as a consequence of Proposition \ref{prop:adjoint_inclusion_sets}.

The separation functors
\[ \sep: \bSNSets \rhook \bNSets \]
\[ \sep: \bSNSets^{\le 1} \rhook \bNSets^{\le 1} \]
which associate to each semi-normed set $(X, |\cdot|_X)$ a normed one given by 
\[ \sep(X) = \frac{X}{\ker(|\cdot|_X)} \]
where $\ker(|\cdot|_X) = \{x \in X | |x|_X = 0 \}$ (\ie $\sep(X)$ is equal to $X$ with all the points with $|x|_X = 0$ identified to the base point) and the norm of $\sep(X)$ is given by
\[ \|x\| = |y|_X, \ \  \ y \in (X \to \sep(X))^{-1}(x). \]
Clearly $\sep(X)$ is a normed set. Then, there are the completion functors
\[ \ol{(-)}: \bNSets \rhook \bBanSets \]
\[ \ol{(-)}: \bNSets^{\le 1} \rhook \bBanSets^{\le 1} \]
which associate to each normed set its completion with respect to its canonical uniformity. The values of the completion norm are given by the limits of the values to the original norm. The following proposition is easy to check.

\begin{prop} \label{prop:adjoint_inclusion_sets}
The functors $\sep$ are left adjoints to the inclusion functors $\bNSets \rhook \bSNSets$ and $\bNSets^{\le 1} \rhook \bSNSets^{\le 1}$ and the completion functors $\ol{(-)}$ are left adjoints to the inclusion functors $\bBanSets \rhook \bNSets$ and $\bBanSets^{\le 1} \rhook \bNSets^{\le 1}$.
\end{prop}
\begin{proof}
Straight forward checking.
\end{proof}

\begin{prop} \label{prop:lim_colim_ban_sets}
The categories $\bBox\bSets$ are finitely complete and finitely cocomplete. The categories $\bBox\bSets^{\le 1}$ are complete and cocomplete.
\end{prop}
\begin{proof}
The finite products are given by the product of the underlying pointed sets equipped with the semi-norm obtained by the component-wise sum of the semi-norms (or equivalently with the max norm). One can check that this construction gives finite products in all the categories mentioned above. The finite limits can be constructed, in all cases, as subspaces of finite products, where the subspaces are equipped with the semi-norm induced by the inclusion. The finite coproducts are given by the coproducts of the underlying pointed sets (\ie disjoint unions where the marked points are identified) with the obvious extension of the given semi-norms to the coproduct.
The finite colimits in $\bSNSets$ and $\bSNSets^{\le 1}$ are calculated by taking quotients of coproducts, where coequalizers are simply coequalizers of the underlying pointed sets equipped with the quotient semi-norm which, on each point, is computed as the $\inf$ of the semi-norm of its fiber. The finite colimits in $\bNSets$, $\bNSets^{\le 1}$, $\bBanSets$ and $\bBanSets^{\le 1}$ are obtained by applying the separation and completion functors, respectively, to the colimits computed in the categories $\bSNSets$ and $\bSNSets^{\le 1}$.

The same reasoning used to prove that the categories of Banach (normed, semi-normed) vector spaces over a valued field do not have infinite products applies to prove that $\bBox\bSets$ do not have infinite products, nor coproducts. Finally, we can describe the infinite products of a family of objects $\{X_i\}_{i \in I}$ in the contracting categories as the subset of the product of pointed set given by 
\[ {\prod_{i \in I}}^{\le 1} X_i \doteq \{ (x_i) \in \prod_{i \in I} X_i | \sup_{i \in I} |x_i|_i < \infty  \} \]
equipped with the semi-norm
\[ \|(x_i)\| = \sup_{i \in I} |x_i|_i. \]
Whereas coproducts are given by the wedged disjoint union of sets 
\[ \coprod_{i \in I} X_i  \]
with the obvious extension of the semi-norms of the $X_i$.
\end{proof}

\begin{rmk} 
Notice that the coproduct of an arbitrary family of Banach sets is a Banach sets even if the coproduct is computed as a coproduct of normed sets. This is in contrast to what happens in the category of contracting Banach modules over a Banach ring for which infinite coproducts of the underlying normed sets are not complete in general.
\end{rmk}

\begin{notation} \label{notation:contracting_lim}
We will use the notation
\[  {\prod_{i \in I}}^{\le 1} X_i \] 
for the product of $\{X_i\}_{i \in I}$ in the contracting categories and 
\[  {\coprod_{i \in I}}^{\le 1} X_i \] 
for the coproduct in the contracting categories.
\end{notation}

As for the theory of Banach spaces, it is often useful to consider equivalence classes of semi-norms instead of just a single semi-norm.

\begin{defn} \label{defn:equivalence_norm}
	Let $X$ be a pointed set and $|\cdot|_1$, $|\cdot|_2$ two semi-norms on $X$. We say that $|\cdot|_1$ is \emph{equivalent} to $|\cdot|_2$ if the identity map $(X, |\cdot|_1) \to (X, |\cdot|_2)$ is an isomorphism in $\bSN \bSets$.
\end{defn}

Following \cite{Kas}, Section 5.1, we recall that on any category with finite limits and finite colimits one can define the notion of image and coimage of a morphism. For the convenience of the reader we recall their definitions.

\begin{defn} \label{defn:image_coimage}
	Let $\bC$ be a category with finite limits and finite colimits and $f: X \to Y$ a morphism in $\bC$. We define
	\begin{itemize}
	\item the \emph{image} of $f$ as
	\[ \im(f) \doteq \ker ( X \rightrightarrows Y \coprod_X Y ); \]
	\item the \emph{coimage} of $f$ as
	\[ \coim(f) \doteq \coker ( X \times_Y X \rightrightarrows Y ). \]
	\end{itemize}
\end{defn}

\begin{defn} \label{defn:strict_morphism}
	Let $\bC$ be a category with finite limits and finite colimits and $f: X \to Y$ a morphism in $\bC$. We say that $f$ is \emph{strict} if the canonical morphism $\coim(f) \to \im(f)$ is an isomorphism.
\end{defn}

We recall the following results from Section 5.1 of \cite{Kas}.

\begin{prop} \label{prop:kas}
	Let $\bC$ be a category with finite limits and finite colimits and let $f: X \to Y$ be a morphism, then
	\begin{itemize}
		\item $f$ is an epimorphism if and only if the canonical map $\im f \to Y$ is an isomorphism.
		\item $f$ is a strict epimorphism if and only if the canonical map $\coim f \to Y$ is an isomorphism.
	\end{itemize}
	The dual statements, for monomorphisms, hold.
\end{prop}
\begin{proof}
	\cite{Kas} Proposition 5.1.2 (iv) and Proposition 5.1.5 (i).
\end{proof}

The next proposition describes monomorphisms, epimorphisms and strict morphisms in the categories $\bBox \bSets$ as easy corollaries of Proposition \ref{prop:kas}.

\begin{prop} \label{prop:morphisms_normed_sets}
	Let $\phi: X \to Y$ be a morphism in $\bBox \bSets$, then
	\ben
	\item $\phi$ is a monomorphism if and only if it is injective;
	\item $\phi$ is a strict monomorphism if and only if it is injective and the semi-norm on $X$ is equivalent to the semi-norm induced by the restriction of the one on $Y$;
	\item if $\bBox = \bSN$ or $\bBox = \bNr$, $\phi$ is an epimorphism if and only if it is surjective;
	\item if $\bBox = \bBan$, $\phi$ is an epimorphism if and only if it has dense image;
	\item $\phi$ is a strict epimorphism if and only if it is surjective and the quotient norm induced by $\phi$ is equivalent to the norm of $Y$.
	\een
\end{prop}
\begin{proof}
	All the statements follow easily from Proposition \ref{prop:kas}.
\end{proof}

Next proposition describes monomorphisms, epimorphisms and strict morphisms in the categories $\bBox \bSets^{\le 1}$.

\begin{prop} \label{prop:morphisms_contracting_normed_sets}
	Let $\phi: X \to Y$ be a morphism in $\bBox \bSets^{\le 1}$, then
	\ben
	\item $\phi$ is a monomorphism if and only if it is injective;
	\item $\phi$ is a strict monomorphism if and only if it is injective and the semi-norm on $X$ is identical to the semi-norm induced by the restriction of the one on $Y$;
	\item if $\bBox = \bSN$ or $\bBox = \bNr$, $\phi$ is an epimorphism if and only if it is surjective;
	\item if $\bBox = \bBan$, $\phi$ is an epimorphism if and only if it has dense image;
	\item $\phi$ is a strict epimorphism if and only if it is surjective and the quotient norm induced by $\phi$ is identical to the norm of $Y$.
	\een
\end{prop}
\begin{proof}
	All the statements follow easily from Proposition \ref{prop:kas}.
\end{proof}

The categories $\bBox\bSets$, and $\bBox\bSets^{\le 1}$ have a symmetric closed monoidal structure described in a uniform way as follows. The internal hom is given by
\[ \ul{\Hom}((X, |\cdot|_X), (Y, |\cdot|_Y)) = ( \Hom_{\bBox\bSets}((X, |\cdot|_X), (Y, |\cdot|_Y)), \|\cdot\|_{\sup} ) \]
where $\|\cdot\|_{\sup}$ of a bounded morphism $h$ is given by
\[ \|h\|_{\sup} = \sup_{x \in  X } \frac{|h(x)|_Y}{|x|_X}. \]

The monoidal structure is given by
\[ (X, |\cdot|_X) \ootimes_{\F_1} (Y, |\cdot|_Y) = (X \wedge Y, |\cdot|_X|\cdot|_Y), \]
where $X \wedge Y$ denotes the smash product of the underlying pointed sets.

\begin{prop} \label{prop:monoidal_structure}
The categories $\bBox\bSets$, and $\bBox\bSets^{\le 1}$ have a symmetric closed monoidal structure as described so far.
\end{prop}
\begin{proof}
	For the morphisms of the underlying pointed sets, one has a bijection
	\[ \Hom_\bSets(X \wedge Y, Z) \cong \Hom_\bSets(X, \ul{\Hom}_\bSets(Y, Z))  \]
	given by
	\[ f((-, -)) \mapsto f((-, y)) \]
	for any $y \in Y$. Now, suppose that $f$ is bounded. This means that for each $z \in Z$ there exists a $C > 0$ such that
	\[ |f(x, y)|_Z \le C |(x, y)|_{X \times Y} = C|x|_X|y|_Y. \]
	Now, if we consider the map
	\[ f(x, -): Y \to Z \]
	we get that 
	\[ |f(x, -)|_{\sup} = \sup_{y \in  Y } \frac{|f(x, y)|_Z}{|y|_Y} \le C |x|_X  \]
	which proves that every bounded map in $\Hom_\bSets(X \wedge Y, Z)$ is mapped to a bounded map in $\Hom_\bSets(X, \ul{\Hom}(Y, Z))$. On the other hand, suppose that the inequality 
	\[ |f(x, -)|_{\sup} \le C |x|_X  \]
  holds. This implies that
  \[ \sup_{y \in  Y } \frac{|f(x, y)|_Z}{|y|_Y} \le C |x|_X \]
  which implies that
  \[ \frac{|f(x, y)|_Z}{|y|_Y} \le C |x|_X \]
  is true for all $y \in Y$, and hence $f$ is bounded.
  
  Clearly, if $X$ and $Y$ are normed (resp. Banach) also $X \ootimes_{\F_1} Y$ is normed (resp. Banach). The fact that also $\ul{\Hom}(X, Y)$ is normed (resp. Banach) is a consequence of the fact that 
  \[ |f|_{\sup} = 0 \Leftrightarrow \sup_{x \in  X} \frac{|f(x)|_Y}{|x|_X} = 0 \Leftrightarrow f \equiv 0 \]
  and that if $f_i$ is a convergent sequence of bounded maps in $\ul{\Hom}(X, Y)$ then the map
  \[ f(x) = \lim_{i \to \infty} f_i(x) \] 
  is a well-defined element of $\ul{\Hom}(X, Y)$ which is the limit of the sequence.
\end{proof}

\begin{rmk}
Notice that one does not need to introduce the completed tensor product of Banach sets as it is needed in classical functional analysis for Banach spaces.
\end{rmk}

Although the categories $\bBox\bSets^{\le 1}$ are complete and cocomplete we would like to add all limits and colimits to $\bSNSets$, $\bNSets$, $\bBanSets$ in a more natural way, \ie in a way that permits to continue to work with all bounded morphisms and not only with the contracting ones. There are more substantial reasons for enlarging the categories $\bBox\bSets$ which will become clearer later on.

The most natural way to achieve this goal is by considering the ind-categories $\bInd(\bBox\bSets)$. These latter categories are complete and cocomplete: assertion b) of Proposition 8.9.1 of the first expos\'e of \cite{SGA4} ensures the existence of projective limits and the combination of Proposition 8.5.1 and assertion c) of the Proposition 8.9.1 of the same expos\'e ensures the existence of colimits. Moreover, the closed monoidal symmetric structures naturally extend to the ind-categories by the formula
\[ {``\limind_{i \in I}"} M_i \ootimes_{\F_1} {``\limind_{j \in J}"} N_j = \underset{(i, j) \in I \times J}{``\limind"} M_i \ootimes_{\F_1} N_j \]
and
\[ \ul{\Hom}({``\limind_{i \in I}"} M_i, {``\limind_{j \in J}"} N_j) = \limpro_{i \in I} \limind_{j \in J} \ul{\Hom}(M_i, N_j). \]

\begin{defn} \label{defn:bornological_sets}
We define (cf. Notation \ref{notation:box} for the meaning of the symbol $\bInd^m$)
\begin{itemize}
\item the category of \emph{bornological sets} as the category $\bInd^m(\bSNSets)$;
\item the category of \emph{separated bornological sets} as the category $\bInd^m(\bNSets)$;
\item the category of \emph{complete bornological sets} as the category $\bInd^m(\bBanSets)$.
\end{itemize}
\end{defn}

The categories $\bInd(\bBox\bSets)$ have a canonical functor $U$ to $\bBox\bSets$ given by
\[ U({``\limind_{i \in I}"} M_i) = \limind_{i \in I} U(M_i) \]
where $U(M_i)$ is the forgetful functor from $\bBox\bSets$ to $\bSets$. One can prove that $U$ is not a forgetful functor, \ie it is not faithful, using the same reasoning of Remark 3.33 of \cite{BaBe}. The next proposition shows one of the advantages of using the category of essential monomorphic objects instead of $\bInd(\bBox\bSets)$.

\begin{prop} \label{prop:concrete_ind_mono}
The categories $\bInd^m(\bBox\bSets)$ are concrete, \ie the functor $U$ is faithful when restricted to $\bInd^m(\bBox\bSets)$.
\end{prop}
\begin{proof}
	Proposition 3.31 of \cite{BaBe} can be applied to the case of $\bInd^m(\bBox\bSets)$ thanks to Proposition \ref{prop:morphisms_normed_sets}.
\end{proof}

\begin{rmk}
Proposition \ref{prop:concrete_ind_mono} tells that the categories $\bInd^m(\bBox\bSets)$ can be thought as categories made of objects which are sets endowed with an additional structure, which is given by a family of semi-norms, and whose morphisms are maps of sets which are compatible with respect to these structures. This is in analogy with the theory of bornological vector spaces over non-trivially valued fields. This concrete description of objects and morphisms does not exist for the categories $\bInd(\bBox\bSets)$.
\end{rmk}

We need to describe strict morphisms in  $\bInd(\bBox\bSets)$ and $\bInd^m(\bBox\bSets)$.

\begin{prop} \label{prop:morphisms_ind_normed_sets}
	Let $\phi: X \to Y$ be a morphism in $\bInd(\bBox \bSets)$ or $\bInd^m(\bBox\bSets)$, then $\phi$ is a monomorphism (resp. strict monomorphism, resp. epimorphism, resp. strict epimorphism) if and only if it can be represented as a direct limit of monomorphisms (resp. strict monomorphism, resp. epimorphism, resp. strict epimorphism).
\end{prop}
\begin{proof}
Analogous to Proposition 2.10 of \cite{BaBe}.
\end{proof}

We give some key examples of normed and bornological sets.

\begin{exa} \label{exa:basic}
\ben
\item Let $R$ be a Banach ring. The underlying set of any semi-normed, normed, Banach module over $R$ is a semi-normed, normed, Banach set.
\item Let $k$ be a non-trivially valued field. The underlying set of any bornological, separated bornological, complete bornological vector space over $k$ is a bornological, separated bornological, complete bornological set (see \cite{H2} for an introduction to the theory of bornological vector spaces).
\item We denote by $\ast_r$, with $r \in \R_+$, the one point bornological set $(\{\ast, 0\}, |\cdot|_r)$ for which $|\ast|_r = r$.
\item An important object in $\bComm(\bBox \bSets^{\le 1})$ is the symmetric algebra with radius $r \in \R_+$
\[ S_{\bBox \bSets^{\le 1}}(\ast_r) = {\coprod_{n \in \N}}^{\le 1} \ast_r^{\ootimes n}. \]
More explicitly, $S_{\bBox \bSets^{\le 1}}(\ast_r)$ is the normed monoid whose underlying monoid is $\N$ and whose norm is given by
\[ |n|_r = r^n \] 
for $n > 0$ and $|0|_r = 0$.
\item The norm on $S_{\bBox \bSets^{\le 1}}(\ast_r)$ can be immediately extended to $\Z$. We denote the Banach set obtained in this way by $\cZ_r$. Also, the norm can be extended to $\Q$ simply by
\[ \l |\frac{p}{q} \r |_r = r^{\frac{p}{q}} \]
for any $p, q \in \Z$. We denote this normed set with $\cQ_r$. Notice that $\cQ_r$ is not a Banach set because it is not complete with respect to its canonical uniformity.
\item We denote by $\cQ_r^+ \subset \cQ_r$ the subset of positive rational numbers equipped with the induced structure of bornological set.
\item It is useful to notice the following isomorphism of bornological sets
\begin{equation} \label{eqn:limind_Q_r}
\cQ_r^+ \cong \limind_{n \in \N} \frac{1}{n} S_{\bNr_{\F_1}^{\le 1}}(\ast_{\sqrt[n]{r}})
\end{equation}
where 
\[ \frac{1}{n} S_{\bNr_{\F_1}^{\le 1}}(\ast_{r}) = (\frac{1}{n}\N, |\cdot|_r) \]
and
\[ \l |\frac{m}{n} \r|_r = r^m. \]
\item Another interesting example of bornological set comes from the bijection $\Z_p \cong \F_p^\N$. Considering on $\F_p$ the trivial norm, and denoting by $[\F_p]_r$ the rescaling of the norm by the factor $r$, we can write the isomorphism of bare Banach sets
\[ \Z_p \cong {\prod_{n \in \N}}^{\le 1} [\F_p]_{p^{-n}}. \]
We will say more on these kind of constructions in the last section of the paper (cf. section \ref{sec:witt}).
\item Later on we will need to consider on $\Z$ and $\Q$ ``geometric" norms, as the ones introduced so far, but with different ``radii" for the positive and the negative numbers. Thus, if $r_1 < r_2$ we denote by $\cZ_{r_1, r_2}$ the group $\Z$ equipped with the norm
\[ |n|_{r_1, r_2} = \begin{cases}
(r_1)^n \text{ if } n < 0 \\
(r_2)^n \text{ if } n > 0
\end{cases}.
 \]
The same description holds for $\cQ_{r_1, r_2}$.
\een
\end{exa}

We end this section with some lemmata and propositions about the computation of limits and colimits of bornological sets. Colimits of bornological sets are easy to compute, whereas limits are harder.

\begin{lemma} \label{lemma:product_born_sets}
Let $\{ M_i \}_{i \in I}$ be a family of objects of $\bInd(\bBox \bSets)$. Let's write
\[ M_i \cong ``\limind_{j \in J_i}" M_{i, j}. \]
Then, we have the isomorphism
\[ \prod_{i \in I} M_i = ``\limind_{\phi \in \Phi}" M_\phi \]
where $\phi = (\phi_1, \phi_2)$ is an element of the set of functions
\[ \Phi = \{ (\phi_1, \phi_2) | \phi_1: I \to \prod_{i \in I} J_i, \phi_i(i) \in J_i, \phi_2: I \to \N_{\ge 1}  \} \]
equipped with the partial order
\[ (\phi_1, \phi_2) \le (\phi_1', \phi_2') \iff \phi_1(i) \le \phi_1'(i), \phi_2(i) \le \phi_2'(i), \forall i \in I,  \]
and 
\[ M_\phi = \l \{ (x_i) \in \prod_{i \in I} M_{i, \phi_1(i)} | \frac{|x_i|}{\phi_2(i)} \text{ is bounded}  \r \} \]
is equipped with the norm
\[ |(x_i)|_{\phi} = \sup_{i \in I} \frac{|x_i|}{\phi_2(i)}. \]
\end{lemma}
\begin{proof}
One can easily check that the object defined by
\[ ``\limind_{\phi \in \Phi}" M_\phi \]
satisfies the universal property of the product by noticing that $M_\phi$ is the contracting coproduct of the family $\{ M_i \}_{i \in I}$ where the norm of $M_i$ is rescaled by the factor $\frac{1}{\phi_2(i)}$. Hence, $M_\phi$ satisfies the universal property of the contracting direct product with respect to the rescaled norms and taking the direct limit for $\frac{1}{\phi_2(i)} \to 0$ we get the universal property of the direct product.
\end{proof}

\begin{rmk} \label{rmk:product_born_sets}
We notice that Lemma \ref{lemma:product_born_sets} does not only apply to the category $\bInd(\bBox \bSets)$ but the same description of direct products holds for $\bInd^m(\bBox \bSets)$ because the inclusion functor $\bInd^m(\bBox \bSets) \rhook \bInd(\bBox \bSets)$ commutes with all limits.
\end{rmk}

\begin{lemma}
Let $U: \bInd^m(\bSN \bSets) \to *\bSets$ be the forgetful functor, then $U$ commutes with all limits and colimits. The forgetful functors $U_S: \bInd^m(\bNr \bSets) \to *\bSets$ and $U_C: \bInd^m(\bBan \bSets) \to *\bSets$ commute with limits and monomorphic filtered colimits.
\end{lemma}
\begin{proof}
Let's consider first the case of $U: \bInd^m(\bSN \bSets) \to * \bSets$. The explicit description of direct products given by Lemma \ref{lemma:product_born_sets} immediately implies that $U$ commutes with products. The description of kernels given in Proposition \ref{prop:morphisms_ind_normed_sets} implies that it commutes with finite limits and hence it commutes with all limits. The commutation of $U$ with colimits follows by easy explicit computations.

Since the inclusion functors $\bInd^m(\bBan \bSets) \rhook \bInd^m(\bNr \bSets) \rhook \bInd^m(\bSN \bSets)$ commute with all limits and monomorphic filtered colimits it follows that $U_S$ and $U_C$ commute with these limits and colimits, but these inclusion functors do not commute with cokernels. Hence, $U_S$ and $U_C$ do not commute with all colimits.
\end{proof}

\begin{exa}[Finite bornologies] \label{exa:finite_dim_bornologies}
This example generalizes a basic result of functional analysis: the fact that on a finite dimensional vector space (over a non-trivially valued field) there exists only one separated bornology of convex type, up to isomorphism (cf. Proposition 12, $n^\circ$ 4, \S 3, Chapter 1 of \cite{H2}).

The first observation is that each finite pointed set $[n] = (0, 1, \ldots, n)$ ($0$ is the base point) admits only one structure of normed set, up to isomorphism. Indeed, consider two norms $([n], |\cdot|_1)$ and $([n], |\cdot|_2)$ then, let $N_1 = \underset{k \in [n]}\max \frac{|k|_1}{|k|_2}$ and $N_2 = \underset{k \in [n]}\max \frac{|k|_2}{|k|_1}$ we have that
\[ |k|_1 \le N_1 |k|_2 \]
and 
\[ |k|_2 \le N_2 |k|_1. \]
Even more, if $\{|\cdot|_i\}_{i \in I}$ is a family of norms on $[n]$ which defines a bornology (we can always think to a bornology in this way thanks to Proposition \ref{prop:concrete_ind_mono}) which is separated. We can consider the norm
\[ |k|_{\inf} \doteq \inf_{i \in I} |k|_i + \delta, \ \ k \ne 0 \]
for a small $\delta > 0$, so $|k|_{\inf} \ne 0$ for $k \ne 0$.
So, for any $k \in [n]$ we can find a $|\cdot|_{i_k}$ such that
\[ \frac{|k|_{i_k}}{|k|_{\inf}} \le 1 + \epsilon \]
and define the norm
\[ \|k\| = |k|_{i_k}. \]
The identity map
\[ ([n], \{|\cdot|_i\}_{i \in I}) \to ([n], |\cdot|_{\inf}) \]
is clearly bounded by the definition of $|\cdot|_{\inf}$ and the identity
\[  ([n], |\cdot|_{\inf}) \to ([n], \{|\cdot|_i\}_{i \in I}) \]
is bounded because it factors through $([n], \|\cdot\|)$ because $[n]$ is a finite set.
\end{exa}

\section{The base change functors}

In this section we look at the categories of Banach and bornological sets with a more geometrical perspective. We consider these categories as the basic categories where analytic geometry over the field with one element is defined. For emphasizing this shift of perspective we change the notation as follows:
\[ \bBox \bSets = \bBox_{\F_1} \]
\[ \bBox \bSets^{\le 1} = \bBox_{\F_1}^{\le 1} \]
as if we are working with modules over the hypothetical base Banach ring $\F_1$.

With this notation, for any Banach ring $R$ one has the base change functors
\[ (-) \ootimes_{\F_1} R: \bBox_{\F_1} \to \bBox_R \] 
from semi-normed (resp. normed, resp. Banach) modules over $\F_1$ to semi-normed (resp. normed, resp. Banach) modules\footnote{We are considering the categories $\bBox_{R}$ over a general Banach ring by making no distinction between the case when $R$ is non-Archimedean or not. Therefore we are considering the categories of all semi-normed, normed and Banach modules in contrast with the most common attitude of considering only non-Archimedean semi-normed modules on non-Archimedean Banach rings. We refer the reader to \cite{BaBe} for more information about this issue.} over $R$ given by
\[ (X, |\cdot|_X) \ootimes_{\F_1} R = {\coprod_{x \in X}}^{\le 1}  [R]_{|x|_X} \]
where $[R]_{|x|_X}$ denotes $R$ thought as a free semi-normed (resp. normed, resp. Banach) $R$-module whose norm has been rescaled by the factor $|x|_X$. 

\begin{rmk} 
	Notice that for the sake of simplicity we are forgetting about the fact that $X$ is a pointed set and that the base point $x_0 \in X$ is identified with the zero element of its base change as an $R$ module. Indeed, for being very precise, with the notation $(X, |\cdot|_X) \ootimes_{\F_1} R$ we are denoting the module 
	\[ \underset{x \in X, x \ne x_0} {{\coprod}^{\le 1}}  [R]_{|x|_X} \cup \{ 0 \}. \]
	For the sake of having a lighter notation we will use the convention that this latter set is denoted 
	\[ {\coprod_{x \in X}}^{\le 1}  [R]_{|x|_X} \]
	keeping in mind that the summand $[R]_{|x_0|_X}$ is contracted to zero.
\end{rmk}

One can think of $(X, |\cdot|_X) \ootimes_{\F_1} R$ as the free Banach module over $R$ generated by the (semi-)normed (pointed) set $X$. Of course the same definition also gives base changes
\[ (-) \ootimes_{\F_1} R: \bBox_{\F_1}^{\le 1} \to \bBox_R^{\le 1}. \]

\begin{rmk} \label{rmk:non-arch_base_change}
	As mentioned, the set 
	\[ (X, |\cdot|_X) \ootimes_{\F_1} R = {\coprod_{x \in X}}^{\le 1}  [R]_{|x|_X} = \{ (r_x)_{x \in X} | \sum_{x \in X} |r_x|_R|x|_X < \infty \}  \]
	is the set of $l^1$-summable sequences equipped with the $l^1$-norm 
	\[ \sum_{x \in X} |r_x|_R|x|_X. \]
	When $R$ is a non-Archimedean ring or a non-Archimedean field it could be convenient to consider the base change to the category of non-Archimedean Banach modules and hence the non-Archimedean coproduct. In this case the base change functor takes the form 
	\[ (X, |\cdot|_X) \ootimes_{\F_1} R = \{ (r_x)_{x \in X} | \forall \epsilon > 0, \#\{|r_x|_R|x|_X > \epsilon \} < \infty \}  \]
	with the norm 
	\[ \max |r_x|_R|x|_X. \] 
	These two base change functors are very different in general but their difference is irrelevant for the applications presented in this work. Indeed, the main application will be worked out for some Fr\'echet spaces which are defined as projective limits of Banach spaces and the two families of norms obtained using the max norms and the $l^1$ norms are equivalent, as it is explained in Appendix \ref{appendix:S_T}. So, the reader interested in non-Archimedean geometry and to study base changes with respect to non-Archimedean Banach rings can consider the $\max$ version of the base change functor everywhere in paper without affecting the theory developed and the final results. We refer the reader to the end of this section for further evidence that in practice the two choices are equivalent (but our choice works better for non-ultrametric rings).
\end{rmk}

\begin{prop} \label{prop:base_change_left_adjoint}
The base change functors $(-) \ootimes_{\F_1} R$ are left adjoints of the forgetful functors $\bBox_R \to \bBox_{\F_1}$ and $\bBox_R^{\le 1} \to \bBox_{\F_1}^{\le 1}$.
\end{prop}
\begin{proof}
Consider the case of a semi-normed set $(A, |\cdot|)$ and a semi-normed $R$-module $B$. By the definition of the base change, to give a morphism $A \ootimes_{\F_1} R \to B$ is equivalent to give an equibounded family of morphisms $\{ \phi_a: [R]_{|a|} \to B \}_{a \in A}$ which is equivalent to give a morphisms $A \to B$ of semi-normed sets.
\end{proof}

\begin{prop} \label{prop:base_change_exact}
	The base change functors $(-) \ootimes_{\F_1} R$ commute with monomorphisms and strict monomorphisms.
\end{prop}
\begin{proof}
	The proposition is a direct consequence of the characterization of monomorphisms and strict monomorphisms in $\bBox_{\F_1}$ given in Proposition \ref{prop:morphisms_normed_sets} and the characterization of monomorphisms and strict monomorphisms in $\bBox_R$ given in Proposition 3.14 of \cite{BaBe}.
\end{proof}

\begin{rmk} \label{rmk:not_commute_product}
	It is not true that the functors $(-) \ootimes_{\F_1} R$ preserve finite products therefore they are not exact functors. Whereas the next lemma shows that the functors $(-) \ootimes_{\F_1} R$ intertwine the monoidal structures of $\bBox_{\F_1}$ and $\bBox_R$. Moreover, it is easy to deduce from Proposition \ref{prop:base_change_exact} that $(-) \ootimes_{\F_1} R$ preserves equalizers.
\end{rmk}

\begin{lemma}
	Let $X, Y \in \bBox_{\F_1}$ and let $R$ be a Banach ring, then 
	\[ (X \ootimes_{\F_1} Y) \ootimes_{\F_1} R \cong (X \ootimes_{\F_1} R) \ootimes_R (Y \ootimes_{\F_1} R). \]
\end{lemma}
\begin{proof}
	Applying the definitions 
	\[ (X \ootimes_{\F_1} Y) \ootimes_{\F_1} R = \underset{(x,y) \in X \times Y}{{\coprod}^{\le 1}} [R]_{|(x,y)|_{X \times Y}} \]
	and
	\[ (X \ootimes_{\F_1} R) \ootimes_R (Y \ootimes_{\F_1} R) = ({\coprod_{x \in X}}^{\le 1} [R]_{|x|_X}) \ootimes_R ({\coprod_{y \in Y}}^{\le 1}  [R]_{|y|_Y}) \]
	and since $(-)\ootimes_R(-)$ commutes with contracting coproducts one gets that
	\[ (X \ootimes_{\F_1} R) \ootimes_R (Y \ootimes_{\F_1} R) \cong {\coprod_{x \in X}}^{\le 1} {\coprod_{y \in Y}}^{\le 1}( [R]_{|x|_X} \ootimes_R  [R]_{|y|_Y}). \]
	It is elementary to check that there is a canonical isometry
	\[ [R]_{|x|_X} \ootimes_R [R]_{|y|_Y} \to  [R]_{|x|_X |y|_Y} \]
	given by the codiagonal morphism.
\end{proof}

The base change functors can be defined also for the ind and bornological categories in the following way.

\begin{defn} \label{defn:born_base_change}
	Let $X \in \bInd(\bBox_{\F_1})$ or $X \in \bInd^m(\bBox_{\F_1})$ and let us write $X \cong \underset{i \in I}{``\limind"} X_i$. Then, for any Banach ring $R$ we define the \emph{base change of $X$ to $R$} as
	\[ X \ootimes_{\F_1} R = \underset{i \in I}{``\limind"} (X_i \ootimes_{\F_1} R) \in \bInd(\bBox_{R}). \]
\end{defn}

Notice that the base change functors $(-)\ootimes_{\F_1} R$ are well-defined also for the category $\bInd^m(\bBox_{\F_1})$ because they preserve monomorphisms (cf. Proposition \ref{prop:base_change_exact}).

\begin{prop} \label{prop:base_change_exact_born}
	The base change functors $(-) \ootimes_{\F_1} R: \bInd(\bBox_{\F_1}) \to \bInd(\bBox_{R})$ and $(-) \ootimes_{\F_1} R: \bInd^m(\bBox_{\F_1}) \to \bInd^m(\bBox_{R})$ are left adjoints to the forgetful functors and they preserve monomorphisms and strict monomorphisms.
\end{prop}
\begin{proof}
	The proof is straightforward.
\end{proof}

\begin{rmk}
	If $M$ is a (semi)-normed or bornological monoid then $M \wotimes_{\F_1} R$ has a natural structure of (semi)-normed or bornological ring over $R$ give by point-wise multiplication. Indeed, one can think to the ring $M \wotimes_{\F_1} R$ as the $\ell^1$-completion of the monoid ring $R[M]$, with respect to the norm
	\[ |\sum_{m \in M} r_m m | = \sum_{m \in M} |r_m|_R|m|_M. \]	
\end{rmk}

The next lemma is the analogue of Lemma \ref{lemma:product_born_sets} for computing products of bornological $R$-modules.

\begin{lemma} \label{lemma:calc_proj_limit_R}
	Let $\{ M_i \}_{i \in I}$ be a family of objects of $\bInd(\bBox_R)$. We write
	\[ M_i = ``\limind_{j \in J_i}" M_{i, j}. \]
	Then, we have the isomorphism
	\[ \prod_{i \in I} M_i \cong ``\limind_{\phi \in \Phi}" M_\phi \]
	where $\phi = (\phi_1, \phi_2)$ is an element of the set of functions
	\[ \Phi = \{ (\phi_1, \phi_2) | \phi_1: I \to \prod_{i \in I} J_i, \phi_i(i) \in J_i, \phi_2: I \to \N_{\ge 1}  \} \]
	equipped with the partial order
	\[ (\phi_1, \phi_2) \le (\phi_1', \phi_2') \iff \phi_1(i) \le \phi_1'(i), \phi_2(i) \le \phi_2'(i), \forall i \in I,  \]
	and 
	\[ M_\phi = \l \{ (x_i) \in \prod_{i \in I} M_{i, \phi_1(i)} | \frac{|x_i|}{\phi_2(i)} \text{ is bounded}  \r \}. \]
\end{lemma}
\begin{proof}
	The proof of Lemma \ref{lemma:product_born_sets} easily adapts.
\end{proof}

In the remaining of this section we introduce another kind of base changes functors which will be useful later on. Although we will use the notation $\ootimes_{\F_1}^{\sup}$ for this functor, we do not give to it a geometric interpretation and we will use it just for simplifying the notation that will be used later on in computations.

\begin{defn} \label{defn:sup_base_change}
	Let $M$ be a semi-normed set, we define the \emph{sup-base change over $\F_1$} of $M$ to a Banach ring $R$ as
	\[ M \ootimes_{\F_1}^{\sup} R \doteq {\prod_{m \in M}}^{\le 1} [R]_{|m|}, \]
	where the contracting product is computed in the category of semi-normed $R$-modules. 
\end{defn}

The definition of the sup-base change can be extended immediately to the categories $\bInd(\bBox_{\F_1})$ using just the functoriality of the category of $\bInd$-objects. Hence, for example, if $M = \underset{i \in I}{``\limind"} M_i$ we have that
\[ (\underset{i \in I}{``\limind"} M_i)\ootimes_{\F_1}^{\sup} R = \underset{i \in I}{``\limind"} M_i\ootimes_{\F_1}^{\sup} R \]
therefore $(\cdot)\ootimes_{\F_1}^{\sup} R$ commutes with filtered direct limits. It is also easy to verify that $(\cdot)\ootimes_{\F_1}^{\sup} R$ defines a functor between the categories of bornological objects because ${\underset{m \in M}\prod}^{\le 1}$ is a left-exact functor and hence it preserves monomorphisms. We state the following proposition for future references.

\begin{prop} \label{prop:sup_base_change}
	The sup-base change functor of bornological sets commutes with filtered colimits and with all limits of normed sets.
\end{prop}
\begin{proof}
	The first claim has been already discussed so far together with the fact that $(\cdot)\ootimes_{\F_1}^{\sup} R$ is a left-exact functor. It remains to show that $(\cdot)\ootimes_{\F_1}^{\sup} R$ commutes with products of normed sets. Let $\{M_i \}_{i \in I}$ be a family of normed sets. By Lemma \ref{lemma:product_born_sets} 
	\[ \prod_{i \in I} M_i \cong ``\limind_{\phi \in \Phi}" M_\phi = ``\limind_{\phi \in \Phi}" {\prod_{i \in I}}^{\le 1} [M_i]_{\frac{1}{\phi_2(i)}} \]
	where the symbols $\phi, \Phi, \phi_2$ ($\phi_1$ in this case is the identity map and we omit it) have the same meaning as in Lemma \ref{lemma:product_born_sets}. Then
	\[ (\prod_{i \in I} M_i) \ootimes_{\F_1}^{\sup} R \cong ``\limind_{\phi \in \Phi}" ({\prod_{i \in I}}^{\le 1} [M_i]_{\frac{1}{\phi_2(i)}} \ootimes_{\F_1}^{\sup} R ) \]
	because we have already shown that the sup-base change commutes with filtered direct limits. Thus
	\[ ``\limind_{\phi \in \Phi}" ({\prod_{i \in I}}^{\le 1} [M_i]_{\frac{1}{\phi_2(i)}} \ootimes_{\F_1}^{\sup} R ) \cong ``\limind_{\phi \in \Phi}" (\underset{m \in \underset{i \in I}{{\prod}^{\le 1}} [M_i]_{\frac{1}{\phi_2(i)}}}{{\prod}^{\le 1}} [R]_{|m|} ) \cong \]
	\[ \cong  ``\limind_{\phi \in \Phi}" (\underset{i \in I}{{\prod}^{\le 1}} \left [ \underset{m \in M_i}{{\prod}^{\le 1}} [R]_{|m|} \right ]_{\frac{1}{\phi_2(i)}} )  \cong \prod_{i \in I} \underset{m \in M_i}{{\prod}^{\le 1}} [R]_{|m|} \cong \prod_{i \in I} M_i \ootimes_{\F_1}^{\sup} R. \]
\end{proof}

It is not clear if the functor $(\cdot)\ootimes_{\F_1}^{\sup} R$ commutes with general products of bornological sets. Probably it will not. 

\begin{rmk} \label{rmk:nuclear}
	There exist some very special bornological sets $M$ for which $M \ootimes_{\F_1} R \cong M \ootimes_{\F_1}^{\sup} R$ for any Banach ring $R$. For example, this is the case when $M$ is a finite set and hence $M \ootimes_{\F_1} R$ is a finite free Banach module. Later on we will describe more interesting examples of such kind of spaces which should be thought as an analogue over $\F_1$ of the nuclear spaces of functional analysis. 
\end{rmk}

The last remark motivates the next definition.

\begin{defn} \label{defn:nuclear_set}
	Let $M$ be a bornological set such that $M \ootimes_{\F_1} R \cong M \wotimes_{\F_1}^{\sup} R$ for all Banach rings $R$, then $M$ is called \emph{nuclear}.
\end{defn}

Finally, it is clear that if $M$ is a (semi-)normed or bornological monoid then $M \ootimes_{\F_1}^{\sup} R$ is not in general a ring because one cannot make sense of the multiplication of elements. Therefore $(\cdot) \ootimes_{\F_1}^{\sup} R$ does induces a ``geometric" functor $\bComm(\bBox_{\F_1}) \to \bComm(\bBox_R)$ as the $\ell^1$-base changes discussed so far do.

\section{Simplicial bornological modules} \label{sec:simplicial_module}

For the purpose of using homological/homotopical methods on the categories introduced so far we need to discuss the categories of simplicial objects on them. Here we bound ourselves in discussing the basic properties of these categories. More detailed studies will come in the future works \cite{BeKr2} and \cite{BeKr3}.

\begin{defn} \label{defn:bornological_simplicial_sets}
We use the following notation (recall from Notation \ref{notation:box} that $\bTri$ means either a Banach ring or $\F_1$):
\begin{itemize}
\item the categories of covariant functors $\Delta^\op \to \bBox_\bTri$, \ie \emph{simplicial semi-normed (resp. normed, resp. Banach) modules} are denoted $\bSimp(\bBox_\bTri)$;
\item the categories of covariant functors $\Delta^\op \to \bBox_\bTri^{\le 1}$, \ie \emph{simplicial semi-normed (resp. normed, resp. Banach) modules with contracting morphisms} are denoted $\bSimp(\bBox_\bTri^{\le 1})$;
\item the categories of covariant functors $\Delta^\op \to \bInd(\bBox_\bTri)$, \ie \emph{simplicial ind-semi-normed (resp. ind-normed, resp. ind-Banach) modules} are denoted $\bSimp(\bInd(\bBox_\bTri))$;
\item the categories of covariant functors $\Delta^\op \to \bInd^m(\bBox_\bTri)$, \ie \emph{simplicial bornological (resp. separated bornological, resp. complete bornological) modules} are denoted $\bSimp(\bInd^m(\bBox_\bTri))$.
\end{itemize}
\end{defn}

In order to avoid smallness issues we fix two Grothendieck universes $U \subset V$, with strict inclusion, so that all the categories of normed, bornological and ind-normed sets are locally small in the universe $V$. 

We recall the following definition from \cite{CH}.

\begin{defn} \label{defn:projective_class}
A projective class on a category $\cC$ is a collection $\sP$ of objects of $\cC$ and a collection $\sE$ of maps in $\cC$ such that
\begin{enumerate}
\item $\sE$ is precisely the collection of maps $X \to Y$ such that 
\[ \Hom_{\cC}(P, X) \to \Hom_{\cC}(P, Y) \]
is surjective for all $X, Y \in \cC$ and $P \in \sP$, and $\sP$ is precisely the collection of objects for which that happens;
\item for each $X \in \cC$ there is a map in $(P \to X ) \in \sE$ for which $P \in \sP$.
\end{enumerate}
\end{defn}

\begin{prop} \label{prop:projective_class}
The collection of all objects and strict epimorphisms is a projective class for $\bBox_{\F_1}$.
\end{prop}
\begin{proof}
Let $f: X \to Y$ be a morphism of $\bBox_{\F_1}$ such that
\[ \Hom_\cC(P, X) \stackrel{f_*}{\to} \Hom_\cC(P, Y) \]
is surjective for all $P \in \bBox_{\F_1}$. In particular, considering $P = Y$ we get that there exists a morphism $g: Y \to X$ such that $f \circ g = \Id_Y$, hence $f$ has a section and in particular it is a strict epimorphism. On the other hand, if $f$ is a strict epimorphism given any $g: P \to Y$ we can define a map $h: P \to X$ such that $f \circ h = g$ by defining $h(p) = x \in f^{-1}(g(p))$, for any such possible choice of pre-images. The fact that $h$ is always a bounded morphism is ensured by the fact that $f$ is a strict epimorphism. The second condition of the definition of projective class is trivial in this case.
\end{proof}

\begin{cor}
The collection of all objects and strict epimorphisms is a projective class also for $\bBox_{\F_1}^{\le 1}$, $\bInd(\bBox_{\F_1})$ and $\bInd^m(\bBox_{\F_1})$.
\end{cor}
\begin{proof}
The same argument of Proposition \ref{prop:projective_class} can be applied to the category $\bBox_{\F_1}^{\le 1}$ by replacing bounded maps with the contracting ones. The case of $\bInd(\bBox_{\F_1})$ and $\bInd^m(\bBox_{\F_1})$ are settled easily by noticing that in $\bInd(\bBox_{\F_1})$ strict epimorphisms can always be represented by a map of inductive systems which is a strict epimorphism of bounded sets for each term of the system and for $\bInd^m(\bBox_{\F_1})$ follows easily from the fact that $\bInd^m(\bBox_{\F_1})$ is a concrete category. 
Hence, we are reduced to the case of Proposition \ref{prop:projective_class}.
\end{proof}

We cannot use the categories $\bSimp(\bBox_{\F_1})$ for homotopical algebra comfortably because they are not complete nor cocomplete. But the other categories we are discussing are, and for them we can prove that they admit a nice model structure.

\begin{thm} \label{thm:existence_model_structure}
The categories of simplicial objects over $\bBox_{\F_1}^{\le 1}$, $\bInd(\bBox_{\F_1})$ and $\bInd^m(\bBox_{\F_1})$ admit a model structure which endow them of a structure of simplicial model category. Moreover, this model structure is cofibrantly generated.
\end{thm}
\begin{proof}
It is enough to apply Theorem 6.3 of \cite{CH}. The model structure deduced in this way is determined as follows. A morphism $f: X^\bullet \to Y^\bullet$ of simplicial objects is an equivalence if for each projective object $P$ (in our case any object in the considered category) the map of simplicial sets 
\[ \Hom(P, X^\bullet) \to \Hom(P, Y^\bullet) \]
is a weak equivalence of simplicial sets, where $\Hom(P, X^\bullet)$ is the simplicial set obtained by applying the functor $\Hom(P, -)$ at each degree. The (trivial) fibrations are defined to be the morphisms such that for each projective $P$ 
\[ \Hom(P, X^\bullet) \to \Hom(P, Y^\bullet) \]
is a (trivial) fibration of simplicial sets. Cofibrations are identified by the left lifting property with respect to the trivial fibrations.

We notice also that we can apply Theorem 6.3 of \cite{CH} because our categories satisfy condition (**) stated before the theorem in ibid. This is the case because $\bBox_{\F_1}^{\le 1}$ is small by our assumptions and both $\bInd(\bBox_{\F_1})$ and $\bInd^m(\bBox_{\F_1})$ are generated by filtered colimits of elements of $\bBox_{\F_1}$ which form a small subcategory of compact objects.
\end{proof}

We introduce some notation for describing a generating set of cofibrations for the models structure given by Theorem \ref{thm:existence_model_structure}. We denote by $\Delta_n$ the standard simplex, $\ol{\Delta}_n$ its boundary and $\La_{k, n}$ the $k$-th horn of the $n$-simplex. For any $X \in \bC$, with $\bC$ having finite coproducts, we define the simplicial object
\[ (X \ \ootimes \ \Delta_n)^j = \coprod_{\sigma \in (\Delta_n)^j} X, \ \ j \le n  \]
with the obvious maps. In the same way we define $X \ \ootimes \ \ol{\Delta}_n$ and $\La_{k,n}$.

\begin{cor} \label{cor:existence_model_structure}
The model structures of $\bBox_{\F_1}^{\le 1}$, $\bInd(\bBox_{\F_1})$ and $\bInd^m(\bBox_{\F_1})$ can be generated by the following set of cofibrations
\[ I = \{ P \ \ootimes \ \ol{\Delta}_n \to P \ \ootimes \ \Delta_n | n \ge 0 \} \]
\[ J = \{ P \ \ootimes \ \La_{k, n} \to P \ \ootimes \  \Delta_n | n > 0, 0 \le k \le n \} \]
for all objects $P \in \bBox_{\F_1}$.
\end{cor}
\begin{proof}
This corollary is still part of Theorem 6.3 of \cite{CH}. Notice that since the categories $\bBox_{\F_1}$ are small the classes $I$ and $J$ are actually sets.
\end{proof}

Notice that we cannot apply Theorem 6.3 of \cite{CH} directly to the category $\bBox_{\F_1}$ because it is not complete and cocomplete. We denote by $\bHo(\bBox_{\F_1}^{\le 1})$, $\bHo(\bInd(\bBox_{\F_1}))$ and $\bHo(\bInd^m(\bBox_{\F_1}))$ the homotopy categories relative to the model structures introduced so far. 

\begin{prop} \label{prop:existence_model_structure_banach_ring}
Let $R$ be a Banach ring. The categories of simplicial objects over $\bBox_R^{\le 1}$, $\bInd(\bBox_R)$ and $\bInd^m(\bBox_R)$ admit a model structure which endow them of a structure of simplicial model category. Moreover, this model structure is cofibrantly generated and throught the Dold-Kahn correspondence two simplicial objects are weak equivalent if and only if their associated complexes are strictly quasi-isomorphic.
\end{prop}
\begin{proof}
This proposition is essentially proved in \cite{CH}, Corollary 6.4. The class of cofibrations, fibrations and weak-equivalences can be described as in Theorem \ref{thm:existence_model_structure}. Notice that we can also apply Corollary 6.4 of \cite{CH} to the abelian envelop of each of the categories $\bBox_R^{\le 1}$, $\bInd(\bBox_R)$ and $\bInd^m(\bBox_R)$. We devote the rest of the proof in giving more details about the argument discussed so far. 

We discuss only the case of $\bInd(\bBox_R)$. It is clear that the class of projective objects (in the sense of quasi-abelian category theory, cf. \cite{BaBe} Section 2.2) and strict epimorphisms define a projective class in $\bInd(\bBox_R)$, in the sense of Definition \ref{defn:projective_class}. Let us denote with $LH(\bInd(\bBox_R))$ the abelian envelop of $\bInd(\bBox_R)$ (recall that $LH(\bInd(\bBox_R))$ and $\bInd(\bBox_R)$ are derived equivalent). We can define on $\bSimp(\bInd(\bBox_R))$ and $\bSimp(LH(\bInd(\bBox_R)))$ the model structures given as in Theorem \ref{thm:existence_model_structure}: a (trivial) fibration $f: X^\bullet \to Y^\bullet$ is a map such that
\[ \Hom(P, X^\bullet) \to \Hom(P, Y^\bullet) \]
is a (trivial) fibration of simplicial sets and a weak-equivalence is a map such that
\[ \Hom(P, X^\bullet) \to \Hom(P, Y^\bullet) \]
is a weak equivalence of simplicial sets for all projectives $P$. Cofibrations are defined by the left lifting property with respect to trivial fibrations. Now, since the projective classes in $\bInd(\bBox_R)$ and in $LH(\bInd(\bBox_R))$ are the same, it is easy to see that the adjunction
\[ \iota: \bSimp(\bInd(\bBox_R)) \rightleftarrows \bSimp(LH(\bInd(\bBox_R))): R \]
is not only a Quillen adjunction but even a Quillen equivalence. Indeed, it is immediate to check that $\iota$ and $R$ preserves cofibrations, fibrations and weak-equivalences. This proves the claim.
\end{proof}

\begin{rmk}
Proposition \ref{prop:existence_model_structure_banach_ring} implies that the notion of derived category and derived functors we will discuss in this paper is coincides with the notion discussed by Schneiders in its study of quasi-abelian categories (\cf \cite{ScQA}) and used by the authors in \cite{BeKr}, \cite{BaBe}, \cite{BaBeKr}.
\end{rmk}

The next theorem generalizes the derived equivalences between the category of bornological vector spaces over a valued field and the category of $\bInd$-Banach spaces (see \cite{PrSc}, Proposition 3.10, Proposition 4.12 (c), Proposition 5.16 (b)) to a derived equivalence between the categories $\bInd^m(\bBox_\bTri)$ and $\bInd(\bBox_\bTri)$.

\begin{lemma} \label{lemma:essmono_equiv_ind}
The projective class generated by the essentially monomorphic objects and strict epimorphisms in $\bSimp(\bInd(\bBox_\bTri))$ is equal to the projective class generated by all projective objects.
\end{lemma}
\begin{proof}
We can apply the non-abelian version of Lemma 1.5 of \cite{CH}, whose proof is identical to the one given in loc. cit. Therefore, it is enough to check that every object of $\bInd(\bBox_\bTri)$ is a retract of coproducts of projective objects of $\bInd^m(\bBox_\bTri)$.

We divide the proof in two case: $\bTri = \F_1$ and $\bTri = R$. In the first case we notice that that any strict epimorphism of objects of $\bInd(\bBox_{\F_1})$ is split because it is a cokernel of a pair of morphisms and cokernels of a morphism of ind-objects can be computed objectwise. The fact that strict epimorphisms of $\bBox_{\F_1}$ split has been proved in the proof of Proposition \ref{prop:projective_class}. Therefore, it is enough to check that for each object of $\bInd(\bBox_{\F_1})$ there exists an essentially monomorphic object with a strict epimorphism to it. So, if $\underset{i \in I}{``\limind"} X_i \in \bInd(\bBox_{\F_1})$ we can consider the object $\underset{i \in I}{\coprod} X_i \in \bInd(\bBox_{\F_1})$ which comes with a canonical strict epimorphism $\underset{i \in I}{\coprod} X_i \to \underset{i \in I}{``\limind"} X_i$. $\underset{i \in I}{\coprod} X_i$ is obviously essentially monomorphic.

The same statement for $\bInd(\bBox_R)$ can be deduced from Lemma 3.29 of \cite{BaBe} and the explicit description of projectives given there.
\end{proof}

\begin{thm} \label{thm:essmono_equiv_ind}
The categories $\bSimp(\bInd(\bBox_\bTri))$ and $\bSimp(\bInd^m(\bBox_\bTri))$ endowed with the model structures discussed so far are Quillen equivalent.
\end{thm}
\begin{proof}
Consider the Quillen adjunction 
\[ \iota: \bSimp(\bInd^m(\bBox_\bTri)) \rightleftarrows \bSimp(\bInd(\bBox_\bTri)): R \]
where $\iota$ is the inclusion functor and $R$ is its adjoint. Notice that $\iota$ is the right adjoint and therefore for every $X \in \bSimp(\bInd(\bBox_\bTri))$ we have a canonical map $X \to \iota(R(X))$. 

Notice also that $(f: X \to Y) \in \bSimp(\bInd^m(\bBox_\bTri))$ is a weak equivalence if and only if $\iota(f)$ is a weak equivalence because by Lemma \ref{lemma:essmono_equiv_ind} we can check the condition of being a weak equivalence in $\bSimp(\bInd(\bBox_\bTri))$ on essentially monomorphic objects. 

Because of that, it is enough to check that the canonical morphism $X \to \iota(R(X))$ is a weak equivalence. It is enough to pick any object $Z \in \bInd^m(\bBox_\bTri)$ and check that
\[ \Hom(\iota(Z), X) \to \Hom(\iota(Z), \iota (R(X))) \]
is a weak equivalence. This follows from the adjunction
\[ \Hom(\iota(Z), X) \cong \Hom(Z, R(X)) \cong \Hom(\iota(Z), \iota (R(X))) \]
where the second isomorphism of simplicial sets is given by the fact that $\iota$ is fully faithful.
\end{proof}

In order to define derived analytic spaces we need to introduce the categories of simplicial commutative monoids over the simplicial model categories introduced so far.

\begin{defn} \label{defn:bornological_simplicial_monoids}
We use the following notation:
\begin{itemize}
\item $\bsComm(\bBox_\bTri)$ denotes the category of \emph{simplicial semi-normed (resp. normed, resp. Banach) monoids}, \ie the category of monoids of $\bSimp(\bBox_\bTri)$;
\item $\bsComm(\bBox_\bTri^{\le 1})$ denotes the category of \emph{simplicial semi-normed (resp. normed, resp. Banach) monoids with contracting morphisms}, \ie the category of monoids of $\bSimp(\bBox_\bTri^{\le 1})$;
\item $\bsComm(\bInd(\bBox_\bTri))$ denotes the category of \emph{simplicial ind-semi-normed (resp. ind-normed, resp. ind-Banach) monoids}, \ie the category of monoids of $\bSimp(\bInd(\bBox_\bTri))$;
\item $\bsComm(\bInd^m(\bBox_\bTri))$ denotes the category of \emph{simplicial bornological (resp. separated bornological, resp. complete bornological) monoids}, \ie the category of monoids of $\bSimp(\bInd^m(\bBox_\bTri))$.
\end{itemize}
\end{defn}

In the rest of this section we will show that the categories of simplicial commutative monoids introduced in Definition \ref{defn:bornological_simplicial_monoids} inherit a structure of model category from their base categories which has good properties.

\begin{lemma} \label{lemma:fib_we}
Let $f: X^\bullet \to Y^\bullet$ be a fibration (resp. weak equivalence) in $\bSimp(\bBox_\bTri^{\le 1})$ or $\bSimp(\bInd^m(\bBox_\bTri))$ then $f$ is a fibration (resp. weak equivalence) of the underlying simplicial sets.
\end{lemma}
\begin{proof}
Consider the map of simplicial sets
\[ f^*: \Hom(\{\ast_1\}, X^\bullet) \to \Hom(\{\ast_1\}, Y^\bullet) \]
induced by $f$. The functor $\Hom(\{\ast_1\}, \cdot)$ is isomorphic to the forgetful functor and by hypothesis if $f$ is a fibration (resp. weak equivalence) then $f^*$ is a fibration (resp. weak equivalence). 
\end{proof}

\begin{lemma} \label{lemma:cofib}
Let $f: P \ \ootimes \ \ol{\Delta}_n \to  P \ \ootimes \ \Delta_n$ be a generating cofibration of the model structure of $\bSimp(\bInd^m(\bBox_\bTri))$ or $\bSimp(\bBox_\bTri^{\le 1})$, then $f$ is a cofibration of the underlying simplicial sets.
\end{lemma}
\begin{proof}
The underlying simplicial set of $P \ \ootimes \ \ol{\Delta}_n$ is just the coproduct of the simplicial set $\ol{\Delta}_n$ indexed by $X$, because for each $i$
\[ (X \ \ootimes \ \Delta_n)^i = \coprod_{\sigma \in (\Delta_n)^i} X = \coprod_{x \in X} (\Delta_n)^i  \]
and, as a map of underlying simplicial sets 
\[ f = \coprod_{x \in X} g : P \ \ootimes \ \ol{\Delta}_n \to  P \ \ootimes \ \Delta_n \]
where
\[ g: \ol{\Delta}_n \to  \Delta_n \]
is the inclusion of the boundary of $\Delta_n$. Hence it is a cofibration of simplicial sets.
\end{proof}

\begin{thm} \label{thm:model_monoidal}
The categories $\bSimp(\bBox_\bTri^{\le 1})$, $\bSimp(\bInd(\bBox_\bTri))$ and $\bSimp(\bInd^m(\bBox_\bTri))$ are symmetric monoidal model categories.
\end{thm}
\begin{proof}
We have to check that these categories satisfy the monoidal axiom and the commutative monoidal axiom. We work out the case of $\bSimp(\bInd^m(\bBox_\bTri))$ because it is the one we care most. For $\bSimp(\bInd(\bBox_\bTri))$ one notice that we proved that it is Quillen equivalent to $\bSimp(\bInd^m(\bBox_\bTri))$ and for $\bSimp(\bBox_\bTri^{\le 1})$ we notice that our proof can be adapted for these categories.

The first axiom to check is the unit axiom, which requires that for each $X \in \bSimp(\bInd^m(\bBox_\bTri))$ and $I \in \bSimp(\bInd^m(\bBox_\bTri))$ the identity of the tensor product one has that
\[ Q (I) \ootimes_{\bTri} X \to I \ootimes_{\bTri} X \]
is a weak equivalence, where $Q(I)$ is the cofibrant replacement of $I$. This is true because in all our categories $I$ is cofibrant, hence $Q(I)$ is actually isomorphic to $I$. Then, we need to check the pushout-product axiom. We can check it for the generating cofibrations (which are described in Corollary \ref{cor:existence_model_structure}) for which it is stated as follows: Let $f: X = P \ \ootimes \ \ol{\Delta}_n \to Y = P \ \ootimes \ \Delta_n$ and $g: X' = Q \ \ootimes \ \ol{\Delta}_n \to Y' = Q \ \ootimes \ \Delta_n$ be two generating cofibrations then 
\[ h: (Y \ootimes_\bTri X') \coprod_{X \ootimes_\bTri X'} (X \ootimes_\bTri Y') \to Y \ootimes_\bTri Y' \]
is a cofibration. Notice that $f$ and $f'$ are cofibrations of the underlying simplicial sets and that every trivial fibration of $\bSimp(\bInd^m(\bBox_\bTri))$ is a trivial fibration of the underlying simplicial sets. Hence the diagram
\[
\begin{tikzpicture}
\matrix(m)[matrix of math nodes,
row sep=2.6em, column sep=2.8em,
text height=1.5ex, text depth=0.25ex]
{ (Y \ootimes_\bTri X') \coprod_{X \ootimes_\bTri X'} (X \ootimes_\bTri Y')  & E   \\
  Y \ootimes_\bTri Y'  & F  \\};
\path[->,font=\scriptsize]
(m-1-1) edge node[auto] {} (m-1-2);
\path[->,font=\scriptsize]
(m-1-1) edge node[auto] {$h$} (m-2-1);
\path[->,font=\scriptsize]
(m-1-2) edge node[auto] {} (m-2-2);
\path[->,font=\scriptsize]
(m-2-1) edge node[auto] {} (m-2-2);
\path[->,dashed, font=\scriptsize]
(m-2-1) edge node[auto] {} (m-1-2);
\end{tikzpicture}
\]
can be completed with the dashed arrow as a diagram of simplicial sets. This follows by the fact that $f$ and $f'$ are cofibrations of the underlying simplicial sets which implies that $h$ is a cofibration. So, since by Lemma \ref{lemma:fib_we} trivial fibrations of bornological simplicial sets induce trivial fibrations of the underlying simplicial sets, the dashed arrow always exists as a map of simplicial sets. 
But since $(Y \ootimes_\bTri X') \coprod_{X \ootimes_\bTri X'} (Y' \ootimes_\bTri X)  \to E$ is bounded and $h$ is a strict morphism, then the diagonal map is always bounded. To see that $h$ is a strict morphism we notice that in the pushout diagram
\[
\begin{tikzpicture}
\matrix(m)[matrix of math nodes,
row sep=2.6em, column sep=2.8em,
text height=1.5ex, text depth=0.25ex]
{ X \ootimes_\bTri X'  & X \ootimes_\bTri Y'   \\
  Y \ootimes_\bTri X'  & (Y \ootimes_\bTri X') \coprod_{X \ootimes_\bTri X'} (Y' \ootimes_\bTri X)  \\};
\path[->,font=\scriptsize]
(m-1-1) edge node[auto] {} (m-1-2);
\path[->,font=\scriptsize]
(m-1-1) edge node[auto] {} (m-2-1);
\path[->,font=\scriptsize]
(m-1-2) edge node[auto] {} (m-2-2);
\path[->,font=\scriptsize]
(m-2-1) edge node[auto] {} (m-2-2);
\end{tikzpicture}
\]
to upper horizontal and the left vertical maps are obviously strict monomorphisms which implies that also the other two are, because pushouts preserve strict monos. Hence in the commutative diagram
\[
\begin{tikzpicture}
\matrix(m)[matrix of math nodes,
row sep=2.6em, column sep=2.8em,
text height=1.5ex, text depth=0.25ex]
{ X \ootimes_\bTri Y'  &   \\
    & (Y \ootimes_\bTri X') \coprod_{X \ootimes_\bTri X'} (Y' \ootimes_\bTri X)   &   X' \ootimes_\bTri Y' \\
    Y \ootimes_\bTri X'  \\};
\path[->,font=\scriptsize]
(m-1-1) edge node[auto] {} (m-2-2);
\path[->,font=\scriptsize]
(m-3-1) edge node[auto] {} (m-2-2);
\path[->,font=\scriptsize] 
(m-1-1) edge node[auto] {} (m-2-3);
\path[->,font=\scriptsize]
(m-3-1) edge node[auto] {} (m-2-3);
\path[->,font=\scriptsize]
(m-2-2) edge node[auto] {$h$} (m-2-3);
\end{tikzpicture}
\]
all maps are known to be strict but $h$ which is then a strict morphism too.

An analogue argument permits to prove that $f^{\Box^n}/S_n$, the $n$-fold pushout product of $f$ with itself modulo the action of the symmetric group, is always a cofibration. Proving the symmetric monoidal axiom and the theorem.
\end{proof}

\begin{cor} 
The categories $\bsComm(\bBox_\bTri^{\le 1})$, $\bsComm(\bInd(\bBox_\bTri))$ and $\bsComm(\bInd^m(\bBox_\bTri))$ inherit a structure of combinatorical model structure from the model structures defined so far on the categories of simplicial modules.
\end{cor}
\begin{proof}
Theorem \ref{thm:model_monoidal} implies that we can define model structures on the categories $\bsComm(\bBox_\bTri^{\le 1})$, $\bsComm(\bInd(\bBox_\bTri))$ and $\bsComm(\bInd^m(\bBox_\bTri))$ by defining fibrations (resp. weak equivalences) to be the maps which are fibrations (resp. weak equivalences) of the underlying simplicial objects of $\bSimp(\bBox_\bTri^{\le 1})$, $\bSimp(\bInd(\bBox_\bTri))$ and $\bSimp(\bInd^m(\bBox_\bTri))$, respectively, as a consequence of Theorem 3.2 of \cite{Whi}.
\end{proof}

\begin{rmk}
It is a standard fact that the categories $\bsComm(\bBox_\bTri^{\le 1})$, $\bsComm(\bInd(\bBox_\bTri))$ and $\bsComm(\bInd^m(\bBox_\bTri))$ with the models structure defined in Theorem \ref{thm:model_monoidal} can be interpreted as the categories of $\E_\infty$-algebras over $\bSimp(\bBox_\bTri^{\le 1})$, $\bSimp(\bInd(\bBox_\bTri))$ and $\bSimp(\bInd^m(\bBox_\bTri))$ respectively.
\end{rmk}

Let $A$ be a simplicial commutative object of one of the categories introduced so far, then we can associate to $A$ the categoy of its modules, denoted $\bMod(A)$, and the category of $A$-algebras, denoted $\bComm(A)$.

\begin{prop} 
With the notation introduced so far, both $\bMod(A)$ and $\bComm(A)$ inherit a structure of combinatorical model categories for which $\bMod(A)$ is a simplicial model category.
\end{prop}
\begin{proof}
Again this is a consequence of Theorem \ref{thm:model_monoidal}.
\end{proof}

\begin{rmk} \label{rmk:infty_cat}
Later on, we will use also the language of $\infty$-categories for discussing derived analytic spaces over $\F_1$ and mainly for defining analytic stacks over $\F_1$. For all the purposes of this work the language of model categories and of $\infty$-categories are to be considered equivalent and a model category has to be considered just as a presentation of an $\infty$-category. In particular, notice that the model categories introduced so far define $\infty$-categories.

We will switch between these two languages freely, using the one that fits better in our discussion at each time.
\end{rmk}

\section{Bornological rings} \label{sec:bornological_rings}

Up to now we discussed bornological modules over $\F_1$ or over Banch rings but we will need to use more general base rings. In this section we introduce the notion of (complete) bornological ring and its spectrum and we give some fundamental examples which will be needed in the applications. In the subsequent sections it will be clear why we need to introduce this notion and we cannot bound ourselves to work only with Banach rings (cf. Theorem \ref{thm:main}).

\begin{defn} \label{defn:bornological_ring}
A (resp. \emph{separated}, resp. \emph{complete}) \emph{bornological ring} is an object of $\bComm(\bInd^m(\bSN_\Z))$ (resp. $\bComm(\bInd^m(\bNr_\Z))$, resp. $\bComm(\bInd^m(\bBan_\Z))$)
\end{defn}

We will mainly work with complete bornological rings therefore, if not differently specified, all our bornological rings are supposed to be complete.

\begin{defn} \label{defn:spectrum_bornological_ring}
Let $R$ be a bornological ring. The \emph{spectrum} of $R$ is defined to be the topological space $\cM(R)$ of bounded multiplicative semi-norms on $R$ equipped with the weak topology.
\end{defn}

\begin{rmk}
Notice this notion of spectrum is compatible with the notion of spectrum of a Banach ring given by Berkovich in \cite{Ber1990}, cf. Chapter 1.
\end{rmk}

The next example shows the main examples of bornological rings we are interested in.

\begin{exa} \label{exa:born_rings}
\begin{enumerate}
\item We briefly recall some analytic properties of the ring of integers. The ring of integers with the Euclidean absolute value $(\Z, |\cdot|_\infty)$ is a Banach ring. As explained in the first chapter of \cite{Ber1990} its spectrum, denoted $\cM(\Z)$, is a pro-finite graph as depicted in the following picture.
\begin{figure}[h]
    \begin{tikzpicture}[scale=2.5,cap=round,>=latex]

         \draw[gray] (0cm,0cm) -- (90:1cm);
         \filldraw[black] (90:1cm) circle(0.4pt);
         
         \draw[gray] (0cm,0cm) -- (210:1.3cm);
         \filldraw[black] (210:1.3cm) circle(0.4pt);
         
         \draw[gray] (0cm,0cm) -- (240:1.1cm);
         \filldraw[black] (240:1.1cm) circle(0.4pt);
         
         \draw[gray] (0cm,0cm) -- (270:1cm);
         \filldraw[black] (270:1cm) circle(0.4pt);
         
         \draw[gray, dashed] (0cm,0cm) -- (300:1.1cm);
         \filldraw[black] (300:1.1cm) circle(0.4pt);
         
         \draw[gray] (0cm,0cm) -- (330:1.3cm);
         \filldraw[black] (330:1.3cm) circle(0.4pt);

        \filldraw[black] (0cm:0cm) circle(0.4pt);

       \draw (0cm,1cm) node[above=0pt] {$|\cdot|_\infty$}
             (-1.2cm,-0.7cm)  node[below=1pt] {$2$}
             (1.2cm,-0.7cm)  node[below=1pt] {$p$}
             (0cm,-1cm)  node[below=1pt] {$5$}
             (-0.7cm,-0.9cm)  node[below=1pt] {$3$}
             (0.7cm,-0.9cm)  node[below=1pt] {}
             (0.15cm, 0cm) node[above] {$|\cdot|_0$};
    \end{tikzpicture}
    
    \caption{$\cM(\Z)$}\label{fig:1}

\end{figure}
It is a tree with one node of infinite valence, corresponding to the trivial norm from which start one branch for each prime number and one for the Archimedean place of $\Z$. The topology is equivalent to the topology of the pro-finite tree. In the same way one can describe the spectrum of the ring of integers of a number field $\cO_K \subset K$ when $\cO_K$ is equipped with the norm 
\[  \|x\| = \max_{\sigma: K \rhook \C} |x|_\sigma \ \, x \in \cO_K. \]
The spectrum $\cM((\cO_K, \|\cdot\|))$ is a tree similar to $\cM(\Z)$ with a branch for each place of $\cO_K$ starting from the central node which corresponds to the trivial norm.

\item Berkovich in \cite{BerMix} defined the Banach ring $(\C, \|\cdot\|)$ where
\[ \| x \| \doteq \max \{ |x|_0, |x|_\infty \}. \]
In a more functorial way, one can think about this ring as 
\[ (\C, \|\cdot\|) \cong \Im(\C \stackrel{\Delta}{\to} (\C, |\cdot|_0) \times (\C, |\cdot|_\infty) ) \]
where $\Delta$ is the diagonal map. Notice that $\cM((\C, \|\cdot\|)) \cong [0, 1]$ where $\epsilon \in [0, 1]$ is identified with $|\cdot|_\infty^\epsilon$. The same construction can be done for $\R$ in place of $\C$ and for any arbitrary sub-interval of $[r_1, r_2] \subset [0, 1]$ yielding a Banach ring $(\R, \|\cdot\|_{r_1, r_2})$ whose spectrum can be identified with the interval $[r_1, r_2]$. The concept of bornological ring permits to generalize this construction to open intervals (or half-open) $(r_1, r_2) \subset [0, 1]$ by considering the ring
\[ \R_{(r_1, r_2)} = \limpro_{r_1 < r < r' < r_2} (\R, \|\cdot\|_{r, r'}). \]
Clearly $\cM(\R_{(r_1, r_2)}) \cong (r_1, r_2)$. Notice that the spaces $(\R_{(r_1, r_2)}, \cM(\R_{(r_1, r_2)}))$ can be canonically identified as analytic subspaces (in the sense explained in the next section) of $\cM(\Z)$ and also $((\R, \|\cdot\|_{r_1, r_2}), \cM((\R, \|\cdot\|_{r_1, r_2})))$ if $r_1 > 0$, as Figure \ref{fig:2} shows.

\begin{figure}[h] \label{fig:2}
    \begin{tikzpicture}[scale=2.5,cap=round,>=latex]
         
         \draw[gray] (-0.75cm, 0.75cm) -- (-0.75cm,0.12cm);
         \draw[gray, dashed] (-0.75cm, 0.75cm) -- (90:0.75cm);
         \draw[gray, dashed] (-0.75cm, 0.12cm) -- (90:0.12cm);
        \draw[black] (-0.75cm,0.75cm) circle(0.4pt);
        \draw[black] (-0.75cm,0.12cm) circle(0.4pt);

         \draw[gray] (-1.24cm, -0.45cm) -- (-0.81cm, -0.21cm);
         \filldraw[black] (-1.24cm, -0.45cm) circle (0.4pt);
         \draw[black] (-0.81cm, -0.21cm) circle (0.4pt);
         \draw[gray, dashed] (-1.13cm, -0.65cm) -- (-1.24cm, -0.45cm);
         \draw[gray, dashed] (-0.7cm, -0.4cm) -- (-0.81cm, -0.21cm);    
         
         \draw[gray] (0.97cm, -0.33cm) -- (0.44cm, -0.03cm);
         \draw[black] (0.97cm, -0.33cm) circle (0.4pt);
         \draw[black] (0.44cm, -0.03cm) circle (0.4pt);
         \draw[gray, dashed] (0.87cm, -0.5cm) -- (0.97cm, -0.33cm);
         \draw[gray, dashed] (0.34cm, -0.2cm) -- (0.44cm, -0.03cm);       
         
         \draw[gray] (0cm,0cm) -- (90:1cm);
         \filldraw[black] (90:1cm) circle(0.4pt);
         
         \draw[gray] (0cm,0cm) -- (210:1.3cm);
         \filldraw[black] (210:1.3cm) circle(0.4pt);
         
         \draw[gray] (0cm,0cm) -- (240:1.1cm);
         \filldraw[black] (240:1.1cm) circle(0.4pt);
         
         \draw[gray] (0cm,0cm) -- (270:1cm);
         \filldraw[black] (270:1cm) circle(0.4pt);
         
         \draw[gray, dashed] (0cm,0cm) -- (300:1.1cm);
         \filldraw[black] (300:1.1cm) circle(0.4pt);
         
         \draw[gray] (0cm,0cm) -- (330:1.3cm);
         \filldraw[black] (330:1.3cm) circle(0.4pt);

        \filldraw[black] (0cm:0cm) circle(0.4pt);

       \draw (0cm,1cm) node[above=0pt] {$|\cdot|_\infty$}
             (-1.2cm,-0.7cm)  node[below=1pt] {$2$}
             (1.2cm,-0.7cm)  node[below=1pt] {$p$}
             (0cm,-1cm)  node[below=1pt] {$5$}
             (-0.7cm,-0.9cm)  node[below=1pt] {$3$}
             (0.7cm,-0.9cm)  node[below=1pt] {}
             (0.15cm, 0cm) node[above] {$|\cdot|_0$}
             (-1.2cm, 0.3cm) node[above] {$\cM(\R_{(r_1, r_2)})$}
             (-1.43cm, -0.3cm) node[above] {$\cM((\Z_2)_{[0, r_2)})$}
             (1cm, -0.15cm) node[above] {$\cM((\Q_p)_{(r_1, r_2)})$};
    \end{tikzpicture}
    
    \caption{$\cM(\R_{(r_1, r_2)})$, $\cM((\Q_p){(r_1, r_2)})$ and $\cM((\Z_2)_{[0, r_2)})$ as a subspaces of $\cM(\Z)$}

\end{figure}
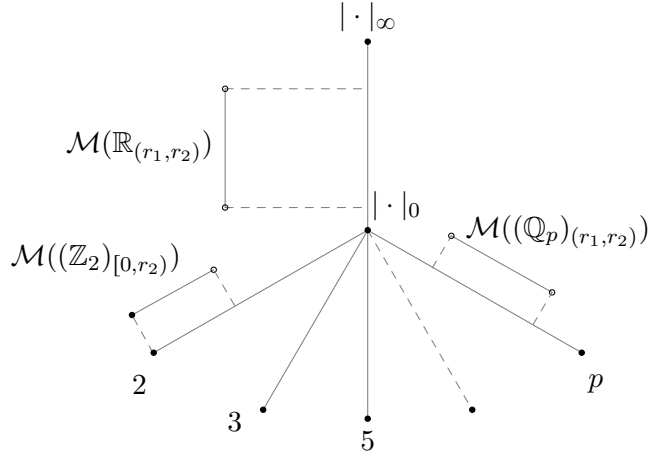

\item The same construction of previous example can be worked out for the $p$-adic branches of $\cM(\Z)$. Briefly, we can define
\[ (\Q_p, \|\cdot\|_{r_1, r_2}) \cong \Im(\Q_p \stackrel{\Delta}{\to} (\Q_p, |\cdot|_{r_1}) \times (\Q_p, |\cdot|_{r_2}) ), \]
explicitly $\Q_p$ is equipped with the norm
\[ \|x\|_{r_1, r_2} = \max \{ |x|_p^{r_1}, |x|_p^{r_2} \} , \ \ x \in \Q_p \]
where now $[r_1, r_2] \subset [0, \infty]$. And then
\[ (\Q_p)_{(r_1, r_2)} = \limpro_{r_1 < r < r' < r_2} (\Q_p, \|\cdot\|_{r, r'}). \]
and again the analytic spaces defined by these bornological rings can be canonically identified with sub-spaces of $\cM(\Z)$ as shown in Figure \ref{fig:2}.

\item In the same vein one can define also $(\Z_p)_{(r_1, r_2)}$, but in this case one has that $|\cdot|_p^{r_1} \le |\cdot|_p^{r_2}$ if $r_1 \le r_2$, hence $(\Z_p)_{(r_1, r_2)} \cong (\Z_p)_{[0, r_2)}$. Thus, $\Z_p$ determines open and closed neighborhoods of the maximal point $\{p\}$ in $\cM(\Z)$.

\item Building on the last example, we can give an even more geometric interpretation of the neighborhoods of the form $(\Z_p)_{[0, r)}$ of $\{p\}$ in $\cM(\Z)$. One can introduce the algebras of convergent analytic functions on $\Z$ as the algebra $S_\Z(\rho)$, for any radius $\rho$, as in equation (\ref{eqn:S}). We refer to the appendix and to Section 6 of \cite{BaBe} for general properties of these algebras. In this example we notice that we can define Weierstrass localizations of these algebras in the usual way: For any $f \in S_\Z(\rho)$ we define
\[ S_\Z(\rho) \lt \sigma^{-1} f \gt = \frac{S_\Z(\rho) \lt \sigma^{-1} T \gt}{(T - f)} \]
for $\sigma > 0$, which has the usual geometric interpretation as the subspace of $\cM(\Z)$ where $|f| \le \sigma$. It is easy to show that in the particular case when $f = p$ and $\sigma = r^{-1} p$ one gets
\[ S_\Z(\rho) \lt \sigma^{-1} f \gt = (\Z_p)_{[0, r]}, \]
hence $(\Z_p)_{[0, r)}$ can be seen as a Weierstrass localization of the convergent power-series over $\Z$.
\end{enumerate}
\end{exa}

We will see later on how these segments relate to $p$-adic Hodge theory, now we introduce the base change functor from $\F_1$ for bornological rings.

\begin{defn} \label{defn:base_change_born}
Let $R \cong \underset{i \in I}{``\limind"} R_i$ be a (complete) bornological ring and $M$ a normed set. Then we define the \emph{base change} of $M$ to $R$ as
\[ M \otimes_{\F_1} R \doteq  ``\limind_{i \in I}" {\coprod_{x \in M}}^{\le 1} [R_i]_{|x|_M} \]
and its completed version, using the same formula
\[ M \wotimes_{\F_1} R \doteq  ``\limind_{i \in I}" {\coprod_{x \in M}}^{\le 1} [R_i]_{|x|_M} \]
where in the first case ${\underset{x \in M}\coprod}^{\le 1}$ is computed in $\bNr_R$ and in the second case in $\bBan_R$.
\end{defn}

Definition \ref{defn:base_change_born} can be immediately generalized to arbitrary bornological or $\bInd$-normed sets, but we do not discuss this notion since it will not be used in this work. It is also clear that $M \wotimes_{\F_1} (\cdot)$ is a functor from $\bComm(\bCBorn_\Z)$ to $\bCBorn_\Z$.

\begin{rmk}
In the previous definition we do not write
\[ {\coprod_{x \in M}}^{\le 1} [R_i]_{|x|_M} = M \otimes_{\F_1} R_i \]
because $R_i$ is not a sub-ring of $R$ in general and we reserve the notation $(\cdot) \otimes_{\F_1} R$ for rings.
\end{rmk}

Notice that if $R$ is a Fr\'echet-like bornological ring, \ie one naturally presented as a projective limit
\[ R \cong \limpro_{i \in I} R_i \]
for calculating the base change form $\F_1$ to $R$ of a bornological module we first have to represent $R$ as an inductive limit of normed modules by applying Lemma \ref{lemma:calc_proj_limit_R}. So, if one has a projective limit of bornological rings as above it is not clear that one has
\[ M \wotimes_{\F_1} R \cong \limpro_{i \in I} (M \wotimes_{\F_1} R_i). \]
The next proposition shows that this is true in an important case.

\begin{prop} \label{prop:born_base_change_ban_proj_lim}
Let $M$ be normed set and $R \cong \underset{i \in I}\limpro R_i$ a bornological ring. Then
\[ M \wotimes_{\F_1} R \cong \limpro_{i \in I} (M \wotimes_{\F_1} R_i) \]
as a bornological module.
\end{prop}
\begin{proof}
First we notice that the functor $M \wotimes_{\F_1} (\cdot)$ commutes with equalizers. Indeed, if $R \rightrightarrows S$ is a pair of maps of Banach rings then
\[ \eq(M \wotimes_{\F_1} R \rightrightarrows M \wotimes_{\F_1} S ) = \eq ( {\coprod_{x \in M}}^{\le 1} R_{|x|_M} \rightrightarrows {\coprod_{x \in M}}^{\le 1} S_{|x|_M}) \]
and since taking contracting coproducts of $\Z$-modules is an exact functor in $\bBan_\Z$ we get that last expression is isomorphic to
\[ M \wotimes_{\F_1} \l ( \eq (R \rightrightarrows S)  \r ). \]
Hence it remains to show that $M \wotimes_{\F_1}(\cdot)$ commutes with products. We can apply Lemma \ref{lemma:calc_proj_limit_R} to compute $R = \prod_{i \in I} R_i$ as a bornological ring which yields
\[ R = ``\limind_{\phi \in \Phi}" R_\phi \]
using the notation of the lemma and in this particular case $\Phi$ is the set of maps $I \to \N_{\ge 1}$. Notice that we can write
\[ R_\phi = {\prod_{i \in I}}^{\le 1} [R_i]_{\frac{1}{\phi(i)}} \]
where $[R_i]_{\frac{1}{\phi(i)}}$ denotes the Banach module $R_i$ with the norm rescaled by the factor $\frac{1}{\phi(i)}$. The same lemma applied to 
\[ \prod_{i \in I} (M \wotimes_{\F_1} R_i) \]
yields 
\[ \prod_{i \in I} (M \wotimes_{\F_1} R_i) = ``\limind_{\phi \in \Phi}" M_\phi \]
where denoted
\[ M_\phi = {\prod_{i \in I}}^{\le 1} [M \wotimes_{\F_1} R_i]_{\frac{1}{\phi(i)}}. \]
Now notice that fo each fixed $\phi \in \Phi$ there exists a $\phi'$ such that
\[ \phi(i) < \phi'(i), \ \ \forall i \in I. \]
Hence the system maps
\[ R_\phi \to R_{\phi'} \]
factors through
\[ {\coprod_{i \in I}}^{\le 1} [R_i]_{\frac{1}{\phi'(i)}}  \rhook R_{\phi'}. \]
Hence, we can write
\[ M \wotimes_{\F_1} (``\limind_{\phi \in \Phi}" R_\phi) \cong M \wotimes_{\F_1} (``\limind_{\phi \in \Phi}" {\coprod_{i \in I}}^{\le 1} [R_i]_{\frac{1}{\phi(i)}}) = ``\limind_{\phi \in \Phi}" {\coprod_{x \in M}}^{\le 1} [ {\coprod_{i \in I}}^{\le 1} [R_i]_{\frac{1}{\phi(i)}}]_{|x|_M} \cong \]
\[ \cong ``\limind_{\phi \in \Phi}" {\coprod_{i \in I}}^{\le 1} [ {\coprod_{x \in M}}^{\le 1} [R_i]_{|x|_M}]_{\frac{1}{\phi(i)}} \cong ``\limind_{\phi \in \Phi}" {\coprod_{i \in I}}^{\le 1} [ M \wotimes_{\F_1} R_i ]_{\frac{1}{\phi(i)}} \]
and, by our previous remark, the last object can be identified with $\underset{i \in I}\prod (M \wotimes_{\F_1} R_i)$. Notice that we used the easy-to-check relation
\[ {\coprod_{i \in I}}^{\le 1} [ {\coprod_{x \in M}}^{\le 1} [R_i]_{|x|_M}]_{\frac{1}{\phi(i)}} \cong {\coprod_{i \in I}}^{\le 1}  {\coprod_{x \in M}}^{\le 1} [R_i]_{|x|_M \frac{1}{\phi(i)}} \cong {\coprod_{x \in M}}^{\le 1} {\coprod_{i \in I}}^{\le 1}  [R_i]_{|x|_M \frac{1}{\phi(i)}} \cong \]
\[ \cong  {\coprod_{x \in M}}^{\le 1} [ {\coprod_{i \in I}}^{\le 1} [R_i]_{\frac{1}{\phi(i)}}]_{|x|_M}. \]
\end{proof}

\begin{rmk}
Proposition \ref{prop:born_base_change_ban_proj_lim} does not generalize to any bornological set $M$.
\end{rmk}

\section{Analytic spaces over $\F_1$}

In this section we discuss an analogue of the theory of affinoid spaces and Stein spaces over $\F_1$. We develop a basic dictionary for discussing of analytic spaces over any base and when $\bTri$ is taken to be a complete valued field our notions are compatible with the usual ones of complex analytic geometry and non-Archimedean geometry. We do not develop the abstract theory in depth because it is not needed for the goal of this work. We bound ourself in introducing some basic objects, analogous to the open and closed polydisks of analytic geometry and we show that their base change to valued fields give the usual polydisks discussed in analytic geometry, both over $\C$ and over non-Archimedean fields. Then, we discuss how to put a topology on the category of such spaces using homotopical algebra. This topology generalizes the transcendental topology of analytic geometry and coincides with it when applied for analytic spaces over valued fields. Again, only the basic features of these theories are discussed in this work. Much more details will be given in \cite{BeKr2} and \cite{BeKr3}.

We start by introducing the notion of derived analytic space. Our geometry will be base on the categories of bornological modules, completed when the base is a Banach ring and normed when the base is $\F_1$. Therefore in this section we will use the shorthand notation
\[ \bComm_\bTri = \begin{cases} 
\bComm(\bInd^m(\bNr_{\F_1})) \text{ if } \bTri = \F_1 \\
\bComm(\bInd^m(\bBan_R)) \text{ otherwise}
\end{cases}  \] 
and
\[ \bsComm_\bTri = \begin{cases} 
\bsComm(\bInd^m(\bNr_{\F_1})) \text{ if } \bTri = \F_1 \\
\bsComm(\bInd^m(\bBan_R)) \text{ otherwise}
\end{cases}. \] 
This is motivated by the fact that there exists objects in $\bsComm(\bInd^m(\bNr_{\F_1}))$ which we want to consider as analytic spaces over $\F_1$ which are not Banach sets, therefore the notion of complete normed set over $\F_1$ does not seems to be as useful as the notion of Banach module over Banach rings.

The first step is to introduce the notions of affine analytic spaces.  As remarked in Remark \ref{rmk:infty_cat} we can associate to the model categories we are working with $\infty$-categories in a canonical way. We denote the categories obtained in this way by adding a $\infty-$ in the notation before the name of their model category, as for example $\infty-\bsComm_\bTri$. 

\begin{defn} \label{defn:infty_affine_analytic_spaces}
The category of \emph{$\infty$-affine analytic spaces over $\bTri$} is the category $\infty-\bAff_\bTri = \infty-\bsComm_\bTri^\op$. We denote the duality functor by $\Spec: \infty-\bsComm_\bTri \to \infty-\bAff_\bTri$.
\end{defn}

\begin{defn} \label{defn:derived_affine_analytic_spaces}
The category of \emph{derived affine analytic spaces over $\bTri$} is the category $\bdAff_\bTri = \bHo(\bsComm_\bTri^\op)$, the homotopy category of the opposite of the category of $\bsComm_\bTri$. We denote the duality functor by $\Spec: \bHo(\bsComm_\bTri) \to \bdAff_\bTri$.
\end{defn}

\begin{rmk}
There is a fully faithful inclusion functor
\[ \iota: \bComm_\bTri^\op \rhook \bdAff_\bTri \]
which is adjoint to the functor
\[ \pi_0: \bdAff_\bTri \to \bComm_\bTri^\op. \]
\end{rmk}

\begin{defn} \label{defn:affine_analytic_spaces}
The category of \emph{affine analytic spaces over $\bTri$} is the full sub-category of $\bdAff_\bTri$ identified by the functor $\iota$ and it will be denoted by $\bAff_\bTri$.
\end{defn}

\begin{rmk}
Notice that thanks to the classical results of analytic geometry, when $\bTri = k$, a complete valued field, the categories of affine analytic space, \ie affinoid spaces, Stein spaces, dagger affinoid spaces, etc.. embed naturally in the categories we have defined so far. For more information about this topic see \cite{Bam}, \cite{BaBe}, \cite{BaBeKr}.
\end{rmk}

Recall that we defined a model structure on $\bMod(A)$ for any $A \in \bsComm_\bTri$. This allows us to speak about derived functors amongst these categories and we will use them for defining a Grothendieck topolgoy on $\bHo(\bsComm_\bTri)$. 

\begin{defn} \label{defn:homotopy_epi}
	A morphism $f: A \to B$ in $\bsComm_\bTri$ is said to be a \emph{homotopy epimorphism} if
	\[ B \wotimes^\L_A B \to B \] 
	is an isomorphism in $\bHo(\bMod(A))$. The opposite map in $\infty-\bAff_\bTri$ is called \emph{homotopy monomorphism}.
\end{defn}

\begin{defn} \label{defn:covering}
	We say that a collection of morphisms $\{f_i: A \to B_i\}$ is a \emph{covering} of $\Spec(A)$ if the family of functors $(f_i)^*: \bHo(\bMod(A)) \to \bHo(\bMod(B_i))$ is conservative.
\end{defn}

\begin{prop} \label{prop:homotopy_zariski}
	The family of homotopy monomorphisms with finite coverings in the sense of Definition \ref{defn:covering} define a Grothendieck model topology on $\infty-\bAff_\bTri$ in the sense of HAG (cf. Definition 1.3.1.1 of \cite{TVe3}).
\end{prop}
\begin{proof}
We need to check that the notions we introduced define a Grothendieck topology on $\bdAff_\bTri$. Thus, homotopy monomorphisms are stable by pullbacks because they are monomorphisms in $\bdAff_\bTri$ and the composition of two coverings is clearly a covering because the composite covering clearly defines a conservative family of functors because we are dealing only with finite coverings.
\end{proof}

\begin{defn} \label{defn:homotopy_zar_topology}
	The topology defined in Proposition \ref{prop:homotopy_zariski} is called the \emph{homotopy Zariski topology}.
\end{defn}

\begin{rmk}
We call the topology defined in Definition \ref{defn:homotopy_zar_topology} homotopy Zariski topology whereas formal homotopy Zariski topology would have been a more suitable name. But we do not have a good notion of what is a finitely presented morphism, neither in $\bComm(\bCBorn_k)$, when $k$ is a valued field. The usual categorical notion of finitely presentedness fails to encompass the spaces one would like to call of finite presentation. Hence we take take this topology as our analogue of the Zariski topology, keeping in mind that it is more like the formal Zariski topology of algebraic geometry.
\end{rmk}

We postpone to next sub-section the discussion of some basic examples of homotopy monomorphisms for analytic spaces over $\F_1$ (see Example \ref{exa:homotopy_mono}). We will do that after we have introduced the disks and polydisk over $\F_1$ and then see some examples of open embeddings between them. In the next example we recall how the homotopy Zariski topology looks like in classical settings. 

\begin{exa}
\ben
\item Let's consider $\bTri = k$ a non-Archimedean complete valued field.
 In this case the category $\bsComm_k$ is very big but it contains as a full sub-category the category of affinoid algebras. It has been shown in \cite{BeKr} that a morphism of affinoid algebras $A \to B$ is a homotopy epimorphisms if and only if it is an affinoid subdomain localization, \ie if it identifies $\Spec(B)$ as an affinoid subdomain of $\Spec(A)$. Therefore, the homotopy Zariski topology of $\bsComm_k^\op$ restricts to the weak G-topology of classical rigid geometry on the category of affinoid algebra.

\item Now let's consider $\bTri = k$ any complete (for the sake of simplicity non-trivially) valued field, Archimedean or not. In the same vein as one defines affinoid algebras one can define dagger affinoid algebras and dagger affinoid spaces. The advantage of this point of view is that it works uniformly over any base field (cf. \cite{Bam} and \cite{BaBe} for more details). In \cite{BaBe} it has been shown that the homotopy Zariski topology of $\bsComm_k^\op$ restricts to the weak G-topology of dagger affinoid spaces. In particular, a morphism of dagger affinoid algebras is a homotopy epimorphism if and only if it is dagger affinoid sub-domain localization.

\item Let $\bTri=k$ be as in last example. One has the theory of Stein spaces over $k$, which are spaces completely determined by the Frech\'et algebra of global analytic functions, hence they are affine analytic spaces in the sense of Definition \ref{defn:affine_analytic_spaces}. These spaces are without borders and hence an open embedding between them really identifies a space as a open subspace of the other, even topologically (in particular for $k = \C$ one has the usual topology of complex spaces). In \cite{BaBeKr} it has been shown that the homotopy Zariski topology of $\bsComm_k^\op$ is strictly related to the transcendental topology of analytic spaces. Again, a morphism of Stein spaces is a homotopy monomorphism if and only if it is an open immersion. 

\een
\end{exa}

Of course one can also introduce the notion of \'etale map using the cotangent complex and then the \'etale topology or other topologies of interest. We will discuss \'etale maps and the \'etale topology in a future work. 

We use the functor of points approach for defining analytic spaces over $\F_1$, hence it is useful to recall what is a stack in the settings we are working with.

\begin{defn} \label{defn:der_analytic_stack}
The category of \emph{$\infty$-analytic stacks} over $\bTri$ is defined to be the full-category of $\infty-\bFunc(\infty-\bAff_\bTri^\op, \bsSets)$ given by the objects that preserves weak-equivalences and satisfies descent for homotopy hypercoverings associated to the homotopy Zariski topology and it is denoted $\infty-\bStack_\bTri$.

Let $\sF: \bdAff^\op \to \bHo(\bsSets)$ be a functor, we say that $\sF$ is a \emph{derived analytic stack} if it fits in the diagram
\[
\begin{tikzpicture}
\matrix(m)[matrix of math nodes,
row sep=2.6em, column sep=2.8em,
text height=1.5ex, text depth=0.25ex]
{ \infty-\bAff_\bTri^\op  & \bsSets   \\
  \bdAff_\bTri^\op  & \bHo(\bsSets)  \\};
\path[->,font=\scriptsize]
(m-1-1) edge node[auto] {} (m-1-2);
\path[->,font=\scriptsize]
(m-1-1) edge node[auto] {} (m-2-1);
\path[->,font=\scriptsize]
(m-1-2) edge node[auto] {} (m-2-2);
\path[->,font=\scriptsize]
(m-2-1) edge node[auto] {$\sF$} (m-2-2);
\end{tikzpicture}
\]
where the upper horizontal arrow is an $\infty$-stack. The category of derived analytic stacks over $\bTri$ is denoted $\bdStack_\bTri$. 

A functor $\sF: \bAff^\op \to \bSets$ is said to be an \emph{analytic stack} if it fits in the diagram
\[
\begin{tikzpicture}
\matrix(m)[matrix of math nodes,
row sep=2.6em, column sep=2.8em,
text height=1.5ex, text depth=0.25ex]
{ \bdAff_\bTri^\op  & \bHo(\bsSets)   \\
  \bAff_\bTri^\op  & \bSets  \\};
\path[->,font=\scriptsize]
(m-1-1) edge node[auto] {} (m-1-2);
\path[->,font=\scriptsize]
(m-1-1) edge node[auto] {$\pi_0$} (m-2-1);
\path[->,font=\scriptsize]
(m-1-2) edge node[auto] {$\pi_0$} (m-2-2);
\path[->,font=\scriptsize]
(m-2-1) edge node[auto] {$\sF$} (m-2-2);
\end{tikzpicture}
\]
where the upper horizontal arrow is a derived stack. The category of analytic stacks over $\bTri$ is denoted $\bStack_\bTri$.
\end{defn}

Every $\infty$-affine analytic space defines an $\infty$-analytic stack, every derived affine analytic space defines a derived analytic stack and every affine analytic space defines an analytic stack through the suitable version of Yoneda lemma (\cf Corollary 1.3.2.5 of \cite{TVe3}). Moreover, the class of (resp. $\infty$, resp. derived) analytic stacks form a category and it makes sense to discuss about (homotopy) monomorphisms between them.

\begin{defn} \label{defn:covering_stacks}
Let $X$ be a (resp. $\infty$, resp. derived) affine analytic space.
\begin{itemize}
\item We say that a (resp. $\infty$, resp. derived) analytic stack $\sF$ is a \emph{Zariski open of $X$},  if there exists a family $\{X_i \to X \}_{i \in I}$ of Zariski open embeddings, with $X_i$ (resp. $\infty$, resp. derived) affine, such that $\Im(\underset{i \in I} \coprod X_i \to X) \cong \sF$.
\item We say that a morphism $f: \sF \to \sG$ between (resp. $\infty$, resp. derived) analytic stacks is a \emph{Zariski open embedding} if for all $X$ and all $X \to \sG$ the map $\sF \times_\sG X \to X$ is a (resp. homotopy) monomorphism whose image is a Zariski open of $X$.
\end{itemize}

\end{defn}

\begin{defn} \label{defn:der_analytic_space}
We say that a (resp. $\infty$, resp. derived) analytic stack $\sF$ is a \emph{(resp. $\infty$, resp. derived) analytic space} if there is a family of Zariski embeddings $\{X_i \to \sF \}_{i \in I}$ from (resp. $\infty$, resp. derived) affine analytic spaces such that $\underset{i \in I}\coprod X_i \to \sF$ is a an epimorphism. 
The respective categories are denotes $\infty-\bAn_\bTri$, $\bdAn_\bTri$ and $\bAn_\bTri$.
\end{defn}

So far we defined the base change functor $(\cdot) \wotimes_{\F_1} R: \bInd^m(\bNr_{\F_1}) \to \bInd^m(\bBan_R)$. It is clear that this functor induces a base change $(\cdot) \wotimes_{\F_1} R: \bAff_{\F_1} \to \bAff_R$. The next proposition checks that this functor is compatible with the homotopy Zariski topology.

\begin{prop} \label{prop:base_changes}
The base change functor $(\cdot) \wotimes_{\F_1} R$ induces the following base change functors:
\ben
\item $(\cdot) \wotimes_{\F_1} R: \infty-\bAn_{\F_1} \to \infty-\bAn_R$;
\item $(\cdot) \wotimes_{\F_1}^\L R: \bdAn_{\F_1} \to \bdAn_R$;
\item $(\cdot) \wotimes_{\F_1} R: \bAn_{\F_1} \to \bAn_R$;
\item $(\cdot) \wotimes_{\F_1} R: \infty-\bStack_{\F_1} \to \infty-\bStack_R$;
\item $(\cdot) \wotimes_{\F_1}^\L R: \bdStack_{\F_1} \to \bdStack_R$;
\item $(\cdot) \wotimes_{\F_1} R: \bStack_{\F_1} \to \bStack_R$.
\een
\end{prop}
\begin{proof}
The statements can be proved in the same way as Proposition 5.12 of \cite{Toen} is proved.
\end{proof}

Now that we have the basic tools for doing analytic geometry we introduce disks, polydisks and annuli over $\F_1$.

\subsection{Disks and annuli}

The simplest objects of analytic geometry are disks and polydisks and this is true also for the analytic geometry over $\F_1$ we are defining. These objects are the analogue of the polydisks of usual analytic geometry over valued fields and they come in different flavours: open ones, closed ones and closed overconvergent. Moreover, we also introduce some kind of perfect disks which will be linked to the theory of perfectoid field and $p$-adic Hodge theory later on.

\begin{defn} \label{defn:open_closed_disk}
We define the \emph{closed disk} of radius $r$ over $\F_1$ as
\[ D_r^\bullet = \Spec (S_{\bNr_{\F_1}^{\le 1}}(\ast_r)) \]
and the \emph{open disk} of radius $r$ as
\[ D_r^\circ = \Spec (\limpro_{\rho < r} S_{\bNr_{\F_1}^{\le 1}}(\ast_\rho)) \]
and the \emph{overconvergent disk} of radius $r$ as
\[ D_r^\dagger = \Spec (\limind_{\rho > r} S_{\bNr_{\F_1}^{\le 1}}(\ast_\rho)). \]
\end{defn}

Definition \ref{defn:open_closed_disk} we used the notation for the symmetric algebra introduced in Notation \ref{notation:box} and Example \ref{exa:basic}, which is compatible with the notation used in the appendix (cf. Remark \ref{rmk:S_algebra_compatible_symm}).

\begin{prop} \label{prop:base_change_disks}
Let $K$ be a non-trivially valued field then
\[ D_r^\bullet \wotimes_{\F_1} K = D_r^\bullet(K), \ \ D_r^\circ \wotimes_{\F_1} K = D_r^\circ(K), \ \ D_r^\dagger \wotimes_{\F_1} K = D_r^\dagger(K) \]
where $D_r^\bullet(K)$, $D_r^\circ(K)$ and $D_r^\dagger(K)$ denote the $\ell^1$-closed, open and overconvergent disk over $K$ of radius $r$. 
\end{prop}
\begin{proof}
By definition 
\[ D_r^\bullet(K) \cong \Spec(S_K(r)), \]
thus the first isomorphism follows from the formula
\[ S_{\bNr_{\F_1}^{\le 1}}(\ast_r) \wotimes_{\F_1} K = S_K(r), \]
and $S_K(r)$ is the algebra of $\ell^1$-analytic functions on the closed disk of radius $r$.

The isomorphism
\[ D_r^\dagger \wotimes_{\F_1} K = D_r^\dagger(K) \]
follows by the fact that 
\[ (\limind_{\rho > r} S_{\bNr_{\F_1}^{\le 1}}(\ast_r)) \wotimes_{\F_1} K \cong (\limind_{\rho > r} S_{\bNr_{\F_1}^{\le 1}}(\ast_r) \wotimes_{\F_1} K) \cong (\limind_{\rho > r} S_{\bNr_{\F_1}^{\le 1}}(r)) \]
because $(\cdot) \wotimes_{\F_1} R$ is a left adjoint functor. Finally 
\[ D_r^\circ \wotimes_{\F_1} K = D_r^\circ(K) \]
follows by the direct computation
\[ (\limpro_{\rho < r} S_{\bNr_{\F_1}^{\le 1}}(\ast_r)) \wotimes_{\F_1} K \cong   (\limpro_{\rho < r} S_{\bNr_{\F_1}^{\le 1}}(\ast_r) \wotimes_{\F_1} K). \]
\end{proof}

The same reasoning given for global analytic spaces in the sense of Poineau, see \cite{Poi}.

\begin{prop} \label{prop:base_change_global_disks}
Let $R$ be a Banach ring then
\[ D_r^\bullet \wotimes_{\F_1} R = D_r^\bullet(R), \ \ D_r^\circ \wotimes_{\F_1} R = D_r^\circ(R), \ \ D_r^\dagger \wotimes_{\F_1} R = D_r^\dagger(R) \]
where $D_r^\bullet(R)$ is the global closed $\ell^1$-disk $D_r^\circ(K)$ is the global open disk $D_r^\dagger(K)$ is the global overconvergent disk.
\end{prop}
\begin{proof}
The proof uses the same ideas of \ref{prop:base_change_disks}. We refer to \cite{Poi} for the notion of global analytic spaces.
\end{proof}

The same constructions and results can be given for polydisks. Briefly, the $n$-dimensional closed polydisk of polyradius $r = (r_1, \ldots, r_n) > 0$ is defined by
\[ D_r^{\bullet, n} = \Spec (S_{\bNr_{\F_1}^{\le 1}}(\ast_{r_1}) \otimes_{\F_1} \ldots \otimes_{\F_1} S_{\bNr_{\F_1}^{\le 1}}(\ast_{r_n})). \]
The open and dagger polydisks are defined in analogous ways. We leave the details of these constructions to the reader since in this work we will be only interested to disks of one dimension. In the next examples we discuss basic examples of homotopy epimorphisms.

\begin{exa} \label{exa:homotopy_mono}
\ben
\item Consider two closed disks $D_r^\bullet$ and $D_{r'}^\bullet$ with $r' < r$. The identity map $S_r = S_{\bNr_{\F_1}^{\le 1}}(\ast_r) \to S_{r'} = S_{\bNr_{\F_1}^{\le 1}}(\ast_{r'})$ is bounded (but it is not an isomorphism) and hence induces a morphism $D_{r'}^\bullet \to D_r^\bullet$ of analytic spaces over $\F_1$. We show that this morphism is a homotopy monomorphism so that we can think that the disks over $\F_1$ of smaller radii are embedded in the ones with bigger radii, as in usual analytic geometry. We need to show that 
\[ S_{r'} \otimes_{S_r}^\L S_{r'} \to S_{r'} \]
is an isomorphism. We can use the bar resolution of $S_{r'}$ for representing $S_{r'} \otimes_{S_r}^\L S_{r'}$ and we have to show that
\[ \coker( S_{r'} \otimes_{\F_1} S_{r'} \otimes_{\F_1} S_r \rightrightarrows S_{r'} \otimes_{\F_1} S_r ) \cong S_{r'}. \]
The underlying map of sets is clearly bijective because $S_r \to S_{r'}$ is the identity map, hence we just need to show that it is a strict morphism. Thus, as a monoid
\[ \coker( S_{r'} \otimes_{\F_1} S_{r'} \otimes_{\F_1} S_r \rightrightarrows S_{r'} \otimes_{\F_1} S_r ) \cong \frac{\N^2}{\Delta(\N)} \cong \N \]
and with this identification the norm on the quotient coincides with the norm on $S_{r'}$.

Notice that when $k$ is a (non-Archimedean) valued field the morphism 
\[ D_{r'}^\bullet \wotimes_{\F_1} k \cong D_{r'}^\bullet(k) \subset D_r^\bullet \wotimes_{\F_1} k \cong D_r^\bullet(k)  \]
induced by the base change is the usual inclusion of ($\ell^1$-)disks centered in zero.

\item Of course one can consider inclusions between open disk and overconvergent disk over $\F_1$. For an inclusion of overconvergent disks, one can check that $D_{r'}^\dagger \to D_r^\dagger$  is a homotopy monomorphism reducing to the case of closed disks by noticing that filtered direct limits commute with quotients. Hence, 
\[  \coker ( \limind_{\rho' > r'} S_{\rho'} \otimes_{\F_1} \limind_{\rho' > r'} S_{\rho'} \otimes_{\F_1} \limind_{\rho > r} S_\rho \rightrightarrows \limind_{\rho' > r'} S_{\rho'} \otimes_{\F_1} \limind_{\rho > r} S_\rho ) \cong \]
\[ \cong \limind_{\rho' > r'} \limind_{\rho > r'} \coker( S_{\rho'} \otimes_{\F_1} S_{\rho'} \otimes_{\F_1} S_\rho \rightrightarrows S_{\rho'} \otimes_{\F_1} S_\rho ) \cong \limind_{\rho' > r'} S_{\rho'}. \]
A similar reasoning shows that also the inclusion $D_{r'}^\circ \to D_r^\circ$ are homotopy monomorphisms. In this case one can show that $\underset{\rho > r}{\limpro}^{(1)} = 0$ for reducing to the case of closed disks. We omit the details.

\item Building on the previous examples one can show that also the inclusions of disks of different kinds are homotopy monomorphisms, \ie the inclusion of closed disk in open ones or overconvergents in opens etc.. in all possible combinations are all homotopy epimorphisms.
\item We can define the affine line over $\F_1$ as the analytic space
\[ \A_{\F_1}^1 = \Spec(\limpro_{\rho < \infty} S_\rho). \]
Notice that
\[ \A_{\F_1}^1 \wotimes_{\F_1} \C \cong \A_\C^1 \]
and that
\[ (\limpro_{\rho < \infty} S_\rho) \wotimes_{\F_1} \C \cong \limpro_{\rho < \infty} S_\rho \wotimes_{\F_1} \C \]
is the Fr\'echet algebra of entire analytic functions. Similar considerations hold for any non-Archimedean complete valued field in place of $\C$. Using $n$-dimensional polydisks we can define $n$-dimensional affine spaces in a similar fashion.

\een
\end{exa}

\begin{prop} \label{prop:nuclear_disks}
The functions on the open and the overconvergent disks over $\F_1$ are examples of nuclear bornological sets (in the sense of Definition \ref{defn:nuclear_set}) whereas closed disks are not.
\end{prop}
\begin{proof}
Direct consequence of Theorem \ref{thm:ind_cofinal} and Theorem \ref{thm:pro_cofinal}.
\end{proof}

By replacing the monoid $\N$ with its group completion $\Z$ one gets the definition of analytic functions on closed/open/overconvergent annuli over $\F_1$ whose base change to valued fields agree with the usual notion of analytic functions over annuli. For fixing the notation we give the following definition.

\begin{defn} \label{defn:annuli_F1}
We define the \emph{closed annulus} of radii $0 < r_1 < r_2$ over $\F_1$ as
\[ A_{r_1, r_2}^\bullet = \Spec (\cZ_{r_1, r_2}) \]
(cf. Example \ref{exa:basic} (9) for the notation $\cZ_{r_1, r_2}$); the \emph{open annulus} of radii $0 < r_1 < r_2$ as
\[ A_{r_1, r_2}^\circ = \Spec (\limpro_{r_1 < \rho_1 < \rho_2 < r_2} \cZ_{\rho_1, \rho_2}) \]
and the \emph{overconvergent disk} of radius $r$ as
\[ A_{r_1, r_2}^\dagger = \Spec (\limind_{\rho_1 < r_1, r_2 < \rho_2}  \cZ_{\rho_1, \rho_2}). \]
\end{defn}

The observations made so far for disks can be worked out also for annuli. In particular, the functions on the open and closed annuli are examples of nuclear bornological sets. 

The next kind of disks we are interested to discuss are disks for which the underlying monoid is perfect, in some sense (\ie by requiring some divisibility property). Therefore we introduce the following notions.

\begin{defn} \label{defn:perfect_disks}
We define the \emph{closed perfect disk} of radius $r > 0$ as
\[ P D_r^\bullet \doteq \Spec(\cQ_r^+),  \]
and the \emph{open perfect disk} as
\[ P D_r^\circ \doteq \Spec(\limpro_{\rho < r} \cQ_\rho^+). \]
\end{defn}

Recall that $\cQ_r^+$ is the monoid introduced in Example \ref{exa:basic}. One can easily introduce also the overconvergent perfect disks but we omit to discuss this kind because they will not be used later on.

\begin{defn} \label{defn:p_perfect_disks}
We define the \emph{closed $p$-perfect disk} of radius $r > 0$ as
\[ pP D_r^\bullet \doteq \Spec(\cZ_r^+[\frac{1}{p}]),  \]
and the \emph{open $p$-perfect disk} as
\[ pP D_r^\circ \doteq \Spec(\limpro_{\rho < r} \cZ_\rho^+[\frac{1}{p}]). \]
\end{defn}

Again the notation $\cZ_r^+[\frac{1}{p}]$ refers to Example \ref{exa:basic}.
It is interesting to compute the base change of these perfect disks to some base field, as we do in the next example.

\begin{exa} \label{exa:perfect_fields}
\ben
\item Let $k$ be a trivially valued field. Then, for $0 < r < 1$
\[ P D_r^\bullet \wotimes_{\F_1} k \cong \Spec(\cQ_r^+ \wotimes_{\F_1} k ) \]
is isomorphic to the spectrum of the ring of strictly convergent analytic functions on the disk of radius $r$ with all $n$-th roots of the variable added, \ie the direct limit of
\[ \Spec( {\limind_{T \mapsto T^n}}^{\le 1} k\lt T \gt ) = \Spec( \what{\bigcup_{n \in \N} k [[ T^{\frac{1}{n}}]]}) , \]
with suitable normalizations on the norms (cf. Example \ref{exa:basic} for the normalization of the norms).

\item Let $k$ be trivially valued. Then, for $0 < r < 1$
\[ p P D_r^\bullet \wotimes_{\F_1} k \cong \Spec(\cZ_r[\frac{1}{p}]^+ \wotimes_{\F_1} k ) \]
is isomorphic to the spectrum of the ring of strictly convergent analytic functions on the disk of radius $r$ with all $p$-th roots of the variable added, \ie the direct limit of
\[ \Spec( {\limind_{T \mapsto T^{p^n}}}^{\le 1} k\lt T \gt ) = \Spec( \what{k [[ T^{\frac{1}{p^\infty}}]]}) , \]
with suitable normalizations on the norms.
\een
\end{exa}

We need to introduce the concept of perfect annuli for the applications of the next section.

\begin{defn} \label{defn:perfect_annuli}
We define the \emph{closed perfect annulus} of radius $0 < r_1 < r_2$ as
\[ P A_{r_1, r_2}^\bullet \doteq \Spec(\cQ_{r_1, r_2}),  \]
the \emph{open perfect annulus} as
\[ P A_{r_1, r_2}^\circ \doteq \Spec(\limpro_{r_1 < \rho_1 < \rho_2 < r_2} \cQ_{\rho_1, \rho_2}). \]
the \emph{closed $p$-perfect annulus} as
\[ pP A_{r_1, r_2}^\bullet \doteq \Spec(\cZ_{r_1, r_2}  [ {\scriptstyle \frac{1}{p}} ]),  \]
and the \emph{open $p$-perfect annulus} as
\[ pP A_{r_1, r_2}^\circ \doteq \Spec(\limpro_{r_1 < \rho_1 < \rho_2 < r_2} \cZ_{\rho_1, \rho_2}[{\scriptstyle \frac{1}{p}}]). \]
\end{defn}

Building on Example \ref{exa:perfect_fields} we get the next example.

\begin{exa} \label{exa:perfect_fields2}
Let $k$ be a trivially valued field.

\ben
\item  For $0 < r_1, r_2 < 1$ one has
\[ \Spec(\cQ_{r_1, r_2} \wotimes_{\F_1} k ) \cong \Spec( {\limind_{T \mapsto T^n}}^{\le 1} \Frac (k\lt T \gt ) ) \cong \Spec( \what{\bigcup_{n \in \N} k (( T^{\frac{1}{n}}))}) \]
the affine analytic space associated to the field of fractions of $\cQ_r^+ \wotimes_{\F_1} k$.

\item Again, for $0 < r < 1$ one has
\[ \Spec(\cZ_r[\frac{1}{p}]_r \wotimes_{\F_1} k ) \cong \Spec( {\limind_{T \mapsto T^{p^n}}}^{\le 1} \Frac (k\lt T \gt ) ) \cong \Spec( \what{ k (( T^{\frac{1}{p^\infty}}))}) \]
the affine analytic space associated to the field of fractions of $\cZ_r[\frac{1}{p}]^+ \wotimes_{\F_1} k$.

\een
\end{exa}

\begin{prop} \label{prop:nuclear_perfect_disks}
	The open and the overconvergent disks and annuli over $\F_1$ are examples of nuclear bornological sets whereas closed disks/annuli are not.
\end{prop}
\begin{proof}
Direct consequence of Corollary \ref{cor:ind_pro_cofinal}.
\end{proof}

More in general one can think to ordered Banach monoids as generalized annuli over $\F_1$ and study their base changes to various valued fields or to Banach rings.

\begin{prop} \label{prop:perfectoid_fields}
Let $M$ be a $p$-divisible normed commutative group without torsion and $F$ a trivially valued perfect field of characteristic $p$. Then, if the norm on $M$ is multiplicative and it induces a total order on $M$ one has that
\[ M \wotimes_{\F_1} F \]
is a perfectoid field of characteristic $p$.
\end{prop}
\begin{proof}
Consider the subset 
\[ M^\circ = \{ x \in M | |x|_M \le 1 \}. \]
Clearly $M = M^\circ \cup (M^\circ)^{-1}$. $M \wotimes_{\F_1} F$ is a field because its elements can be described as Laurent series whose principal part has finite support indexed by $(M^\circ)^{-1}$ and whose positive part is indexed by $M^\circ$.

Now, we can consider the non-Archimedean base change of $M$ to $F$ (cf. Remark \ref{rmk:non-arch_base_change}). Then $M \wotimes_{\F_1}^\na F$ becomes a complete valued field, with dense valuation, whose valuation ring is $M^\circ \wotimes_{\F_1} F$. Since the residue field of $M^\circ \wotimes_{\F_1}^\na F$ is $F$, then $M \wotimes_{\F_1}^\na F$ is a perfectoid field. Finally, we notice that it is easy to check the topology of $M \wotimes_{\F_1} F$ is equivalent to the topology of $M \wotimes_{\F_1}^\na F$, and actually the norm of $M \wotimes_{\F_1}^\na F$ is the spectral norm of $M \wotimes_{\F_1} F$, hence the topology of $M \wotimes_{\F_1} F$ is the one of a perfectoid field, which is the only data required for defining a perfectoid field.
\end{proof}

\section{Applications}

In this section we discuss some applications of the theory developed so far in two important and related topics: the Fargues-Fontaine curve and the $p$-typical Witt vectors. We begin this section with some recall from the theory of the Fargues-Fontaine curve. A detailed account of the theory can be found in \cite{FF} and \cite{FF2}.

\subsection{Recall on the theory of the Fargue-Fontaine curve} \label{sec:recall}

Recall that there exists a functor $(\cdot)^\flat: \bPerfField_{0,p} \to \bPerfField_p$, from the category of perfectoid field of mixed characteristic $(0, p)$ to the category of perfectoid fields of equal characteristic $p$, that associates to a perfectoid field $E$ its tilt $E^\flat$. $E^\flat$ is constructed by considering the valuation ring $\cO_E$ of $E$ and by noticing that the multiplicative monoid
\[ \cO_{E^\flat} = \limpro_{x \mapsto x^p} \cO_E \]
has a natural addition operation which renders $\cO_{E^\flat}$ a valuation ring of equal characteristic $p$. Then $E^\flat$ is defined as the fraction field of $\cO_{E^\flat}$. Notice that the multiplicative monoids of $E$ and $E^\flat$ are isomorphic.

Let now $K$ be a perfectoid field of characteristic $p$. The Fargues-Fontaine curve associated to $K$ classifies the equivalence classes of un-tilts of $K$, more precisely the closed points of the Fargues-Fontaine curve are in bijection with un-tilts of $K$. So, we briefly recall these notions.

\begin{defn} \label{defn:untilt}
Let $K$ be a perfectoid field of characteristic $p$. A perfectoid field $E$ characteristic zero is said to be an \emph{un-tilt of $K$} if there exists a morphism $\iota: K \rhook E^\flat$ which is an embedding of topological fields and a finite extension. 
\end{defn}

Since we are interested on the geometry of the Fargues-Fontaine curve we use the word un-tilt also for un-tilts of degree bigger than $1$ (the degree of an un-tilt is the degree of the extension $E^\flat/\iota(K)$) instead of the use of some authors of the word un-tilt for referring to un-tilts of degree 1.
We say that two un-tilts $\iota_E: K \rhook E^\flat$ and $\iota_F: K \rhook F^\flat$ are \emph{isomorphic} if there exists an isomorphism of topological fields $\sigma: E \to F$ such that $\sigma^\flat \circ \iota_E = \iota_F$. 
The Fargues-Fontaine curve classifies un-tilts up to equivalence.

\begin{defn} \label{defn:equivalence_untilts}
Let $E_1$ and $E_2$ be two un-tilts of $K$ and $\phi$ the Frobenius automorphism of $K$. We say that $E_1$ and $E_2$ are \emph{equivalent} if there exists an $n \in \N$ and an isomorphism $E_1 \to E_2$ such that the diagram
\[
\begin{tikzpicture}
\matrix(m)[matrix of math nodes,
row sep=2.6em, column sep=2.8em,
text height=1.5ex, text depth=0.25ex]
{ K  & E_1^\flat   \\
  K  & E_2^\flat  \\};
\path[->,font=\scriptsize]
(m-1-1) edge node[auto] {$\iota_{E_1}$} (m-1-2);
\path[->,font=\scriptsize]
(m-1-1) edge node[auto] {$\phi^n$} (m-2-1);
\path[->,font=\scriptsize]
(m-1-2) edge node[auto] {} (m-2-2);
\path[->,font=\scriptsize]
(m-2-1) edge node[auto] {$\iota_{E_2}$} (m-2-2);
\end{tikzpicture}
\]
commutes.
\end{defn}

The Fargues-Fontaine curve associated to a perfectoid field $K$, of characteristic $p$, is an adic space constructed in the following way. Consider the ring of $p$-typical Witt vectors $W_p(K^\circ)$ and its ring of fractions $\bW_K = W_p(K^\circ)[\frac{1}{p}][\frac{1}{[\omega}]]$, where $\omega$ is an element in $K^\circ$ with norm less than $1$ and $[\omega]$ its Teichmuller lift. For all  $\underset{n >> -\infty}\sum a_n p^n \in \bW_K$ one can consider the family of norms defined
\begin{equation} \label{eqn:FF_norm}
|\sum_{n >> -\infty} a_n p^n |_r = \sup_{n \in \Z} |a_n|p^{- r n}, \ \ 0 < r < \infty. 
\end{equation}
One can take the Fr\'echet completion of $\bW_K$ with respect to metric induced by these norms (in the sense of the theory of Fr\'echet spaces) obtaining a Fr\'echet algebra that we denote $A_K$. The completion of $W_p(K^\circ)[\frac{1}{p}]$ with respect to the same family of norms is denoted $A_K^+$.
More explicitly, one can form the following system of Banach algebras. For any $1 \le r < \infty$ one can consider on $\bW_K$ the norm
\begin{equation} \label{eqn:FF_norm3} 
|a|_{r, r^{-1}} = \max \{ |a|_r, |a|_{r^{-1}} \}, \ \ a \in \bW_K.
\end{equation}
The completion of $\bW_K$ with respect to the norm $|\cdot|_{r, r^{-1}}$ is a Banach algebra that can be explicitly described as formed by elements of the form $\underset{n >> -\infty}\sum a_n p^n$ such that
\[ \sup_{n \in \N} |a_n|p^{-r^{-1} n} < \infty, \ \ \sup_{n \in \N} |a_{-n}|p^{-r n} < \infty, \]
and it is denoted by $A_{K, [r^{-1}, r]}$. Notice that for each $1 \le r < r'$ the canonical map
\[ A_{K, [(r')^{-1}, r']} \to A_{K, [r^{-1}, r]} \]
is bounded, hence we can compute the projective limit as $r \to \infty$ and it turns out that
\begin{equation} \label{eqn:FF_norm2} 
\limpro_{r \to \infty} A_{K, [r^{-1}, r]} \cong A_K. 
\end{equation}
$A_K$ has the property that the set $\Max(A_K)$ of its closed maximal ideals is canonically in bijection with the isomorphism classes of un-tilts of $K$. The Frobenius endomorphism $\phi$ of $K$ gives rise to an action of $\Z$ on $\Max(A_K)$ whose quotient $[\frac{\Max(A_K)}{\phi^\Z}]$ correspond to un-tilts of $K$ up to equivalence and it can be identified with the closed points of an adic space. This latter adic space is called \emph{the Fargues-Fontaine curve} of $K$ and we denote it by $\FF_K$

In a more geometric fashion, $A_K$ is the ring of analytic functions of an affine adic space, that we can denote $\Spec(A_K)$, whose quotient by the action of the Frobienus is $\FF_K$. We will now describe how to obtain $A_K$ and the Fargues-Fontaine curve as base changes of analytic spaces over $\F_1$ for suitable choices of $K$.

\subsection{Some examples of Fargue-Fontaine curves over $\F_1$}

\begin{notation}
From now on we denote 
\begin{itemize}
\item  $\kappa = \what{\underset{n \in \N} \bigcup \F_p((t^{\frac{1}{n}}))}$, $\kappa^\circ = \what{\underset{n \in \N} \bigcup \F_p[[t^{\frac{1}{n}}]]}$;
\item $\kappa' = \what{\underset{n \in \N} \bigcup \F_p((t^{\frac{1}{p^n}}))} $, $(\kappa')^\circ = \what{\underset{n \in \N} \bigcup \F_p[[t^{\frac{1}{p^n}}]]}$.
\end{itemize}
\end{notation}

The next is our main result.

\begin{thm} \label{thm:main}
Consider the perfectoid field 
\begin{equation} \label{eqn:perf_field}
\cQ_r \wotimes_{\F_1} \F_p \cong \kappa
\end{equation} 
equipped with its valuation normalized in a way that $|t| = r < 1$ (cf. Proposition \ref{prop:perfectoid_fields} and Example \ref{exa:perfect_fields2}).
The following isomorphism of Fr\'echet bornological algebras holds
\[ \cQ_r^+ \wotimes_{\F_1} (\Q_p)_{(0, \infty)} \cong A^+_\kappa, \]
where $(\Q_p)_{(0,\infty)}$ denotes the Fr\'echet-like structure on the $p$-adic numbers defined in Example \ref{exa:born_rings}.
\end{thm}
\begin{proof}
On the one hand,
\begin{equation} \label{eqn:proj_lim}
 A^+_\kappa = \limpro_{\rho \to \infty} (W_p(\cQ_r^+ \wotimes_{\F_1} \F_p)[\frac{1}{p}], |\cdot|_{\rho, \rho^{-1}})
 \end{equation}
where $|\cdot|_{\rho,\rho^{-1}}$ are the norms of equation (\ref{eqn:FF_norm3}) (notice that we are using equation (\ref{eqn:perf_field}) for identifying $\cQ_r^+ \wotimes_{\F_1} \F_p$ with $\kappa^\circ$). As bare bornological sets we can write
\[ W_p(\cQ_r^+ \wotimes_{\F_1} \F_p, |\cdot|_{\rho, \rho^{-1}})[\frac{1}{p}] \cong {\prod_{n \in \N}}^{\le 1} [\cQ_r^+ \wotimes_{\F_1} \F_p]_{p^{- \rho^{-1} n}} \times {\prod_{n \in \N}}^{\le 1} [\cQ_r^+ \wotimes_{\F_1} \F_p]_{p^{-\rho n}} \]
which by the cofinality result of Theorem \ref{thm:pro_cofinal}, can be replaced by 
\[ {\coprod_{n \in \N}}^{\le 1} [\cQ_r^+ \wotimes_{\F_1} \F_p]_{p^{-\rho^{-1} n}} \times {\coprod_{n \in \N}}^{\le 1} [\cQ_r^+ \wotimes_{\F_1} \F_p]_{p^{- \rho n}} \]
without changing the projective limit (\ref{eqn:proj_lim}). Then,
\begin{equation} \label{eqn:one_hand}
 {\coprod_{n \in \N}}^{\le 1} [\cQ_r^+ \wotimes_{\F_1} \F_p]_{p^{-\rho^{-1} n}} = {\coprod_{n \in \N}}^{\le 1} [{\coprod_{x \in \cQ_r^+}}^{\le 1} [\F_p]_{|x|}]_{p^{-\rho^{-1} n}} \cong \underset{n \in \N, x \in \cQ_r^+}{{\coprod}^{\le 1}} [\F_p]_{|x| p^{-\rho^{-1} n}},
\end{equation}
and also 
\begin{equation*}
 {\coprod_{n \in \N}}^{\le 1} [\cQ_r^+ \wotimes_{\F_1} \F_p]_{p^{-\rho n}} \cong \underset{n \in \N, x \in \cQ_r^+}{{\coprod}^{\le 1}} [\F_p]_{|x|p^{-\rho n}},
\end{equation*}
On the other hand,
\[ \cQ_r^+ \wotimes_{\F_1} (\Q_p)_{(0, \infty)} = \cQ_r^+ \wotimes_{\F_1} (\limpro_{\rho < \infty} (\Q_p, \|\cdot\|_{\rho^{-1}, \rho})) \cong \limpro_{\rho < \infty}\cQ_r^+ \wotimes_{\F_1} (\Q_p, \|\cdot\|_{\rho^{-1}, \rho}) \]
where we used Proposition \ref{prop:born_base_change_ban_proj_lim} for the commutation of the base change functor with projective limits. Then
\[ \cQ_r^+ \wotimes_{\F_1} (\Q_p, \|\cdot\|_{\rho, \rho^{-1}}) \cong {\coprod_{x \in \cQ_r^+}}^{\le 1} [(\Q_p, \|\cdot\|_{\rho, \rho^{-1}})]_{|x|} \cong {\coprod_{x \in \cQ_r^+}}^{\le 1} [{\prod_{n \in \N}}^{\le 1} [\F_p]_{p^{-\rho^{-1} n}} \times {\prod_{n \in \N}}^{\le 1} [\F_p]_{p^{-\rho n}} ]_{|x|} \]
again by the cofinality argument of Theorem \ref{thm:pro_cofinal} we can replace the last bornological module with
\[  {\coprod_{x \in \cQ_r^+}}^{\le 1} [{\coprod_{n \in \N}}^{\le 1} [\F_p]_{p^{-\rho^{-1} n}} \times {\coprod_{n \in \N}}^{\le 1} [\F_p]_{p^{-\rho n}}]_{|x|} \cong \underset{x \in \cQ_r^+, n \in \N}{{\coprod}^{\le 1}} [\F_p]_{p^{-\rho^{-1} n}|x|} \times \underset{x \in \cQ_r^+, n \in \N}{{\coprod}^{\le 1}} [\F_p]_{p^{-\rho n}|x|} \]
without changing the value of the projective limit we are computing. By equation (\ref{eqn:one_hand}) this last module is isomorphic with $W_p(\cQ_r^+ \wotimes_{\F_1} \F_p)[\frac{1}{p}]$.

This shows that the underlying Fr\'echet spaces are isomorphic. The fact that this isomorphism is also an isomorphism of algebras follows easily by noticing that 
\[ \cQ_r^+ \wotimes_{\F_1} (\Q_p)_{(0, \infty)} \cong (\cQ_r^+ \wotimes_{\F_1} (\Z_p)_{(0, \infty)} )[\frac{1}{p}] \]
and 
\[ \cQ_r^+ \wotimes_{\F_1} (\Z_p)_{(0, \infty)}  \]
is a strict $p$-ring whose quotient is $\cQ_r^+ \wotimes_{\F_1} \F_p$ (cf. the beginning of section 2 of \cite{CD} for the notion of strict $p$-ring) and this property characterizes the ring $W_p(\cQ_r^+ \wotimes_{\F_1} \F_p)$. Clearly, the ring isomorphism also identifies the underlying bornological modules.
\end{proof}

\begin{cor} \label{cor:main}
	Theorem \ref{thm:main}, can be restated, in a dual fashion, in the category of affine analytic spaces, using perfect disks (cf. Definition \ref{defn:perfect_disks}):
	\[ P D_r^\bullet \wotimes_{\F_1}^\L (\Q_p)_{(0,\infty)} \cong P D_r^\bullet \wotimes_{\F_1} (\Q_p)_{(0,\infty)} \cong \Spec(A^+_\kappa). \]
\end{cor}
\begin{proof}
	The fact that all discrete objects of $\bComm_{\F_1}$ are fibrant and cofibrant in $\bsComm_{\F_1}$ follows easily from the description of the model structure on $\bsComm_{\F_1}$ given so far. Therefore, the derived tensor product coincides with the underived one as claimed.
\end{proof}

Theorem \ref{thm:main}, can be stated very easily also for the monoid $\cZ_r[\frac{1}{p}]^+$ as follows.

\begin{cor} \label{cor:main_2} We have the isomorphism of bornological algebras
	\[ \cZ_r[\frac{1}{p}]^+ \wotimes_{\F_1} (\Q_p)_{(0, \infty)} \cong A^+_{\kappa'} \]
	and of affine analytic spaces
	\[ p P D_r^\bullet \wotimes_{\F_1}^\L (\Q_p)_{(0, \infty)} \cong p P D_r^\bullet \wotimes_{\F_1} (\Q_p)_{(0, \infty)} \cong \Spec(A^+_{\kappa'}). \]
\end{cor}
\begin{proof}
The proof goes on the same lines of the proofs given so far.
\end{proof}

\begin{rmk} \label{rmk:general}
Theorem \ref{thm:main} can be generalized in various ways, by replacing the field $(\Q_p)_{(0, \infty)}$ with other non-Archimedean fields equipped with suitable Fr\'echet-like structures or by replacing $\cQ_r^+$ and $\cZ_r[\frac{1}{p}]^+$ by other $p$-divisible Banach monoids. Although such a more general study is possible, we bound ourselves in discussing only the key examples given so far in this work.
\end{rmk}

Last corollary to Theorem \ref{thm:main} is straightforward but nonetheless important.

\begin{cor} \label{cor:main_3} 
	Let $0 < r < 1$, then
	\[ \cQ_r \wotimes_{\F_1} (\Q_p)_{(0, \infty)} \cong A_\kappa. \]
	The same result is true also for $\cZ_r[\frac{1}{p}]$ in place of $\Q$ and for the dual results in the category of affine analytic spaces.
\end{cor}
\begin{proof}
	The proof of the fact that $\cQ_r \wotimes_{\F_1} (\Q_p)_{(0, \infty)}$ and $A_\kappa$ are isomorphic as bornological modules goes on the same lines of the proof of Theorem \ref{thm:main}. The fact that the rings are also isomorphic is consequence of the fact that both rings can be described as the completion of $W_p(K^\circ)[\frac{1}{p}][\frac{1}{[\omega]}]$ with respect to the same bornological structure. 
\end{proof}

Theorem \ref{thm:main} gives the following picture for the affine space associated with $A_\kappa$.

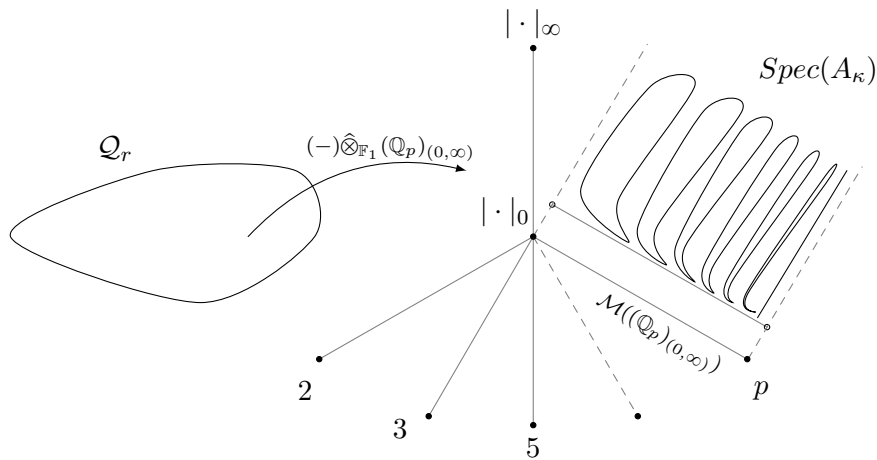
\begin{figure}[H] \label{fig:3}
    \begin{tikzpicture}[scale=2.5,cap=round,>=latex]

\draw [black] plot [smooth cycle] coordinates {(-2.75,0) (-2,0.35) (-1.25,0.35) (-1.15,0) (-1.75,-0.35)};

\draw [black] plot [smooth cycle] coordinates {(0.25,0.25) (0.62,0.8) (0.85,0.8) (0.45, 0.2) (0.5,-0.03)};   

\draw [black] plot [smooth cycle] coordinates {(0.55,0.05) (0.9,0.67) (1.1,0.67) (0.65,0.05) (0.7,-0.15)};   

\draw [black] plot [smooth cycle] coordinates {(0.75,-0.05) (1.1,0.57) (1.25,0.57) (0.80,-0.05) (0.85,-0.24)};   

\draw [black] plot [smooth cycle] coordinates {(0.9,-0.15) (1.28,0.47) (1.38,0.47) (0.95,-0.15) (0.95,-0.3)};   
 
 \draw [black] plot [smooth cycle] coordinates {(1.02,-0.23) (1.415,0.39) (1.485,0.39) (1.05,-0.23) (1.05,-0.35)};
 
 \draw [black] plot [smooth cycle] coordinates {(1.12,-0.3) (1.535, 0.32) (1.565,0.32) (1.13,-0.3) (1.17,-0.4)};
 
 \draw[black] (1.18cm, -0.43cm) -- (1.65cm, 0.35cm);

\draw[color=black, ->] (-1.5, 0) to [bend left=30] (-0.35,0.35);
         
         \draw[gray] (1.23cm, -0.48cm) -- (0.1cm, 0.17cm);
         \draw[black] (1.23cm, -0.48cm) circle (0.4pt);
         \draw[black] (0.1cm, 0.17cm) circle (0.4pt);
         \draw[gray, dashed] (1.13cm, -0.65cm) -- (1.73cm, 0.37cm);
         \draw[gray, dashed] (0cm, 0cm) -- (0.6cm, 1.02cm);       
         
         \draw[gray] (0cm,0cm) -- (90:1cm);
         \filldraw[black] (90:1cm) circle(0.4pt);
         
         \draw[gray] (0cm,0cm) -- (210:1.3cm);
         \filldraw[black] (210:1.3cm) circle(0.4pt);
         
         \draw[gray] (0cm,0cm) -- (240:1.1cm);
         \filldraw[black] (240:1.1cm) circle(0.4pt);
         
         \draw[gray] (0cm,0cm) -- (270:1cm);
         \filldraw[black] (270:1cm) circle(0.4pt);
         
         \draw[gray, dashed] (0cm,0cm) -- (300:1.1cm);
         \filldraw[black] (300:1.1cm) circle(0.4pt);
         
         \draw[gray] (0cm,0cm) -- (330:1.3cm);
         \filldraw[black] (330:1.3cm) circle(0.4pt);

        \filldraw[black] (0cm:0cm) circle(0.4pt);

       \draw (0cm,1cm) node[above=0pt] {$|\cdot|_\infty$}
             (-1.2cm,-0.7cm)  node[below=1pt] {$2$}
             (1.2cm,-0.7cm)  node[below=1pt] {$p$}
             (0cm,-1cm)  node[below=1pt] {$5$}
             (-0.7cm,-0.9cm)  node[below=1pt] {$3$}
             (0.7cm,-0.9cm)  node[below=1pt] {}
             (-0.15cm, 0cm) node[above] {$|\cdot|_0$}
             (-0.75cm, 0.35cm) node[above] {\begingroup\makeatletter\def\f@size{8}\check@mathfonts$(-) \wotimes_{\F_1} (\Q_p)_{(0,\infty)}$\endgroup}
             (1.5cm, 0.75cm) node[above] {$Spec(A_\kappa)$}
             (-2.2cm, 0.35cm) node[above] {$\cQ_r$}
             (0.60cm, -0.62cm) node[above, rotate=-30] {\begingroup\makeatletter\def\f@size{8}\check@mathfonts
$\cM((\Q_p)_{(0, \infty)})$\endgroup};
    \end{tikzpicture}
    
    \caption{The space $\Spec(A_\kappa)$ as an analytic space over $\cM((\Q_p)_{(0, \infty)})$ is fibered over an open subset of $\cM(\Z)$}
\end{figure}

Hence, using Theorem \ref{thm:main} and Corollary \ref{cor:main_3}, we can recover (under suitable hypothesis) the affine covering of the Fargues-Fontaine curve as the base change of an analytic space defined over $\F_1$. Now we would like to recover the action of the Frobenius on it from an action on $\Q$ (or better $\Q^+$ as a first step), \ie an action which is defined over $\F_1$. For doing this we need some lemmata. The problem we have to address is that, even if the multiplication by $p$ is a bounded bijection of $\cQ_r^+$ over itself, because 
\[ r^{p q} \le r^q, \ \ \forall q \in \Q^+ \]
if $r < 1$, the inverse of this morphism, \ie multiplication by $p^{-1}$, is not a bounded map for the same reason, \ie 
\[ r^{p^{-1} q} \ge r^q, \ \ \forall q \in \Q^+. \]
In order to fix this issue, \ie for having multiplication by $p$ is automorphism, we consider on $\Q^+$ the Fr\'echet-like structure of open disk of radius $1$. We need to modify $\cQ_r^+$ for ensuring that the action of the Frobenius is bounded while at the same time not changing its base change as described by Theorem \ref{thm:main}. We now explain how to do it. Later on we will explain for which structure on $\Q$ this action will be bounded.

\begin{lemma} \label{lemma:action_1}
The map of commutative normed monoids $\N_r \to \N_{r^{p^{-1}}}$ given by $n \mapsto p n$ is an isometry onto its image.
\end{lemma}
\begin{proof}
By definition
\[ |p n|_{r^{p^{-1}}} = r^{p n p^{-1}} = r^n = |n|_r. \]
\end{proof}

\begin{lemma}  \label{lemma:action_2}
The map of commutative normed monoids $\cQ_r^+ \to \cQ_{r^{p^{-1}}}^+$ given by $n \mapsto p n$ is an isometry and onto.
\end{lemma}
\begin{proof}
Lemma \ref{lemma:action_1} implies that the action is an isometry and it is onto because $\Q^+$ is divisible.
\end{proof}

\begin{prop} \label{prop:frechet_action_Q}
Consider the open perfect disk of radius $1$
\[ pD^\circ_1 = \Spec(\limpro_{r < 1} \cQ_r^+). \]
Then, multiplication by $p$ is an automorphism of $pD^\circ_1$.
\end{prop}
\begin{proof}
On the one hand, multiplication by $p$ is a bounded map because it can be written as a (bounded) morphism of the projective system defining $\underset{r < 1}\limpro \cQ_r^+$. On the other hand, multiplication by $p^{-1}$ is an isomorphism of $\cQ^+_{r^{p^{-1}}}$ with $\cQ^+_r$, by Lemma \ref{lemma:action_2}. Hence, we have a morphism of system
\[ \limpro_{r < 1} \cQ_r^+ \to \limpro_{r^{p^{-1}} < 1} \cQ_{r^{p-1}}^+ \cong \limpro_{r < 1} \cQ_r^+ \]
which is by its definition the inverse of multiplication by $p$.
\end{proof}

\begin{lemma}  \label{lemma:action_3}
	Consider on $\Q$ the following norm
	\[ |q|_{r, 1 - r} = \begin{cases}
	r^q \ \text{ if } q \ge 1 \\
	(1 - r)^q \ \text{ if } q \le 1 
	\end{cases}, \ q \in \Q, \ \ r < 1. \]
	\ie the normed group $\cQ_{r, 1 - r}$.
	Then for any $0 < r' < r < 1$ the identity map $\cQ_{r, 1 - r} \to \cQ_{r', 1 - r'}$ is bounded and the multiplication by $p^{-1}$ is an isometry $\cQ_{r, 1 - r} \to \cQ_{r^p, (1 - r)^{-p}}$.
\end{lemma}
\begin{proof}
	The Lemma is directly deduced from the analogous results for $\Q^+$ applied twice, once for the positive and once for the negative rational numbers. Notice that as $r \to 1$, $1 - r \to 0$ and that for $r' < r$ one has that $1 - r' > 1 - r$.
\end{proof}

The direct corollary to this lemma is that we get a bounded action by powers of $p$ also on $\Q$, equipped with the correct bornological structure.

\begin{cor} \label{cor:frechet_action_Q}
	Consider the bornological group
	\[ \limpro_{r < 1} \cQ_{r, 1 - r}. \]
	Then, multiplication by $p$ is an automorphism on it.
\end{cor}
\begin{proof}
	Direct consequence of Lemma \ref{lemma:action_3} using the same reasoning of Proposition \ref{prop:frechet_action_Q}.
\end{proof}

Now that we have identified an example of bornological monoid on which $\Z$ acts by the powers of $p$ we can consider the quotient by this action.

\begin{defn} \label{defn:FF_curve_F_1}
Let $M$ be a bornological monoid on which the action of $\Z$ by the powers of $p$ is bounded. We define the \emph{Fargues-Fontaine curve over $\F_1$ defined by $M$ relative to $p$} as 
\[ \FF_{\F_1}^p(M) =  \left [ \frac{M}{\Z} \right ] = \hocolim \left ( \ldots \Z \times^\R \Spec(M) \stackrel{\to}{\rightrightarrows} \Z \times^\R \Spec(M) \rightrightarrows \Spec(M) \right ) \]
where $\Z$ acts on $M$ by $x \mapsto p^n x$.
\end{defn}

Notice that the homotopy colimit that defines $\FF_{\F_1}^p(M)$ is computed in the category of analytic $\infty$-stacks over $\F_1$.

\begin{thm}
	$\FF_{\F_1}^p(M)$ is an analytic $\infty$-stack over $\F_1$, \ie an $\infty$-stack over $\infty-\bAff_{\F_1}$ with respect to the homotopy Zariski topology.
\end{thm}
\begin{proof}
	Notice that at each element of the diagram
	\[ \ldots \Z \times^\R \Spec(M) \stackrel{\to}{\rightrightarrows} \Z \times^\R \Spec(M) \rightrightarrows \Spec(M) \]
	is a representable object of $\bAff_{\F_1}$ (underived). Since the homotopy Zariski topology is sub-canonical, it is a general result that the homotopy colimit is an $\infty$-stack.
\end{proof}

The next step is to compute the base change of $\FF_{\F_1}^p(M)$ with $\Q_p$ for recovering the actual Fargues-Fontaine curve of $p$-adic Hodge Theory. We need some preparatory work for doing that.

The key result is the following lemma.

\begin{lemma} \label{lemma:key}
	Let $K$ be either $\R$ or $\Q_p$, $|\cdot|$ denote the canonical absolute value on $K$ and any $0 < s < \infty$, then
	\begin{equation} \label{eq:equiv_directions}
	(\limpro_{r < 1} \cQ_{r, 1 - r}) \wotimes_{\F_1} (K, |\cdot|^s) \cong \cQ_R \wotimes_{\F_1} (K)_{(0, \infty)},
	\end{equation}
	as bornological ring, for any fixed $0 < R < 1$.
\end{lemma}
\begin{proof}
	The right hand side of (\ref{eq:equiv_directions}) has been described by Corollary \ref{cor:main_3} as a Fr\'echet completion of $\bW_\kappa = W_p(\kappa^\circ)[\frac{1}{p}][\frac{1}{[\omega]}]$, were $[\omega]$ is a Teichm\"uller lift of a non-invertible element of $\kappa^\circ$. It is clear, using the same computations of Theorem \ref{thm:main}, that also the left hand side of (\ref{eq:equiv_directions}) can be described as a Fr\'echet completion of $\bW_\kappa$, hence one only has to check that the Fr\'echet structures agree.
	
	Since by Proposition \ref{prop:nuclear_perfect_disks} one has that
	\[ (\limpro_{r < 1} \cQ_{r, 1 - r}) \wotimes_{\F_1} (K, |\cdot|^s) \cong (\limpro_{r < 1}\cQ_{r, 1 - r}) \wotimes_{\F_1}^{\sup} (K, |\cdot|^s) \]
	then the tensor product commutes with the projective limit (cf. Proposition \ref{prop:sup_base_change}). Then we can compute 
	\[ \limpro_{r < 1} (\cQ_{r, 1 - r} \wotimes_{\F_1}^{\sup} (K, |\cdot|^s)). \]
	and compare it with the projective limit
	\[ \limpro_{s \to \infty} \cQ_R \wotimes_{\F_1} (K, |\cdot|_{s, s^{-1}}) \cong \limpro_{s \to \infty} \cQ_R \wotimes_{\F_1}^{\sup} (K, |\cdot|_{s, s^{-1}}). \]
	There is a linear bijection
	\begin{equation} \label{eqn:psi}
	\psi_{r_1, r_2}: \cQ_{r_1, r_2} \wotimes_{\F_1}^{\sup} (K, |\cdot|^s) \to \cQ_R \wotimes_{\F_1}^{\sup} (K)_{(s_1, s_2)}
	\end{equation} 
	for a particular choice of $s_1$ and $s_2$, given as follows. The left hand side of (\ref{eqn:psi}) can be described as the space of $\Q$-sequences
	\[ \{ (a_q)_{q \in \Q} | a_q \in K, \max \{ \sup_{q \in \Q} |a_q|^s r_1^q, \sup_{q \in \Q} |a_q|^s r_2^q \} < \infty \} \]
	whereas the left hand side as the space of $\Q$-sequences
	\[ \{ (a_q)_{q \in \Q} | a_q \in K, \max \{ \sup_{q \in \Q} |a_q|^{s_1} R^q, \sup_{q \in \Q} |a_q|^{s_2} R^q \} < \infty \}. \]
	We can write the equivalences
	\[ \sup_{q \in \Q} |a_q|^s r_1^q < \infty \Leftrightarrow \inf_{q \in \Q} R^{ s \log_R |a_q| + q \log_R r_1} > 0 \Leftrightarrow \sup_{q \in \Q} s \log_R |a_q| + q \log_R r_1 < \infty \]
	because $R < 1$, so
	\[ \sup_{q \in \Q} s \log_R |a_q| + q \log_R r_1 < \infty \Leftrightarrow \sup_{q \in \Q} \frac{s}{\log_R r_1} \log_R |a_q| + q  < \infty \Leftrightarrow \sup_{q \in \Q} |a_q|^{\frac{s}{\log_R r_1}} R^q < \infty. \]
	In the same way
	\[ \sup_{q \in \Q} |a_q|^s r_2^q < \infty \Leftrightarrow \sup_{q \in \Q} |a_q|^{\frac{s}{\log_R r_2}} R^q < \infty. \]
	So, the bijection $\psi_{r_1, r_2}$ is realized by the identity map when $s_1 = \frac{s}{\log_R r_1}$ and $s_2 = \frac{s}{\log_R r_2}$. Notice now that since $r_1 < R < r_2$
	\[ \max \{ \sup_{q \in \Q} |a_q|^s r_1^q, \sup_{q \in \Q} |a_q|^s r_2^q \} = \max \{ \sup_{q \in \Q^+} |a_q|^s r_2^q, \sup_{q \in \Q^+} |a_{-q}|^s r_1^{-q} \} \]
	whence 
	\[ \sup_{q \in \Q^+} |a_q|^{s_1} R^q \le \sup_{q \in \Q^+} |a_q|^s r_2^q   \]
	if $|a_q| \le 1$ for all $q \in \Q^+$ and
	\[ \sup_{q \in \Q^+} |a_q|^{s_2} R^q \le \sup_{q \in \Q^+} |a_q|^s r_2^q  \]
	if $\exists q$ such that $|a_q| \ge 1$. And also 
	\[ \sup_{q \in \Q^+} |a_{-q}|^{s_2} R^{-q} \le \sup_{q \in \Q^+} |a_{-q}|^s r_1^{-q}  \]
	and if $\exists q$ such that $|a_{-q}| \ge 1$ and
	\[ \sup_{q \in \Q^+} |a_{-q}|^{s_1} R^{-q} \le \sup_{q \in \Q^+} |a_{-q}|^s r_1^{-q}   \]
	if $|a_{-q}| \le 1$ for all $q \in \Q^+$. Therefore, we get that the morphism $\psi_{r_1, r_2}$ is bounded.
	The Open Mapping Theorem for F-spaces (or even just for topological groups) implies that $\phi_{r_1, r_2}$ is a homeomorphism. 
	These level-wise homeomorphisms induce a morphism to the projective limit which is a homeomorphism of the Fr\'echet spaces 
	\[
	(\limpro_{r < 1} \cQ_{r, 1 - r}) \wotimes_{\F_1} (K, |\cdot|^s) \cong \cQ_R \wotimes_{\F_1} (K)_{(0, \infty)}. \]
	Hence these spaces are isomorphic as bornological spaces because $\Q_p$-Fr\'echet spaces are bornologically isomorphic if and only if they are linearly homeomorphic.
\end{proof}

The statement of Lemma \ref{lemma:key} needs to be explained more, mainly for its geometric interpretation.

\begin{rmk} \label{rmk:geometric_key_lemma}
	Notice that Lemma \ref{lemma:key} identifies $(\underset{r < 1}\limpro \cQ_{r, 1 - r}) \wotimes_{\F_1} (K, |\cdot|^s)$ and $\cQ_R \wotimes_{\F_1} (K)_{(0, \infty)}$ only as bornological rings but these two come naturally equipped with two different algebra structures. Indeed, as the first base change is over $(K, |\cdot|^s)$ then is it canonically endowed with a structure of $(K, |\cdot|^s)$-algebra whereas the second is equipped with a structure of $(K)_{(0, \infty)}$-algebra. Hence, although these Fr\'echet algebras are shown to be isomorphic (which just mean homeomorphic) their spectra are not identical because the points of $(\underset{r < 1}\limpro \cQ_{r, 1 - r}) \wotimes_{\F_1} (K, |\cdot|^s)$ all lie above the point $\cM(K, |\cdot|^s)$ of $\cM(\Z)$ whereas the points of $\cQ_R \wotimes_{\F_1} (K)_{(0, \infty)}$ lie above the open segment $\cM((K)_{(0, \infty)})$ of $\cM(\Z)$.
	
	Put in another way, the algebras $(\underset{r < 1}\limpro \cQ_{r, 1 - r}) \wotimes_{\F_1} (K, |\cdot|^s)$ and $\cQ_R \wotimes_{\F_1} (K)_{(0, \infty)}$ are homeomorphic but not isometric and the canonical morphisms $(K, |\cdot|^s) \to (\underset{r < 1}\limpro \cQ_{r, 1 - r}) \wotimes_{\F_1} (K, |\cdot|^s)$ and $(K)_{(0, \infty)} \to \cQ_R \wotimes_{\F_1} (K)_{(0, \infty)}$ are isometric embeddings, whereas the morphisms $(K)_{(0, \infty)} \to (\underset{r < 1}\limpro \cQ_{r, 1 - r}) \wotimes_{\F_1} (K, |\cdot|^s)$ and $(K, |\cdot|^s) \to (\underset{r < 1}\limpro \cQ_{r, 1 - r}) \wotimes_{\F_1} (K, |\cdot|^s)$ are just continuous embeddings. Thus, if one asks for a morphism $(\underset{r < 1}\limpro \cQ_{r, 1 - r}) \wotimes_{\F_1} (K, |\cdot|^s) \to L$, where $L$ is a valued field, to be a contraction with respect to the metrics then one gets that $L$ must be a valued field extension of $(K, |\cdot|^s)$ whereas the same condition for $\cQ_R \wotimes_{\F_1} (K)_{(0, \infty)} \to L$ implies that $L$ must be a valued field extension of $(K, |\cdot|^{s'})$ with $0 < s' < \infty$.
\end{rmk}

With the previous lemmata, the definition of the Fargues-Fontaine curve over $\F_1$ and Theorem \ref{thm:main} we can recover the Fargues-Fontaine curves associated to the perfectoid fields $\kappa$ and $\kappa'$ as base changes of the Fargues-Fontaine curves over $\F_1$.

\begin{thm} \label{thm:FF}
For any $0 < s < \infty$ the following isomorphisms of analytic spaces hold
\[ \FF_\kappa \cong \FF_{\F_1}(\limpro_{r < 1} \cQ_{r, 1 - r}) \wotimes_{\F_1}^\L (\Q_p, |\cdot|_p^s), \]
\[ \FF_{\kappa'} \cong \FF_{\F_1}(\limpro_{r < 1} \cZ_{r, 1 - r}[\frac{1}{p}]) \wotimes_{\F_1}^\L (\Q_p, |\cdot|_p^s), \]
\end{thm}
\begin{proof}
Using Proposition \ref{prop:base_changes} we can write
\[ \FF_{\F_1}(\limpro_{r < 1} \cQ_{r, 1 - r}) \wotimes_{\F_1}^\L (\Q_p, |\cdot|_p^s) \cong \]
\[ \hocolim \left ( \ldots \stackrel{\to}{\rightrightarrows} (\Z \times^\R \Spec(\limpro_{r < 1} \cQ_{r, 1 - r})) \wotimes_{\F_1}^\L (\Q_p, |\cdot|_p^s) \rightrightarrows \Spec(\limpro_{r < 1} \cQ_{r, 1 - r}) \wotimes_{\F_1}^\L (\Q_p, |\cdot|_p^s) \right ). \]
Applying Lemma \ref{lemma:key} and Corollay \ref{cor:main_3}, we see that this analytic stack represents an action of $\Z$ on $\Spec(\limpro_{r < 1} \cQ_{r, 1 - r}) \wotimes_{\F_1}^\L (\Q_p, |\cdot|_p^s) \cong \Spec(A_\kappa)$. Even more, this action is precisely the action on $\Spec(A_\kappa)$ given by the (arithmetic) Frobenius and it is given at the level of algebras by 
\[ \sum_{n \in \Z} a_n p^n \mapsto \sum_{n \in \Z} (a_n)^{p^m} p^n \]
for each $m \in \Z$. Hence the quotient by this action is precisely $\FF_\kappa$ which is an analytic space because the action is totally discontinuous and has no fixed point.

The same reasoning works also for $\FF_{\kappa'}$.
\end{proof}

\subsection{The base change of the Fargues-Fontaine curve to the $p$-adic fields}

In this section we give an description of the curve $\FF_\kappa$ as an analytic space (the same description can be given for $\FF_{\kappa'}$).
So far we showed that the affine covering of $\FF_\kappa$ can be interpreted as analytic spaces over the bornological field $\cM((\Q_p)_{(0,\infty)})$ (see Figure \ref{fig:3}). In this section we describe how the Frobenius acts on this analytic space. It is convenient to work in the framework of usual non-Archimedean geometry, therefore we use the non-Archimedean base change over $\F_1$ (cf. Remark \ref{rmk:non-arch_base_change}) for describring the fibers of $A_\kappa$ over $(\Q_p, |\cdot|_p^s)$. We denote this functor $(\cdot) \wotimes_{\F_1}^{\na} R$, for any Banach ring $R$. Notice that
\[ A_\kappa \wotimes_{\F_1}^\na (\Q_p)_{(0, \infty)} \cong A_\kappa \wotimes_{\F_1} (\Q_p)_{(0, \infty)}. \]

The Frobenius acts on $A_\kappa$ by
\[ \sum_{n \in \Z} a_n p^n \mapsto \sum_{n \in \Z} a_n^p p^n. \]
The fiber of $\Spec(A_\kappa)$ over $\cM((\Q_p, |\cdot|_p^s)) \in \cM(\Z)$ is given by
\[ \cM(A_\kappa \wotimes_{\Z}^{\na} (\Q_p, |\cdot|_p^s)) \cong \cM(A_{\kappa, [s,s]})  \]
using notation of section \ref{sec:recall}, because
\[ A_\kappa \wotimes_{\Z}^{\na} (\Q_p, |\cdot|_p^s) \cong (\cQ_R \wotimes_{\F_1}^{\na} \Z) \wotimes_{\Z}^{\na} (\Q_p, |\cdot|_p^s) \cong \cQ_R \wotimes_{\F_1}^{\na} (\Q_p, |\cdot|_p^s) \cong A_{\kappa, [s,s]}. \]
It is well-know (and easy to check) that the Frobenius induces an isomorphism
\begin{equation} \label{eq:frobenius_iso}
A_{\kappa, [s,s]} \stackrel{\cong}{\to} A_{\kappa, [p s, p s]}.
\end{equation} 
Using the same calculation of Lemma \ref{lemma:key}, one can see that the isomorphism (\ref{eq:frobenius_iso}) is the identity map. This is due to the fact that the elements of $\kappa$ are of the form
\[ a = \sum_{q \in \Q} b_q t^q, \ \ b_q \in \F_p \]
and hence the action of the Frobenius is non-trivial only on $t$. Of course, one can consider in place of $\kappa$, for example, the field $\what{\F_q((t^{\frac{1}{n^\infty}}))}$, with $q = p^f$, and in this case the isomorphism (\ref{eq:frobenius_iso}) is non-trivial. We will consider these kind of perfectoid fields in a future work.

So, in the specific case of $\kappa$, the action of the Frobenius reduces to the action which identically identifies the fiber $A_{\kappa, [s,s]}$ with the fiber $A_{\kappa, [p s, p s]}$. Hence, $\FF_\kappa$ can be seen as a space over the stack
\[ \left [ \frac{\cM((\Q_p)_{(0, \infty)})}{\Z} \right ] \cong \left [ \frac{\R}{\Z} \right ] \cong S^1, \]
where the action is give by $|\cdot|_p^\epsilon \mapsto |\cdot|_p^{p \epsilon}$. The next figure shows how we imagine $\FF_\kappa$. Therefore, the Frobenius action on $A_\kappa$ is herited from the action $|\cdot| \mapsto |\cdot|^p$ on $\cM(\Z)$.

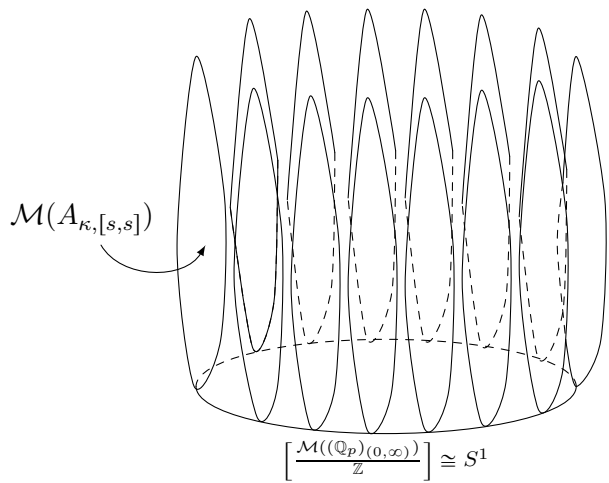
\begin{figure}[H] \label{fig:4}
    \begin{tikzpicture}[scale=2.5,cap=round,>=latex]

    \draw[dashed] (1,0) arc  [x radius = 1cm, y radius = 0.25cm, start angle = 0,
  end angle = 180];
    \draw[black] (1,0) arc  [x radius = 1cm, y radius = 0.25cm, start angle = 0,
  end angle = -180];
  
    \draw (0cm,-0.55cm) node[above=0pt] {\begingroup\makeatletter\def\f@size{8}\check@mathfonts$\left [ \frac{\cM((\Q_p)_{(0, \infty)})}{\Z} \right ]  \cong S^1$\endgroup};
    
    \draw (-1.6cm,0.75cm) node[above=0pt] {$\cM(A_{\kappa, [s,s]})$};

\draw [black] plot [smooth cycle] coordinates {(-1cm, 1.75 cm) (-0.85cm,1cm) (-0.87cm,0.25cm) (-1.01cm, 0.00cm) (-1.1cm, 0.75cm)};         

\draw [black] plot [smooth cycle] coordinates {(-0.70cm, 1.58 cm) (-0.55cm,0.83cm) (-0.57cm,0.01cm) (-0.68cm, -0.16cm) (-0.80cm, 0.58cm)};

\draw [black] plot [smooth cycle] coordinates {(-0.4cm, 1.54 cm) (-0.25cm,0.79cm) (-0.27cm,-0.01cm) (-0.40cm, -0.21cm) (-0.5cm, 0.54cm)};

\draw [black] plot [smooth cycle] coordinates {(-0.1cm, 1.53 cm) (0.05cm,0.78cm) (0.03cm,-0.02cm) (-0.10cm, -0.23cm) (-0.2cm, 0.53cm)};

\draw [black] plot [smooth cycle] coordinates {(0.2cm, 1.53 cm) (0.35cm,0.78cm) (0.33cm,-0.02cm) (0.20cm, -0.22cm) (0.1cm, 0.53cm)};

\draw [black] plot [smooth cycle] coordinates {(0.5cm, 1.57 cm) (0.65cm,0.82cm) (0.63cm,0.02cm) (0.50cm, -0.19cm) (0.4cm, 0.57cm)};

\draw [black] plot [smooth cycle] coordinates {(0.8cm, 1.62 cm) (0.95cm,0.89cm) (0.93cm,0.09cm) (0.80cm, -0.12cm) (0.7cm, 0.64cm)};

\draw [black] plot [smooth] coordinates {(0.93cm, 1.08cm) (1cm, 1.75 cm) (1.15cm,1cm) (1.13cm,0.25cm) (1cm, 0cm) (0.945cm, 0.25cm)};

\draw [black, dashed] plot [smooth] coordinates {(0.93cm, 1.08cm) (0.9cm, 0.75cm)  (0.945cm, 0.25cm)};

\draw [black] plot [smooth] coordinates {(0.7cm, 0.9cm) (0.8cm, 1.9 cm) (0.95cm,1.15cm) };

\draw [black, dashed] plot [smooth] coordinates {(0.95cm,1.15cm) (0.93cm,0.4cm) (0.80cm, 0.15cm) (0.7cm, 0.9cm)};

\draw [black] plot [smooth] coordinates {(0.4cm, 0.97cm) (0.5cm, 1.97 cm) (0.65cm,1.22cm) };

\draw [black, dashed] plot [smooth] coordinates {(0.65cm,1.22cm) (0.63cm,0.47cm) (0.50cm, 0.22cm) (0.4cm, 0.97cm)};

\draw [black] plot [smooth] coordinates {(0.1cm, 1cm) (0.2cm, 2 cm) (0.35cm,1.25cm) };

\draw [black, dashed] plot [smooth] coordinates {(0.35cm,1.25cm) (0.33cm,0.5cm) (0.20cm, 0.25cm) (0.1cm, 1cm)};

\draw [black] plot [smooth] coordinates {(-0.2cm, 1cm) (-0.1cm, 2 cm) (0.05cm,1.25cm)};

\draw [black, dashed] plot [smooth] coordinates {(0.05cm,1.25cm) (0.03cm,0.5cm) (-0.10cm, 0.25cm) (-0.2cm, 1cm)};

\draw [black] plot [smooth] coordinates {(-0.52cm, 0.99cm) (-0.42cm, 1.99 cm) (-0.27cm, 1.24cm) };  

\draw [black, dashed] plot [smooth] coordinates {(-0.27cm, 1.24cm) (-0.29cm, 0.51cm) (-0.42cm, 0.24cm) (-0.52cm, 0.99cm)};

\draw [black] plot [smooth] coordinates {(-0.82cm, 0.95cm) (-0.72cm, 1.95 cm) (-0.57cm, 1.2cm)};

\draw [black, dashed] plot [smooth] coordinates {(-0.57cm, 1.2cm) (-0.59cm, 0.45cm) (-0.7cm, 0.2cm) (-0.82cm, 0.95cm)};

\draw [black] plot [smooth] coordinates {(-0.57cm, 1.2cm) (-0.59cm, 0.45cm) (-0.7cm, 0.2cm) (-0.82cm, 0.95cm)};

\draw[color=black, ->] (-1.5cm, 0.75cm) to [bend left=-60] (-0.95cm,0.75cm);

    \end{tikzpicture}
    
    \caption{The curve $\FF_\kappa$ as a curve over $\left [ \frac{\cM((\Q_p)_{(0, \infty)})}{\Z} \right ]  \cong S^1$}
\end{figure}

One can prove that the fiber of the fibration represented in Figure \ref{fig:4} are contractible (because they are a sort of non-Archimedean annulus), obtaining the following result.

\begin{thm}
The homotopy type of the analytic space associated to $\FF_\kappa$ is the one of a circle.
\end{thm}
\begin{proof}
This theorem can be proved using ideas similar to ones of the proof of Theorem 4.5 of \cite{Ked}. We defer a detailed proof to a future work.
\end{proof}

Beside these local pictures at primes $p$, there is a global one obtained by the base changing with $(\Z, |\cdot|_\infty)$. By doing this base change we obtain an affine analytic space over $\cM(\Z)$, $\Spec(\cQ_R \wotimes_{\F_1} (\Z, |\cdot|_\infty))$ which, on each $p$-adic branch of $\cM(\Z)$, gives the affine covering of $\FF_\kappa$. Notice that now $p$ is not fixed and varies with the branch. In the center of the tree, over the point $|\cdot|_0$ of $\cM(\Z)$ one has the spectrum of the algebra
\[ \cQ_R \wotimes_{\F_1} (\Q, |\cdot|_0) \cong (\Q(\Q_R), |\cdot|_r) \]
where $\Q(\Q_R)$ is the $\Q$-algebra of polynomial with rational exponents equipped with the norm
\[ |\sum_{i \in \Q} a_i t^i|_R = \sum_{i \in \Q} |a_i|_0 R^i = \sum_{i \in \Q} R^i. \]
Therefore, it makes sense to consider also the Archimedean fiber of $\Spec(\cQ_R \wotimes_{\F_1} (\Z, |\cdot|_\infty))$ which gives an analytic space over $\R$ (or $\C$). In the next section we will give a brief description of the Archimedean fiber of the Fargues-Fontaine curve over $\F_1$.

\subsection{The archimedean Fargues-Fontaine curve}

In this section we would like to briefly describe what happens when one base changes the Fargues-Fontaine curve over $\F_1$ the Archimedean branch of $\cM(\Z)$. In order to do this we start by describing the algebra $A_\R \cong (\underset{r < 1}\limpro \cQ_{r, 1 - r}) \wotimes_{\F_1} \R $. By Proposition \ref{prop:nuclear_perfect_disks} and Proposition \ref{prop:sup_base_change} we have that
\[ A_\R = \underset{r < 1}\limpro (\cQ_{r, 1 - r}) \wotimes_{\F_1} \R) \cong \underset{r < 1}\limpro {\coprod_{x \in \cQ_R}}^{\le 1} [(\R, \|\cdot\|_{\rho, 1 - \rho})]_{|x|_R}. \]
See Example \ref{exa:born_rings} (2) for the definition of the norms $\|\cdot\|_{\rho, 1 - \rho}$. So, we first describe what 
\[ {\coprod_{x \in \cQ_R}}^{\le 1} [(\R, \|\cdot\|_{\rho, 1 - \rho})]_{|x|_R} \]
is for a fixed $\rho \in \R$, with $0 < \rho < 1$. Since the algebra $A_\R$ is an algebra of analytic functions we can study the algebra
\[ A_\C \doteq A_\R \wotimes_\R \C = {\coprod_{x \in \cQ_R}}^{\le 1} [(\C, \|\cdot\|_{\rho, 1 - \rho})]_{|x|_R} \]
for having an easier geometric picture of its spectrum. Then, the spectrum of $A_\R$ is just the spectrum of $A_\C$ modulo complex conjugation.

There is an isomorphism of algebras 
\[ {\underset{x \in \cQ_{R_1, R_2}}{{\coprod}^{\le 1}}} [\C]_{|x|_R} \cong \underset{z \mapsto z^n}{{\limind}^{\le 1}} {\underset{x \in \cZ_{R_1, R_2}}{{\coprod}^{\le 1}}} [\C]_{\sqrt[n]{|x|_{R_1, R_2}}} \]
where the latter algebras (indexed by $\Z$) are algebras of analytic functions on the annuli of radii $R_1 < R_2$ which are $\ell^1$-summable on the boundary. The system morphisms of the direct limit are given by the maps $z \mapsto z^n$ which geometrically correspond to the $n$-fold \'etale covering of the annulus. Therefore
\[ {\underset{x \in \cQ_{R_1, R_2}}{{\coprod}^{\le 1}}} [\C]_{|x|_R} \]
is an algebra of analytic functions on the closed strip of the complex plane $S_{\log R_1, \log R_2}$, where we denote the strips
\[ S_{\a_1, \a_2} = \{ z \in \C | \a_1 \le \Re(z) \le \a_2 \} \]
which is the universal covering of the annulus by the exponential map. Therefore, we have proved the following proposition.

\begin{prop} \label{prop:complex_strip}
The algebra $A_\C$ is isomorphic to a closed sub-algebra of the algebra of analytic functions on the open left-half plane $\mathring{S}_{-\infty, 0}$ which are almost periodic in the imaginary direction.
\end{prop}
\begin{proof}
It is clear that all functions in $A_\C$ are almost periodic in the imaginary direction because the set of analytic functions which are almost periodic is given by the closure in $\cO(\mathring{S}_{-\infty, 0})$ of the set of functions
\[ \{ e^{\la z} \}_{\la \in \R} \subset \cO(\mathring{S}_{-\infty, 0}). \]
The sub-algebra generated by the closure of 
\[ \{ e^{\la z} \}_{\la \in \Q} \subset \cO(\mathring{S}_{-\infty, 0}) \]
is precisely $A_\C$.
\end{proof}

\begin{rmk}
Proposition \ref{prop:complex_strip} implies thath the space $A_\C$ can be interepreted as a space of Dirichlet series which are convergent on the left complex plane $\mathring{S}_{-\infty, 0}$. 
\end{rmk}

Next proposition describes the spectrum of $A_\C$.

\begin{prop} \label{prop:complex_spectrum}
There is a canonical homeomorphism
\[ \cM(A_\C) \cong \mathring{S}_{-\infty, 0}. \]
\end{prop}
\begin{proof}
Since $A_\C$ is a Fr\'echet algebra presented as a projective limit of Banach algebras
\[ A_\C = \limpro_{0 < r_1 < r_2 < 1} (\cQ_{r_2, r_1} \wotimes_{\F_1} \C) \]
where each morphism of the system has dense image, it is well known that (see section 2.5 of \cite{Bam})
\[ \cM(A_\C) = \limind_{0 < r_1 < r_2 < 1} \cM(\cQ_{r_2, r_1} \wotimes_{\F_1} \C). \]
Therefore it is enough to check the spectrum of the Banach algebras $\cQ_{r_2, r_1} \wotimes_{\F_1} \C$ is canonically homeomorphic to the closed strip $S_{\log r_2, \log r_1}$. As explained so far, we have that
\[ \cQ_{r_2, r_1} \wotimes_{\F_1} \C \cong \underset{z \mapsto z^n}{{\limind}^{\le 1}} {\underset{x \in \cZ_{r_1, r_2}}{{\coprod}^{\le 1}}} [\C]_{\sqrt[n]{|x|_{r_1, r_2}}} \]
and the algebra
\[ \cA_{r_1, r_2} = {\underset{x \in \cZ_{r_1, r_2}}{{\coprod}^{\le 1}}} [\C]_{\sqrt[n]{|x|_{r_1, r_2}}} \]
is the algebra of analytic functions on the annulus of radii $r_1 < r_2$ which are absolutely summable on the boundary. 

It is clear that $S_{\log r_2, \log r_1} \subset \cM(\cQ_{r_2, r_1} \wotimes_{\F_1} \C)$ because all functions in $\cQ_{r_2, r_1} \wotimes_{\F_1} \C$ can be evaluated at points of $S_{\log r_2, \log r_1}$ and these evaluation morphism give bounded homomorphisms to $\C$. Now, consider a point $x: \cQ_{r_2, r_1} \wotimes_{\F_1} \C \to \C$. Consider the algebraic direct limit of rings
\[ A = \underset{z \mapsto z^n}{\limind} {\underset{x \in \cZ_{r_1, r_2}}{{\coprod}^{\le 1}}} [\C]_{\sqrt[n]{|x|_{r_1, r_2}}}. \]
Since $A$ is endowed with a canonical morphism to $\cQ_{r_2, r_1} \wotimes_{\F_1} \C$, by composition the point $x$ induces a point of $\Max(A)$. But it is well known that
\[ \Max(A) = \underset{z \mapsto z^n}{\limpro} \Max({\underset{x \in \cZ_{r_1, r_2}}{{\coprod}^{\le 1}}} [\C]_{\sqrt[n]{|x|_{r_1, r_2}}}) \cong S_{\log r_2, \log r_1}. \]
Hence every point of the spectrum of $\cQ_{r_2, r_1} \wotimes_{\F_1} \C$ can be identified canonically and uniquely with a point of $S_{\log r_2, \log r_1}$.
\end{proof}

%
%
%

The action of the Frobenius on $A_\R = (\cQ_R \wotimes_{\F_1} (\R)_{(0, 1)})$ induces an automorphism. In this way, one gets an action of $\Z$ (and even of the multiplicate monoids $\Q_+$, by considering all Frobenii at once, and of $\R_+$, by continuity) on this half plane and the quotient stack by this action is what we call the \emph{Archimedean Fargues-Fontaine curve}. Since vector bundles on $\Spec(A_\R)$ are trivial, the stack $\left [ \frac{\Spec(A_\R)}{\R_+}\right ]$ classifies filtered $\R_+$-vector spaces. Therefore, one can imagine a relation between $\Spec(A_\R)$, the Twistor curve and Simpson's approach to complex Hodge theory.
We will devote a future work to explain in details this stack and its connection with Hodge Theory through Simpson's approach.

\subsection{A remark on Witt vectors} \label{sec:witt}
In this last section we propose some new constructions related to the ring of Witt vectors. The idea is that the ring of Witt vectors is often considered as a purely algebraic construction, whereas the rings to which it is applied usually have more structures, like the rings we worked with so far which have valuations. Moreover, it is often an important aspect of the theory of Witt vectors to study analytic structures over it, as it happens for instance in the theory of the Fargues-Fontaine curve. Therefore, in this section we try a first attempt to axiomatize and make functorial constructions which enrich the Witt vectors with some extra (analytic) structure.

A possible way for introducing the ring of $p$-typical Witt vectors is by following the work of Cuntz and Deninger from \cite{CD}. We recall that, if $A$ is a perfect $\F_p$-algebra, we can consider $A$ as a multiplicative monoid and in this way there is an exact sequence
\[ 0 \to I \to \Z[A] \to A \to 0 \]
where the map $\Z[A] \to A$ just realizes in $A$ the formal sums and $I$ is by definition the kernel of this map. Then
\begin{equation} \label{eq:witt}
 W_p(A) \cong \limpro_{n \in \N} \frac{\Z [A]}{I^n}. 
\end{equation}
Geometrically, equation (\ref{eq:witt}) realizes the ring of Witt vectors as the formal neighborhood of $\Spec(A)$ in $\Spec(\Z[A])$

If $A$ has a norm, then each term of the projective limit is equipped with a norm using by the canonical identification of $A$-modules
\[ \frac{\Z[A]}{I^n} \cong A^n. \]
In particular we choose the norm
\[ |(a_1, \ldots, a_n)| = \max_{i} |a_i| \]
on $A^n$. Therefore, we can compute the projective limit
\[ {\limpro_{n \in \N}}^{\le 1} \frac{\Z[A]}{I^n} \]
in the contracting category of Banach rings. This construction gives a Banach ring, but it is not yet as flexible as we want. Hence, we add a parameter by which we scale the norms.

\begin{rmk}
For the rest of this section we suppose that $A$ is non-Archimedean. It is not a necessary restriction. We do it just for relating and discussing more easily wel-known object whereas discussing the Archimedean version of the subsequent cosntruction will lead us too far in discussing new objects and in relating them with the known ones.
\end{rmk}

\begin{defn} \label{defn:semiring_analytic_Witt_vectors}
	Let $\a \in \R_+$, the \emph{semi-ring of $\a$-analytic $p$-typical Witt vectors} associated to $A$ is defined by
	\[ sW_{p, \a}(A) \doteq {\limpro_{n \in \N}}^{\le 1} \left ( \frac{\Z[A]}{I^n}, |\cdot|_{\a_n} \right ), \]
	where $|\cdot|_{\a_n}$ is the norm
	\[ |(a_1, \ldots, a_n)|_{\a_n} = \max_{i} |a_i| \a^i \] 
	when $\frac{\Z[A]}{I^n}$ is (canonically) identified with $A^n$.
\end{defn}

Notice that $sW_{p, \a}(A)$ is just a semi-ring if $\a \ge 1$ whereas it is a ring if $\a < 1$. This is easy to check because the element $(p-1, p-1, \ldots) \in A^\N$ is the inverse of the identity for the Witt vectors operations.
Nevertheless $sW_{p, \a}(A)$ seems to be an interesting object also when it fails to be a ring.

\begin{defn} \label{defn:ring_analytic_Witt_vectors}
	Let $\a \in \R_+$, the \emph{ring of $\a$-analytic $p$-typical Witt vectors} associated to $A$ is defined 
	\[ W_{p, \a}(A) = K(sW_{p, \a}(A)) \]
	where $K$ denotes the Grothendieck group of the abelian additive semi-group $sW_{p, \a}(A)$. The multiplication is extended by linearity.
\end{defn}

\begin{rmk}
The parameter $\a$ in Definition \ref{defn:ring_analytic_Witt_vectors} can be interpreted geometrically as follows. Put $\beta = \log_p \a$ and consider the Banach ring of power-series
\begin{equation} \label{eq:Witt_geometric}
 \Z \lt A \gt_\a \doteq \{ \sum_{a \in A} n_a a | \max_{a \in A} |a| |n_a|^\beta \to 0, n_a \in \Z \} \cong A \wotimes_{\F_1}^{\na} (\Z, |\cdot|_\infty^\beta).
\end{equation}
Notice that the notation $\max_{a \in A} |a| |n_a|^\beta \to 0$ just means that for all $\epsilon > 0$ there exists only finitely many $a \in A$ such that $|a| |n_a|^\beta > \epsilon$. So, the left-hand side of (\ref{eq:Witt_geometric}) coincides with the right hand side, by the definition of the right hand side, and it can be thought as the completion of the monoid algebra $\Z[A]$ with respect to the norm $\underset{a \in A}\max |a| |n_a|^\beta$. The algebra $\Z \lt A \gt_\a$ sits in the strictly short exact sequence 
\[ 0 \to I \to \Z \lt A \gt_\a \to A \to 0 \]
of Banach modules, which is an analogue of (\ref{eq:witt}). Then the Banach ring
\[ \frac{\Z \lt A \gt_\a}{I^n} \cong A^n \]
is endowed with the quotient norm
\[ |(a_1, \ldots, a_n)| = \max_{i} |a_i| (p^i)^\beta = \max_{i} |a_i| \a^i \]
because
\[ (p^i)^\beta = (p^i)^{\log_p \a} = p^{i \log_p \a} = p^{ \log_p \a^i} = \a^i \]
and because the elements of $\frac{\Z \lt A \gt_\a}{I^n}$ can be represented (\cf Remark after Proposition 2 of \cite{CD}) by elements of $\Z \lt A \gt_\a$ of the form
\[ \sum_{i = 1}^n a_i p^i, \ \ a_i \in A. \]

Therefore, $W_{p, \a}(A)$ can be thought as a sort of an analytic neighborhood of $\Spec(A)$ in $\Spec(\Z \lt A \gt_\a)$.
\end{rmk}

We can put last remark in a form of a proposition

\begin{prop}
For $A$ perfect normed $\F_p$-algebra and $\a \in \R_+$ we have that
\[ W_{p, \a}(A) \cong {\limpro_{n \in \N}}^{\le 1} \frac{A \wotimes_{\F_1}^{\na} (\Z, |\cdot|_\infty^\beta)}{I^n} \]
where $\beta = \log_p(\a)$ and $I$ is the kernel of $A \wotimes_{\F_1}^{\na} (\Z, |\cdot|_\infty^\beta) \to A$.
\end{prop}

Notice that the norm on $W_{p, \a}(A)$ extends to a norm on $W_{p, \a}(A) [\frac{1}{p}
]$ which can be completed by this norm.

\begin{prop}
The associations $A \mapsto W_{p, \a}(A)$ and $A \mapsto W_{p, \a}(A)[\frac{1}{p}
]$ define functors of $\bComm(\bBan_{\F_p}) \to \bComm(\bBan_\Z)$.
\end{prop}
\begin{proof}
Suppose that $f: A \to B$ is a bounded morphism between Banach $\F_p$-algebras, then there exists a $C > 0$ such that $|f(x)| \le C |x|$ for all $x \in A$. Hence, for all $n \in \N$ the map
\[ f^n: (A^n, |\cdot|_{\a_n}) \to (B^n, |\cdot|_{\a_n}) \]
is bounded by the same $C$ because
\[ |f^n((x_1, \ldots, x_n))|_{\a_n} = \max_i |f(x_i)| \a^i \le C \max_i |x_i| \a^i = C |(x_i)|_{\a^n}.  \]
\end{proof}

In the next examples we show how to get the many important examples of normed rings from Definition \ref{defn:ring_analytic_Witt_vectors}.

\begin{exa}
\ben
\item Suppose $A = \F_p$ with the trivial norm and $\a = \frac{1}{p}$, then $W_{p, \a}(A)$ is isometrically isomorphic to $\Z_p$ equipped with the $p$-adic norm. Clearly, by varying $0 < \a < 1$ one gets that $(\Z_p, |\cdot|_p^{-\log_p \a})$.
\item The previous example can be considered also with $\a = 1$ or $\a > 1$. In the first case one gets $\Z_p$ equipped with the trivial norm, whereas in the second example the only bounded sequences are the finite ones, hence one gets the ring $\Z$. 
\item The examples of rings of Witt vectors we have seen so far, related to the Fargues-Fontaine curve, can be recast using the functors $W_{p, a}$. Indeed, if $0 < \a < 1$, it is clear that
\[ \what{W_{p, \a}(\kappa^\circ)[\frac{1}{p}]}
 \cong A_{\kappa, [-\log_p(\a), -\log_p(\a)]}^+, \]
using notation of section \ref{sec:recall}.
\item Of course, for the functor $W_{p, \a}$ to make sense we are not bounded any more to consider perfectoid fields which must be presented as a base change from $\F_1$. Therefore, for any perfectoid field $K$ of characteristic $p$ we have that 
\[ \what{W_{p, \a}(K^\circ)[\frac{1}{p}]} \cong A_{K, [-\log_p(\a), -\log_p(\a)]}^+ \]
for $0 < \a < 1$.
\een 
\end{exa}

Therefore, $W_{p, \a}$ define a family of functors which relate with some of the most basic constructions in arithmetic. By considering limits and colimits of the functors $W_{p, \a}(A)$ and $W_{p, \a}(A)[\frac{1}{p}]$ one can make functorial the definition of $A_K^+$ as bornological ring. In particular, for defining these limits and colimits one has to add a parameter for $\frac{1}{p}$ for extending the norm of $W_{p, \a}(A)$ to $W_{p, \a}(A)[\frac{1}{p}]$ in asymmetrical way. Functorially, this can be done by defining $W_{p, \a_1, \a_2}(A)[\frac{1}{p}]$ as the completion of the ``diagonal" of $W_{p, \a_1}(A)[\frac{1}{p}] \times W_{p, \a_2}(A)[\frac{1}{p}]$.

In the last example we show how the introduction of the functors $W_{p, \a}$ can be handy.

\begin{exa} 
For any $0 < \a_1, < \a_2 < 1$, there is an isomorphism of Banach rings
\[ \what{W_{p, \a_1, \a_2}(K)[\frac{1}{p}]} \cong A_{K, [-\log_p(\a_2), -\log_p(\a_1)]}. \]
Notice that this isomorphism is not so trivial as for the case of $A_{K, [-\log_p(\a_2), -\log_p(\a_1)]}^+$. $A_{K, [-\log_p(\a_2), -\log_p(\a_1)]}$ is defined to be the completion of $W_p(K^\circ)[\frac{1}{p}][\frac{1}{\omega}]$ for a non-zero $\omega \in K^\circ$ with $|\omega| < 1$. It is thus smaller than the usual ring of Witt vectors of $K$ and it is naturally contained in it. But in our construction, $W_{p, \a_1, \a_2}(K)[\frac{1}{p}]$ is defined by computing the contracting limit of Definition \ref{defn:ring_analytic_Witt_vectors}, therefore it contains only elements $(a_n) \in K^{\Z}$ for which the norm 
\[ |(a_n)|_{\a_1, \a_2} = \max \{ \sup_{n \in \N} |a_n| p^{-n \log_p \a_1}, \sup_{n \in \N} |a_{-n}| p^{n \log_p \a_2} \} \]
is bounded, which is precisely the condition that determines the elements of $A_{K, [-\log_p(\a_2), -\log_p(\a_1)]}$. 
\end{exa}

These definitions extends to ramified Witt vectors and perfectoid fields over $\F_q$. We will investigate these construction in more details in future works.

\appendix

\section{Nuclear bornological algebras} \label{appendix:S_T}

In this section we briefly recall some results which are well-known instances of complex functional analysis and complex geometry. The main idea is that for any Banach ring the algebra of analytic functions on open polydisks and on overconvergent closed polydisk are canonically endowed with a structure of nuclear (bornological in our context) vector space which is uniquely determined up to isomorphism. This fact ensures that apparently different presentations of those algebras are actually equivalent. In the main body of the present work one can see how the same phenomenon happens in the construction of the Fargues-Fontaine curve, more precisely in the definition of the ring of analytic functions on its affine covering (cf. section \ref{sec:recall}).

In this appendix $R$ will always denote a commutative Banach ring.

\begin{defn} \label{defn:S_algebra}
	For any real number $\rho > 0$ we denote the algebra of power-series
	\begin{equation} \label{eqn:S}
	 S_R(\rho) = \l \{ \sum_{i \in \N} a_i X^i | a_i \in R, \sum_{i \in \N} |a_i| \rho^i < \infty \r \}. 
    \end{equation}
\end{defn}

\begin{rmk} \label{rmk:S_algebra_compatible_symm}
Notice that this notation is compatible with the notation introduced so far (see Notation \ref{notation:box}) for the symmetric algebras of categories of monoids in the sense that
\[ S_R(\rho) = S_{\bBan_R^{\le 1}}([R]_\rho) \cong (S_{\bBan_{\F_1}^{\le 1}}(*_\rho)) \wotimes_{\F_1} R. \]
\end{rmk}

\begin{prop}
	$S_R(\rho)$ is a $R$-Banach algebra.
\end{prop}
\begin{proof}
	Proposition 6.1 of \cite{BaBe}.
\end{proof}

\begin{defn}
	For any real number $\rho > 0$ we denote the Banach $R$-module
	\[ T_R(\rho) = {\prod_{n \in \N}}^{\le 1} [R]_{\rho^n}. \]
\end{defn}

\begin{thm} \label{thm:ind_cofinal}
	For any $\rho > 0$ one has an isomorphism
	\[ \underset{r > \rho}{``\limind"} S_R(r) \cong \underset{r > \rho}{``\limind"} T_R(r) \]
	of objects of $\bInd(\bBan_R)$.
\end{thm}
\begin{proof}
	The proof is easily reduced to the (well-known) observation that if $\rho' < \rho$ and $\{ a_n \}_{n \in \N}$ is a sequence of elements $R$ such that $\sup_{n \in \N} |a_n| \rho^n < \infty$ then $\underset{n \in \N}{\sum} |a_n|(\rho')^n < \infty$.
\end{proof}

The same reasoning used to prove Theorem \ref{thm:ind_cofinal} proves also the following theorem.

\begin{thm} \label{thm:pro_cofinal}
	For any $\rho > 0$ one has an isomorphism
	\[ \limpro_{r < \rho} S_R(r) \cong \limpro_{r < \rho} T_R(r) \]
	of objects of $\bInd(\bBan_R)$.
\end{thm}
\begin{proof}
    Again, the cofinality of the two systems is a consequence of the same observation made in the proof of Theorem \ref{thm:ind_cofinal}.
\end{proof}

Theorems \ref{thm:ind_cofinal} and \ref{thm:pro_cofinal} have the following easy generalization.
Let $M \subset \R$ be any subset, then defines 
\[ S_{M, R}(\rho) = {\coprod_{m \in M}}^{\le 1} [R]_{\rho^m} \]
and
\[ T_{M, R}(\rho) = {\prod_{m \in M}}^{\le 1} [R]_{\rho^m}. \]

\begin{cor} \label{cor:ind_pro_cofinal}
	For any $\rho > 0$ one has an isomorphisms
	\[ \underset{r > \rho}{``\limind"} S_{M, R}(r) \cong \underset{r > \rho}{``\limind"} T_{M, R}(r),  \ \ \limpro_{r < \rho} S_{M, R}(r) \cong \limpro_{r < \rho} T_{M, R}(r) \]
	of objects of $\bInd(\bBan_R)$.
\end{cor}

\end{document}